\documentclass[12pt, leqno]{amsart}
\usepackage{latexsym}
\usepackage{graphicx}
\usepackage{placeins}
\usepackage{float}
\usepackage{subfigure}
\usepackage{color}
\usepackage{amsfonts,amsmath,amsthm,amssymb,amscd, hyperref}
\usepackage{mathrsfs}

\usepackage{color}

\newtheorem{Theorem}{Theorem}[section]
\newtheorem{Proposition}[Theorem]{Proposition}
\newtheorem{Lemma}[Theorem]{Lemma}
\newtheorem{Corollary}[Theorem]{Corollary}
\newtheorem{Claim}{Claim}[Theorem]

\theoremstyle{definition}
\newtheorem{Definition}[Theorem]{Definition}
\newtheorem{Remark}[Theorem]{Remark}
\newtheorem{Example}[Theorem]{Example}
\newtheorem{Conjecture}[Theorem]{Conjecture}
\newtheorem{Question}[Theorem]{Question}

\newtheorem{Assumption}[Theorem]{Assumption}

\renewcommand{\theTheorem}{\arabic{section}.\arabic{Theorem}}
\renewcommand{\theClaim}{\arabic{section}.\arabic{Theorem}.\arabic{Claim}}
\renewcommand{\theequation}{\arabic{section}.\arabic{equation}}

\makeatletter

\newcommand{\Proof}{{\sl Proof.}\quad}
\newcommand{\QED}{{\unskip\nobreak\hfil\penalty50\quad\null\nobreak\hfil
{$\Box$}\parfillskip0pt\finalhyphendemerits0\par\medskip}}

\renewcommand{\thesubsection}{\arabic{section}.\arabic{subsection}}
\renewcommand{\theTheorem}{\arabic{section}.\arabic{Theorem}}
\renewcommand{\theClaim}{\arabic{section}.\arabic{Claim}}
\renewcommand{\theequation}{\arabic{section}.\arabic{equation}}

\renewcommand{\a}{\alpha}
\renewcommand{\b}{\beta}

\renewcommand{\d}{\delta}
\newcommand{\D}{\Delta}

\newcommand{\ve}{\varepsilon}
\newcommand{\g}{\gamma}

\renewcommand{\k}{\kappa}

\newcommand{\vp}{\varphi}

\newcommand{\s}{\sigma}

\newcommand{\ga}{{\mathfrak{a}}}

\newcommand{\Bcal}{{\mathcal B}}
\newcommand{\Ccal}{{\mathcal C}}

\newcommand{\Ecal}{{\mathcal E}}

\newcommand{\Lcal}{{\mathcal L}}

\newcommand{\Ocal}{{\mathcal O}}

\newcommand{\bfH}{\mathbf{H}}
\newcommand{\bfJ}{\mathbf{J}}
\newcommand{\bfP}{\mathbf{P}}
\newcommand{\bfQ}{\mathbf{Q}}

\newcommand{\bfr}{\mathbf{r}}

\newcommand{\FF}{{\mathbb{F}}}
\newcommand{\ZZ}{{\mathbb{Z}}}
\newcommand{\QQ}{{\mathbb{Q}}}
\newcommand{\RR}{{\mathbb{R}}}
\newcommand{\CC}{{\mathbb{C}}}
\newcommand{\PP}{{\mathbb{P}}}

\newcommand{\OO}{{\mathcal{O}}}

\newcommand{\Spec}{\operatorname{Spec}}

\newcommand{\ord}{\operatorname{ord}}

\newcommand{\Pic}{\operatorname{Pic}}

\newcommand{\Gal}{\operatorname{Gal}}

\newcommand{\GL}{\operatorname{GL}}

\newcommand{\rest}[2]{\left.{#1}\right\vert_{{#2}}}

\newcommand{\norm}[1]{\Vert #1 \Vert}
\newcommand{\Aff}{\mathbb{A}}
\newcommand{\KK}{{\mathbb{K}}}
\newcommand{\Phibar}{\overline{\Phi}}

\newcommand{\Lcalbar}{\overline{\Lcal}}
\newcommand{\hhat}{{\hat{h}}}

\newcommand{\reduce}[1]{\widetilde{{#1}}}
\newcommand{\Res}{\operatorname{Res}}
\newcommand{\htilde}{\reduce{h}}
\DeclareMathOperator{\per}{Per}

\newenvironment{parts}[0]{%
  \begin{list}{}%
    {\setlength{\itemindent}{0pt}
     \setlength{\labelwidth}{1.5\parindent}
     \setlength{\labelsep}{.5\parindent}
     \setlength{\leftmargin}{2\parindent}
     \setlength{\itemsep}{0pt}
     }%
   }%
  {\end{list}}
\newcommand{\Part}[1]{\item[\upshape#1]}

\begin{document}

\title{Heights and periodic points for one-parameter families of H\'enon maps }
\author{Liang-Chung Hsia}
\address{Department of Mathematics, National
Taiwan Normal University, Taipei, Taiwan, ROC}
\email{hsia@math.ntnu.edu.tw}
\author{Shu Kawaguchi}
\address{Department of Mathematical Sciences, 
Doshisha University, Kyoto 610-0394, Japan}
\email{kawaguch@mail.doshisha.ac.jp}
\thanks{The first named author's research was supported by Grant MOST106-2115-M-003-014-MY2 and the second named author's research was supported by JSPS Grant 18H01114}
\date{}

\maketitle

\begin{abstract}
In this paper we study arithmetic properties of a one-parameter family $\bfH$ of H\'enon maps 
over the affine line. Given a family of initial points $\bfP$ satisfying a natural condition, 
we show the height function 
$h_{\bfP}$ associated to $\bfH$ and $\bfP$ is the restriction of the height function associated to 
a semipositive adelically metrized line bundle on projective line. We then show various local properties of 
$h_{\bfP}$. Next we consider the set $\Sigma(\bfP)$ consisting of periodic parameter values, 
and study when $\Sigma(\bfP)$ is an infinite set or not. We also study unlikely intersections 
of periodic parameter values. 
\end{abstract}

\section*{Introduction}
\label{sec:introduction}
\renewcommand{\theTheorem}{\Alph{Theorem}}

In this paper, we study arithmetic properties of families of 
H\'enon maps. 
A H\'enon map~$H$ over a field $K$ is an automorphism of the affine plane of the form 
\begin{equation}
\label{eqn:intro}
  H\colon \Aff^2 \to \Aff^2, \quad (x, y) \mapsto (\delta y + f(x), x), 
\end{equation}
where $\delta \in K \setminus \{0\}$ and $f(x) \in K[x]$ is a polynomial of degree $d \geq 2$. 
Up to conjugacy and a finite field extension, we may assume that $f(x)$ is a monic polynomial. 
H\'enon maps constitute an important class of affine plane automorphisms. Indeed, Friedland--Milnor \cite[Theorem~2.6]{FM} showed that any dynamically interesting affine plane automorphism (i.e. one with positive first dynamical degree) over $\CC$ is, up to conjugacy, compositions of H\'enon maps with monic $f(x)$. Properties 
of H\'enon maps over $\CC$ have been studied in e.g. \cite{bs, fs, ho} from among the vast literature. 
Arithmetic properties of H\'enon maps have also been studied in e.g. \cite{Denis, DF, ing, Ka06, Ka, Lee1, Lee, Marcello, SilHenon}. 

In this paper, we are interested in {\em families} of H\'enon maps. To be precise, we consider a one-parameter family $\bfH$ of H\'enon maps  over $\Aff^1 = \Spec(K[t])$, so that we replace $K$ in~\eqref{eqn:intro} with $K[t]$. Note that, to have the inverse $\bfH^{-1}$ over $\Aff^1$, $\delta \in K[t]\setminus \{0\}$ should belong to~$K \setminus \{0\}$, and our family is thus of the form 
\begin{equation}
\label{eqn:intro:2}
  \bfH \colon \Aff^2 \to \Aff^2, \quad (x, y) \mapsto (\delta y + f_t(x), x), 
\end{equation}
where $\delta \in K \setminus \{0\}$ and $f_t(x) \in (K[t])[x]$ is a monic polynomial in $x$ of degree $d \geq 2$ (with respect to $x$). We regard $\bfH$ as a family of H\'enon maps $H_t$ over $\overline{K}$, where $t$ runs through elements of an algebraic closure 
$\overline{K}$ of $K$. 

When $K$ is a field equipped with a set $M_K$ of inequivalent, non-trivial absolute values 
that satisfies the product formula, we have the canonical height function $\widetilde{h}_{H_t}$ for 
$H_t$ for each $t \in \overline{K}$, as we now explain. 
For each $v \in M_K$, we denote by $K_v$ the completion of $K$ with respect to $v$, and 
by $\KK_v$ the completion of an algebraic closure of $K_v$. 
Suppose that $v$ is archimedean, i.e., $\KK_v = \CC$. 
Then, as constructed and proven to be extremely useful in e.g. \cite{bs, fs, ho}, 
we have the (complex) Green function $G_{H_t, v}\colon 
\Aff^2(\CC) \to \RR_{\geq 0}$ defined by  
\begin{equation}
\label{eqn:intro:Green}
  G_{H_t, v}(P) = \max\left\{
  \lim_{n\to +\infty} \frac{1}{d^n} \log^+ \left| H_t^n(P)\right|, 
  \lim_{n\to +\infty} \frac{1}{d^n} \log^+ \left| H_t^{-n}(P)\right|
  \right\},
\end{equation}
where $P \in \Aff^2(\CC)$ and $\log^+(r) := \log \max \{1, r\}$ for any $r \in \RR$. 
When $v$ is nonarchimedean, we also have the $v$-adic Green function 
$G_{H_t, v}\colon \Aff^2(\KK_v) \to \RR_{\geq 0}$, and we have the specialized 
canonical height function $\widetilde{h}_{H_t}\colon \Aff^2(\overline{K}) \to \RR_{\geq 0}$ defined as 
the sum of these Green functions with some normalizing constants $n_v$ (see \S~\ref{sec:preliminaries} for details). 
Since $K(t)$ is the function field of $\PP^1$ over a field $K$, $K(t)$ is also naturally equipped with 
a set of absolute values that satisfies the product formula (see Example~\ref{eg:prod:ft}), and we have 
the function field canonical height 
$\widetilde{h}_{\bfH}\colon \Aff^2(\overline{K(t)}) \to \RR_{\geq 0}$ as well. 

Suppose we are given an initial point $\bfP \in \Aff^2(K[t])$, which we regard as 
a family of initial points $P_t := \bfP(t)$ with $t \in \overline{K}$. Then 
both the specialized canonical heights 
$\widetilde{h}_{H_t}(P_t) \in \RR_{\geq 0}$ for each $t \in \overline{K}$ and 
the function field canonical height $\widetilde{h}_{\bfH}(\bfP) \in \RR_{\geq 0}$ are defined.  
Assuming that $\widetilde{h}_{\bfH}(\bfP) \neq 0$, we put 
\begin{equation}
\label{eqn:intro:hP}
  h_{\bfP}(t) := \frac{\widetilde{h}_{H_t}(P_t)}{\widetilde{h}_{\bfH}(\bfP)}. 
\end{equation}

One of the main results of this paper is the following theorem on variation of canonical heights for H\'enon maps. 

\begin{Theorem}[see Theorem~\ref{thm:main2:revisited}]
\label{thm:main:A}
Let $K$ be a number field or the function field of an integral projective variety that is regular in codimension 
one over another field $F$. Let $\bfH$ and $\bfP$ be given as above such that 
$\widetilde{h}_{\bfH}(\bfP) \ne 0$.     
Then $h_{\bfP}$ is the restriction to $\Aff^1(\overline{K})$ of 
a height function on $\PP^1(\overline{K})$ associated to $\OO_{\PP^1}(1)$ with 
semipositive adelic metrics. In other words, there exists a semipositive 
adelically metrized line bundle $\Lcalbar_{\bfP} = (\Ocal_{\PP^1}(1),\{\norm{\cdot}_v\}_{v\in M_K})$ 
on the parameter space $\PP^1$ such that, if $h_{\Lcalbar_{\bfP}}$ denotes the height function associated to $\Lcalbar_{\bfP}$, then 
we have 
\[
  \widetilde{h}_{H_t}(P_t) = \widetilde{h}_{\bfH}(\bfP)\, h_{\Lcalbar_{\bfP}}(t) 
\]
for any $t \in \Aff^1(\overline{K})$. (For the definitions of a 
semipositive adelically metrized line bundle and the associated height function, see \S~\ref{subsec:function field henon}.)
\end{Theorem}

Theorem~\ref{thm:main:A} may best be compared with variation of canonical heights for families of 
elliptic curves.  Let $X$ (resp. $C$) be a smooth projective surface (resp. curve) over a number field $K$, 
and let $\pi\colon X \to C$ be an elliptic surface. For simplicity, we assume that 
$C = \PP^1$. Let $U$ be a Zariski open subset of $\PP^1$ such that $E_t:= \pi^{-1}(t)$ is 
an elliptic curve for any $t \in U(\overline{K})$. We denote by $\mathbf{E}$ the generic fiber of $\pi$, which is an elliptic curve over $K(\PP^1)$. Let $\bfP\colon \PP^1 \to X$ be a section. 
Then we have the N\'eron--Tate heights $\widehat{h}_{\mathbf{E}}(\bfP)$ and $\widehat{h}_{E_t}(P_t)$ for 
$t \in U(\overline{K})$.  Silverman \cite{silverman} and Tate \cite{tate} showed, among other things,  that 
\begin{equation}
\label{eqn:intro:variation:ell}
  \widehat{h}_{E_t}(P_t) = \widehat{h}_{\mathbf{E}}(\bfP) h(t) + O(1) \quad \text{for any $t \in U(\overline{K})$},
\end{equation}
where $h(t)$ is the (usual) Weil height function on $\PP^1$. For the dynamical setting, Ingram \cite{ing13} 
showed a similar equality for a family of one-variable polynomial maps. In a subsequent paper, 
Ingram~\cite{ing} considered a family of H\'enon maps and obtains similar results  (see 
Theorem~\ref{thm:Ingram:variation}). 

For applications to unlikely intersection problems and further generalizations (see 
\S~\ref{subsec:intro:periodic:para}-\ref{subsec:intro:unlikely} below), it has been noticed that it is important that there is a ``nice'' height function without the $O(1)$ term in \eqref{eqn:intro:variation:ell}  or in its dynamical counterparts, just as it is important that N\'eron-Tate heights and canonical heights in dynamical systems are defined without the $O(1)$ term. In the setting of \eqref{eqn:intro:variation:ell}, assuming that $\widehat{h}_{\mathbf{E}}(\bfP) \neq 0$ and putting $h_{\bfP}(t) :=  \widehat{h}_{E_t}(P_t)/\widehat{h}_{\mathbf{E}}(\bfP)$ (for $t \in U(\overline{K})$), the question is whether or not $h_{\bfP}$ is the restriction to $U(\overline{K})$ of a ``nice'' height function on $\PP^1(\overline{K})$. This question has been 
answered in the affirmative in DeMarco--Wang--Ye \cite[Theorem~1.5]{DWY} and DeMarco--Mavraki \cite{DM}, which gives an alternate proof 
of Masser--Zannier's unlikely intersections of simultaneous torsion sections of elliptic curves 
\cite{MZ1, MZ2, MZ3} for the Legendre family. 
For various families of one-variable polynomial maps, this question has also been answered in the 
affirmative (see \S~\ref{subsec:intro:periodic:para}-\ref{subsec:intro:unlikely} below). However, this question 
is still widely open for families of rational self-maps of  higher dimensional varieties. 
Theorem~\ref{thm:main:A} gives a refinement of Ingram's result 
($=$ Theorem~\ref{thm:Ingram:variation} 
with $\widetilde{h}$ in place of $\widehat{h}$),
and answers this question in the affirmative for a family of H\'enon maps. 

\medskip
In the following, we explain some properties and applications of $h_{\bfP}$ in \eqref{eqn:intro:hP}. 

\subsection{Local (i.e., $v$-adic) properties of $h_{\bfP}$}
\label{subsec:intro:local}
Let $K$ be a field equipped with a set $M_K$ of inequivalent, non-trivial absolute values 
that satisfies the product formula. 
Then, for each $v \in M_K$,  we have the $v$-adic Green function 
$G_{H_t, v}\colon \Aff^2(\KK_v) \to \RR_{\geq 0}$ as explained above (cf. \eqref{eqn:intro:Green}). 
We set 
\[
  G_{\bfP, v}(t) := G_{H_t, v}(P_t)
\] 
for any $t \in \Aff^1(\KK_v)$.  Let $|\cdot|_v$ denote the $v$-adic absolute value, and 
we write $\Vert \mathbf{a} \Vert_v := \max\{|a_1|_v, \ldots, |a_k|_v\}$ 
for $\mathbf{a} = (a_1, \ldots, a_k) \in \KK_v^k.$ We define
\begin{align*}
W_{\bfP, v} & := \left\{t \in \Aff^1(\KK_v)
\;\left|\; \lim_{n\to+\infty} \Vert (H_t^n(P_t), H_t^{-n}(P_t)) \Vert_v = +\infty \right.\right\}, \\
K_{\bfP, v} & := \left\{t \in \Aff^1(\KK_v)
\;\left|\; \text{the sequence $\{H_t^n(P_t)\}_{n \in \ZZ}$ is bounded}
\right. \right\}. 
\end{align*}

Using (the proof of) Theorem~\ref{thm:main:A}, we obtain the following.  

\begin{Theorem}[see Proposition~\ref{prop:decomposition}, Theorem~\ref{thm:K:P:v:characterization}, Corollary~\ref{cor:thm:main:Mn}, Remark~\ref{rmk:cor:thm:main:Mn}]
\label{thm:intro:local}
\begin{parts}
\Part{(1)}
Set $c := \widetilde{h}_{\bfH}(\bfP)$. Then $c$ is a rational number. The function $h_{\bfP}$ decomposes into the sum of $v$-adic functions $G_{\bfP, v}$ with coefficients $n_v/c$ (see~\eqref{eqn:can:height:tilde:2} and \eqref{eqn:can:height:tilde:3}).
Further, 
for any $v \in M_K$, 
\[
  G_{\bfP, v}(t) - c \log |t|_v 
\]
converges as $t \to \infty$. 
\Part{(2)}
We have $\Aff^1(\KK_v) = W_{\bfP, v} \,\amalg\, K_{\bfP, v}$, and 
$K_{\bfP, v}$ is exactly the set of points where $G_{\bfP, v}$ vanish:
\[
  K_{\bfP, v} = \{t \in \Aff^1(\KK_v) \mid G_{\bfP, v}(t) =0 \}. 
\]
\end{parts}
\end{Theorem}

We note that, in the case of one-parameter families of elliptic curves, the continuity at the 
singular fibers was shown by Silverman \cite{Sil1, Sil2, Sil3}, which provided a basis for \cite{DM}. 

\begin{Remark}
The question of rationality of canonical heights associated to dynamical systems over function 
fields has been investigated in~\cite{DG} for one-variable dynamics.  In the case of H\'enon maps, the 
canonical height $\widetilde{h}_{\bfH}(\bfP)$ appearing in Theorem~\ref{thm:intro:local} provides an instance 
of rational canonical height for higher dimensional dynamics. It would be interesting to study whether or not the 
property of rationality of canonical height  holds for general H\'enon maps $\bfH$ and points $\bfP$ over a 
function field. 
\end{Remark}

\medskip
As an illustration of $K_{\bfP, v}$, Figure~1 below, drawn with Qfract (by Hiroyuki Inou), concerns the family of quadratic H\'enon maps $\bfH = (H_t)_{t \in \overline{\QQ}}\colon (x, y) \mapsto (y +x^2 + t, x)$ and the (constant) family of initial points $\bfP = (0, 0)$ centered around $t=0$. Here $K = \QQ$ and 
$\KK_v = \CC$. 
Shades depend on how fast the orbit $\{H^n_t((a, b))\}$ escapes 
as $|n|$ becomes large (the escaping rate is very small in the black region). 
Note that, at the center $t = 0$ of the images, the point $(0, 0)$ is fixed with respect to $H_0$. Note also that, 
for this example, there are infinitely many parameter values $t$ such that 
$(0, 0)$ is periodic  respect to $H_t$ (see Theorem~\ref{thm:intro:infiniteness} below). 

\begin{figure}[htb!]
  \includegraphics[width=7.0cm, bb=0 0 640 640]{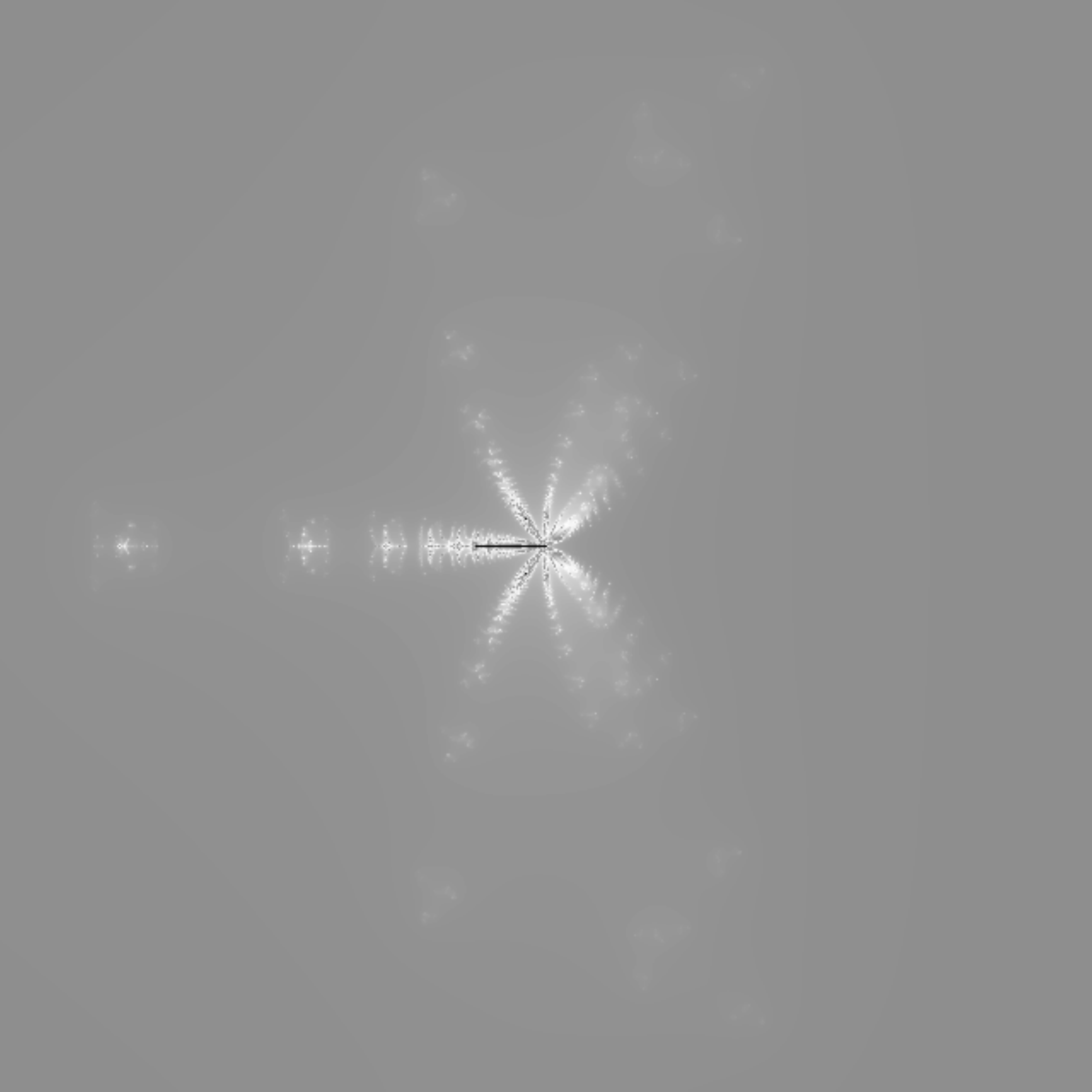}\hspace{2eM}
  \includegraphics[width=7.0cm, bb=0 0 640 640]{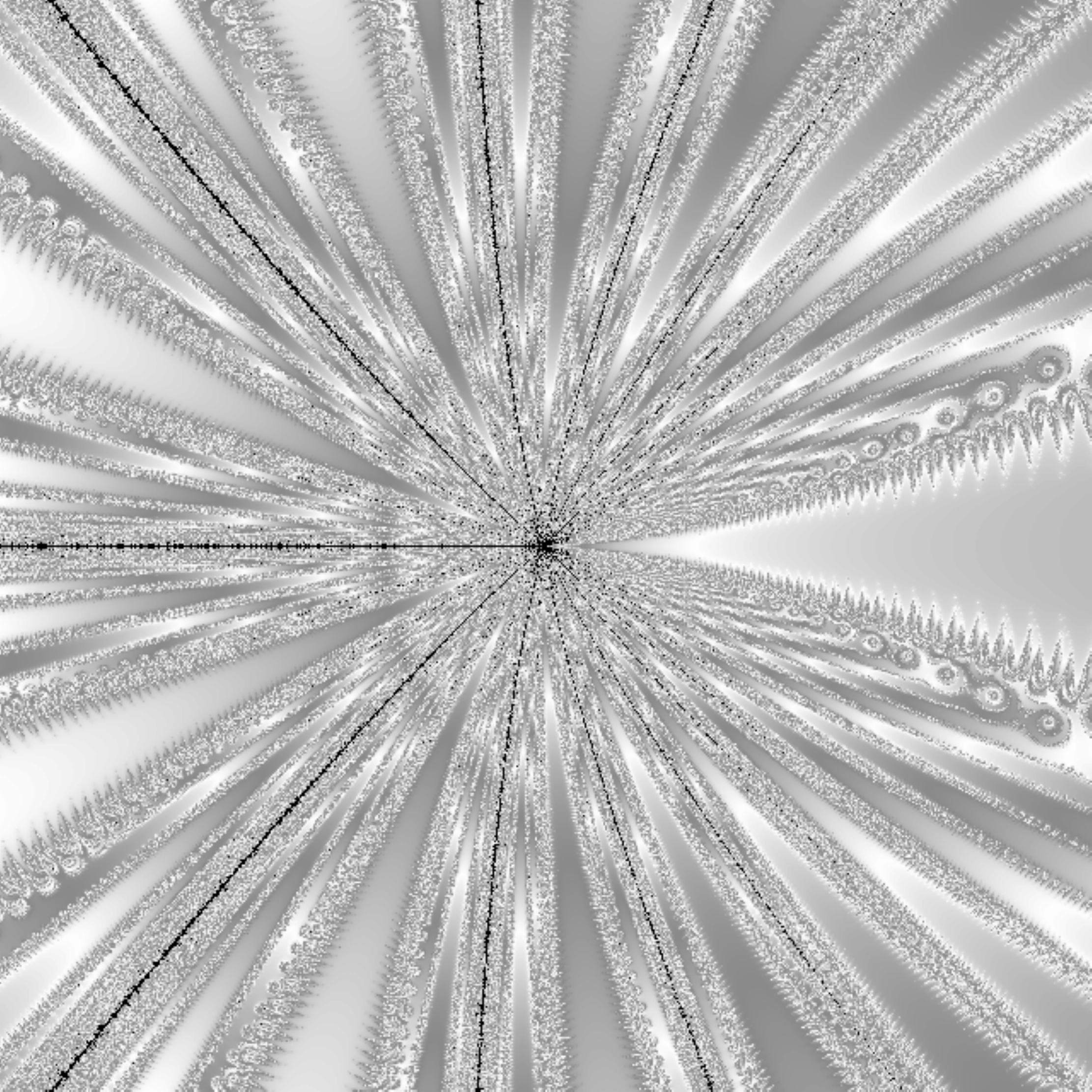}
  \caption{The region on the left is $\{t \in \CC \mid |{\rm Re}(t)| \leq 1, \, 
   |{\rm Im}(t)| \leq 1\}$, and that on the right is $\{t \in \CC \mid |{\rm Re}(t)| \leq 0.01, \, 
   |{\rm Im}(t)| \leq 0.01\}$.}
\end{figure}

\subsection{The set of periodic parameter values}
\label{subsec:intro:periodic:para}
To study some global properties of $h_{\bfP}$, we introduce the 
set $\Sigma(\bfP)$ of periodic parameter values: 
\[
  \Sigma(\bfP) := \{t \in \Aff^1(\overline{K}) \mid \text{$P_t$ is periodic with respect to $H_t$}\}. 
\]
We remark that, if $K$ is a number field, then 
the Northcott property implies that 
$\Sigma(\bfP)$ is exactly the set of points where $h_{\bfP}$ vanish (see 
Proposition~\ref{prop:periodic parameter}):
\[
  \Sigma(\bfP) = \{t \in \Aff^1(\overline{K}) \mid h_{\bfP}(t) =0 \}. 
\]

Now let $\bfQ\in \Aff^2(K[t])$ be another initial point with $\widetilde{h}_{\bfH}(\bfQ) \neq 0$. We 
consider a question  
if there exist infinitely many $t\in \overline{K}$ such that both $P_t$ and $Q_t$ are simultaneously periodic with respect to $H_t$, i.e., {whether or not $\Sigma(\bfP) \cap \Sigma(\bfQ)$ is an infinite set. 

This question is about the study of unlikely intersections in arithmetic dynamics, which was initiated by 
Baker--DeMarco \cite{Matt-Laura} with motivation from Masser--Zannier's study of unlikely intersections 
of simultaneous torsion sections of elliptic curves \cite{MZ1, MZ2, MZ3} and  the Pink--Zilber conjecture in arithmetic geometry (see \cite{Zannier}). Unlikely intersection in arithmetic dynamics has since been further and beautifully explored in e.g. \cite{BD, demarco, DWY, ght1, ght2, ght3, fg}. 
 
For the present family of H\'enon maps, however, we encounter a phenomenon that is not observed in families of one-variable maps, in which the set corresponding $\Sigma(\bfP)$ is an 
infinite set. (This follows from Montel's theorem on normal families in one-variable complex analysis.)
In our case, as we now explain, $\Sigma(\bfP)$ can be either an infinite set or a finite set.  

First we explain our result toward infiniteness of $\Sigma(\bfP)$. Inspired by the paper Dujardin--Favre \cite{DF} on dynamical Manin--Mumford conjecture for affine plane polynomial automorphisms, we suppose that our family $\bfH$ has an involution $\iota\colon \Aff^2 \to \Aff^2$ over $K$ such that 
$\iota \circ \bfH \circ \iota =  \bfH^{-1}$. 
Taking the Jacobian, we get $\delta \in \{1, -1\}$. It is straightforward to check that the involution
$\iota_{\d}\colon (x, y) \mapsto (-\d y, -\d x) $ reverses $\bfH$ provided that $f_t (-\d x) = f_t(x).$ Note that this condition gives non-trivial restriction on $f_{t}(x)$ only when $\d = 1.$  
Let the line 
\[
   C_\delta\colon \, \delta x +  y = 0
\]
in $\Aff^2$ denote the set of $\iota_{\d}$-fixed points for $\delta = 1, -1$. Then we have the following result. 

\begin{Theorem}[$=$ Theorem~\ref{thm:infiniteness:Sigma:P}]
\label{thm:intro:infiniteness}
Let $K$ be a field (of any characteristic). 
Let $\bfH$ be the family of H\'enon maps in \eqref{eqn:intro:2} such that 
$\delta \in \{1, -1\}$. When $\delta = 1$, we assume that $f_t(x)$ is an 
even polynomial in $x$. Let $\bfP = (a(t), b(t)) \in \Aff^2(K[x])$. 
If $\bfP$ lies in $C_\delta$, i.e., $\delta a(t) + b(t) = 0$ in $K[t]$, 
then $\Sigma(\bfP)$ is an infinite set. 
\end{Theorem}

We remark that, for 
the quadratic H\'enon maps we consider in Figure~1, Theorem~\ref{thm:intro:infiniteness} implies that $\Sigma((0, 0))$ is an infinite set. 
It is natural to ask if $\Sigma((a, b))$ is a finite set when $a + b \neq 0$. 
In this opposite direction, we obtain, among other things, the following 
result on {\em emptiness} of $\Sigma(\bfP)$. 

\begin{Proposition}[$=$ Proposition~\ref{prop:resultant:is:monic}]
\label{prop:intro:finiteness}
Let $\bfH\colon (x, y) \mapsto (y + x + t^2, x)$ be the family of quadratic H\'enon maps  over $\QQ$ (or any field of characteristic zero). 
Then, for any 
$b \notin \overline{\ZZ}$ we have $\Sigma((0, b)) = \emptyset$ where 
$\overline{\ZZ}$ denotes the ring of algebraic integers.
\end{Proposition}

\medskip
We say that a H\'enon map $H$ is {\em reversible} if there exist an involution $\sigma$ of affine plane 
and a positive integer $m$ such that $\sigma \circ H^m \circ \sigma = H^{-m}$. 
Based on Theorem~\ref{thm:intro:infiniteness}, the dynamical Manin-Mumford conjecture 
for affine plane automorphisms in \cite{DF}, we ask the following. 

\begin{Question}
\label{q:henon dynamical manin mumford}
Let $K$ be an algebraically closed field of characteristic zero.  
Let $\bfH\colon \Aff^{2} \to \Aff^{2}$ be a H\'enon map over $\Aff^1 = \Spec(K[t])$, 
which we regard as a family $(H_t)_{t \in K}$ parametrized by $t$. 
Put $\per_{\Aff^{2} \times \Aff^1}(\bfH) := \{(P, t) \in \Aff^{2}(K) \times \Aff^1(K) \mid \text{$P$ is periodic with respect to $H_t$}\}$. Let $\Ccal$ be an integral curve in $\Aff^{2} \times \Aff^1$, and we 
assume that $\Ccal(K)\cap \per_{\Aff^{2} \times \Aff^1}(\bfH)$ is infinite. 
Is it true that one of the following conditions must hold?
\begin{enumerate}
\item[(i)]   $\Ccal \subseteq \Aff^{2}_{t}=\Aff^{2}\times \{t\}$ for some $t\in \Aff^1(K)$, 
and there exists an involution $\sigma_t$ of $\Aff^{2}$ over $K$ such that 
$\sigma_t \circ H_t^{m} \circ \sigma_t = H_t^{-m}$ for some $m\ge 1$.  
Furthermore, $\Ccal$ is contained in the set of fixed points of $\s_{t}$; 
\item[(ii)]   There exists a positive integer $m$ such that 
$\Ccal$ is an irreducible component of 
$\{(P, t) \in \Aff^{2}(K) \times \Aff^1(K) \mid H_t^m(P) = P\}$; 
\item[(iii)]  There exists an involution ${\boldsymbol \sigma}$ of $\Aff^{2}$ over $K[t]$ such that ${\boldsymbol \sigma} \circ \bfH^{m} \circ {\boldsymbol \sigma} = \bfH^{-m}$ for some $m\ge 1$ and $\Ccal$ is contained in 
the set of fixed points of ${\boldsymbol \sigma}$. 
\end{enumerate}
\end{Question}

\subsection{Unlikely intersections of 
periodic parameter values}
\label{subsec:intro:unlikely}
By Theorem~\ref{thm:intro:infiniteness}, we have instances that $\Sigma(\bfP)$ is infinite. 
Let $\bfP, \bfQ \in \Aff^2(K[t])$ be two initial points with $\widetilde{h}_{\bfH}(\bfP) \neq 0$ and 
$\widetilde{h}_{\bfH}(\bfQ) \neq 0$ such that 
both $\Sigma(\bfP)$ and $\Sigma(\bfQ)$ are infinite. Now we consider the problem of unlikely intersections, i.e., 
what happens if $\Sigma(\bfP) \cap \Sigma(\bfQ)$ is infinite. 

Theorem~\ref{thm:main:A} precisely allows us to use an arithmetic equidistribution theorem 
(see \cite{autissier, Baker-Rumely06, CL, fr04, fr06, Th} for curves, and \cite{Yuan} for varieties of arbitrary dimension). We obtain the following result. 

\begin{Theorem}[$=$ Theorem~\ref{thm:unlikely intersection thm}]
\label{thm:intro:unlikely:intersection}
Let $K$ be a number field, 
and let $\bfP, \bfQ \in \Aff^2(K[t])$ be as above. Then the following is equivalent. 
\begin{enumerate}
\item[(i)]
$\Sigma(\bfP) \cap \Sigma(\bfQ)$ is infinite. 
\item[(ii)]
$\Sigma(\bfP) = \Sigma(\bfQ)$. 
\item[(iii)]
$G_{\bfP, v} = G_{\bfQ, v}$ for any $v \in M_K$. 
\item[(iv)]
$h_{\bfP} = h_{\bfQ}$. 
\end{enumerate}
\end{Theorem}

Note that in the case of a family of one-variable polynomial dynamics parametrized by points of the affine line, it is proved that, for two given polynomial families of initial points to share infinite periodic parameter values, they must satisfy some orbital relations. The proof relies on the B\"ottcher coordinate in one-variable polynomial dynamics. Although the B\"ottcher coordinates $\varphi^+, \varphi^-$ (each for one direction defined in a different region) also exist for H\'enon maps (see \cite{ho}), the techniques in the one-variable case seem not easily extended to the present situation. We would like to pose the following question (see also Proposition~\ref{prop:converse:q:intro:1}). 

\begin{Question}
\label{q:intro:1}
\begin{parts}
\Part{(1)}
Let $H\colon \Aff^2 \to \Aff^2$ be a H\'enon map over $\CC$, and let $P, Q \in \Aff^2(\CC)$. 
Suppose that $\{H^n(P)\}_{n \in \ZZ}$ and $\{H^n(Q)\}_{n \in \ZZ}$ are both unbounded, 
and that $G(H^n(P)) = G(H^n(Q))$ holds for any sufficiently large $n$, 
where $G$ is the complex Green function defined in \eqref{eqn:intro:Green}. Is it true that 
there exists an invertible affine map $\sigma\colon \Aff^2 \to \Aff^2$ over $\CC$ such that 
$\sigma^{-1} \circ H^m \circ \sigma = H^{m}$ or $\sigma^{-1} \circ H^m \circ \sigma = H^{-m}$ for some $m \geq 1$ and $P = \sigma(Q)$?
\Part{(2)}
Let $K$ be a number field. 
Let $\bfP, \bfQ \in \Aff^2(K[t])$ be two initial points with $\widetilde{h}_{\bfH}(\bfP) \neq 0$ and 
$\widetilde{h}_{\bfH}(\bfQ) \neq 0$. Suppose that $h_{\bfP} = h_{\bfQ}$.
Is it true that there exists an invertible affine map ${\boldsymbol \sigma}\colon \Aff^2 \to \Aff^2$ over $\overline{K}[t]$ with ${\boldsymbol \sigma}^{-1} \circ \bfH^m \circ {\boldsymbol \sigma} = \bfH^{m}$ or ${\boldsymbol \sigma}^{-1} \circ \bfH^m \circ {\boldsymbol \sigma} = \bfH^{-m}$ for some $m \geq 1$ and $\bfP = \bfH^n({\boldsymbol \sigma}(\bfQ))$ for some $n \in \ZZ$?
\end{parts}
\end{Question}

In this paper, we aim to explore some fundamental arithmetic properties of families of H\'enon maps. 
We believe that the methods and techniques used in this paper will be useful in studying further arithmetic properties, for example, in the case where the base parameter space is a punctured disk (in a suitable setting) or a higher dimensional space, or families of higher dimensional H\'enon maps (i.e., affine space regular polynomial automorphisms). 
See also Remark~\ref{rem:primitive divisors} for a question about primitive prime divisors in a family of H\'enon maps.

\medskip
We briefly describe the organization of this paper. 
In Section~\ref{sec:preliminaries} we set up the notation and recall some facts from the theory of canonical heights for H\'enon maps that will be used in the paper.
Section~\ref{sec:valued fields} and Section~\ref{sect:heights and metrized line bundles} are devoted to the
construction of the height function $h_{\bfP}$ and the proof of Theorem~\ref{thm:main:A}.  
In Section~\ref{sec:valued fields}, we consider $v$-adic settings. As in \cite{ght2}, we divide the parameter space $\Aff^1(\KK_v)$ into a bounded region and an unbounded region, and using various properties of H\'enon maps,  we show that certain $v$-adic metrics converge uniformly 
in each region. In Section~\ref{sect:heights and metrized line bundles}, we prove Theorem~\ref{thm:main:A}. In Section~\ref{sec:local:parameters:height}, we study $v$-adic properties of $h_{\bfP}$ and prove Theorem~\ref{thm:intro:local}. In Section~\ref{sec:set:periodic:parameters}, 
we prove Theorem~\ref{thm:intro:infiniteness}, by estimating multiplicities 
of some periodic points. Then in Section~\ref{sec:unlikely intersection} we prove Theorem~\ref{thm:intro:unlikely:intersection} on unlikely intersections of periodic parameter values. Here we use Yuan's equidistribution theorem \cite{Yuan}, in hoping that our methods will be useful when one considers a family of H\'enon maps over a higher-dimensional parameter space. 
In Section~\ref{sec:finiteness:per:param}, we show some results towards emptiness of $\Sigma(\bfP)$, 
including Proposition~\ref{prop:intro:finiteness}. 

\bigskip
\noindent{\sl Acknowledgement.}\quad
We presented an earlier content of this paper at 
the BIRS Workshop on Arithmetic and Complex Dynamics in November, 2017. 
We thank the organizers of the workshop, 
and Eric Bedford, Patrick Ingram and Thomas Gautier for helpful comments at that occasion. 
We also thank Charles Favre for helpful discussions,  
and to Hiroyuki Inou for figures in Introduction and Example~\ref{eg:figures:inou} drawn by {\em Qfract}. 
We thank Dragos Ghioca and Joe Silverman for helpful comments. 
The present paper grew out from our discussions at  the Academia Scinica and 
Kyoto University, and we thank 
Julie Tzu-Yueh Wang and Kazuhiko Yamaki for warm hospitality. 
Part of this work was done during the second named author's stay at the Mathematical Institute of Oxford, and he thanks the institute and Damian R\"ossler for warm hospitality. 

\setcounter{equation}{0}
\section{Preliminaries}
\label{sec:preliminaries}
\renewcommand{\theTheorem}{\arabic{section}.\arabic{Theorem}}

In this section, we recall some basic properties of H\'{e}non maps that will be used in this paper.

\subsection{Notation and terminology}
We list some of the notation and terminology that we use throughout this paper. 
\begin{itemize}
\item
For a field $K$, we denote by $\overline{K}$ an algebraic closure of $K$. 
\item
If $K$ is a number field, we denote by $O_K$ the ring of integers of $K$. 
\item 
If $K$ is equipped with a set $M_{K}$ of inequivalent, non-trivial absolute values (places), we denote by 
$M_K^{\rm fin}$ the subset of $M_K$ of all nonarchimedean absolute values, 
and by $M_K^\infty$ the subset of $M_K$ of all archimedean absolute values. 
\item 
For $v \in M_K$, the corresponding absolute value is denoted by $|\cdot |_v.$ 
We let $K_v$ be the completion of $K$ with respect to $|\cdot|_v$, 
and let $\KK_v$ be the completion of an algebraic closure of $K_{v}$.  
\item 
The (unique) extension of the absolute value $|\cdot |_v$ of $K_{v}$ to $\KK_{v}$ will still be denoted by $|\cdot|_{v}$. 
\item
Let $|\cdot|$ be an absolute value of a field $K$. For $(x_1, \ldots, x_n) \in K^n$, we write  
\[
\Vert (x_1, \ldots, x_n) \Vert 
:= \max\left\{|x_1|, \ldots, |x_n|\right\}.
\] 
\item
For $r \in \RR$, we set $\log^+ r := \log \max\{r, 1\}$. 
\item
Let $f\colon S \to S$ be a bijective self-map of a set $S$. We denote by $f^n$ the $n$-th iterate of 
$f$ under compositions of maps for $n \in \ZZ$. A point $P \in S$ 
is {\em periodic} for $f$ if there there exists an integer $n \geq 1$ such that
$f^n(P) = P$.  The least positive integer with $f^n(P) = P$ is 
called the {\em period} of $P$.  
\end{itemize}

\subsection{H\'{e}non maps}
\label{subsec:henon:maps}
Let $K$ be a field. A {\em H\'{e}non map} 
 is a polynomial automorphism of the affine plane defined over $K$ of the form
\begin{equation}
\label{eqn:Henon:map}
H \colon \Aff^2 \to \Aff^2 \quad  (x,y) \mapsto  \left( \d y + f(x), x\right), 
\end{equation}
where $ f(x) \in K[x]$ is a polynomial of degree $d\ge 2$ and $\d \in K\setminus\{0\}.$  The inverse
$H^{-1}$ is given by
\[
H^{-1}\colon
\Aff^2 \to \Aff^2, \quad
(x, y)
\mapsto
\left(y, \frac{1}{\d} \left(x - f(y)\right)\right).
\]
By taking a suitable finite extension $L$ of $K$ and conjugating by $(x, y) \mapsto (\varepsilon x, \varepsilon y)$ for suitable $\varepsilon \in L$, $f(x)$ becomes a monic polynomial. In this paper, we consider 
a H\'{e}non map such that $f(x)$ is monic.

It is also common to consider a H\'enon map $H^\prime$ of the form
\[
H^\prime\colon
\Aff^2 \to \Aff^2, \quad
(x, y) \mapsto
\left(y , \d x + f(y)\right),
\]
in place of $H\colon (x, y) \mapsto \left(\d y + f(x), x\right)$.
We remark that there is no real difference in choosing $H$ or $H^\prime$ concerning their dynamical properties.
Indeed, let $\jmath\colon \Aff^2 \to \Aff^2$ be the involution given by $(x, y) \mapsto (y, x)$. Then we have
\[
 \jmath \circ H \circ \jmath \, (x, y) = (y,\d x + f(y)) = H^\prime(x, y),
\]
so that $H^{\prime\, n} = \jmath \circ H^n \circ \jmath$ for $n \in \ZZ$. 
Thus dynamical properties of $H^\prime$ are deduced from those of $H$, and {\em vice versa}.

\subsection{Canonical heights for H\'{e}non maps}
\label{subsec: henon canonical height}

We  say a  field $K$ is a {\em product formula field}
if it  is equipped with a set $M_K$ of inequivalent absolute values (places) 
with the following property: 
The set $\{|\cdot|_v \in M_K \mid |\a|_v \neq 1\}$ is finite for any $\a \in K\setminus\{0\}$;   
Further, for each $v\in M_K$ there exists a positive integer $n_v$ such that  the product formula
$\prod_{v \in M_K}\,|\a|_v^{n_v} = 1$ holds for any $\a \in K\setminus\{0\}.$ 
If $K$ is a product formula field, and $L$ is a finite extension field of $K$, 
then $L$ is naturally equipped with a set $M_L$ so that $L$ is also a product formula field (see 
\cite[Proposition~1.4.2]{BG}). 

While some statements in this paper hold for any product formula field, 
we will assume that a product formula field is one of the following two types of fields (see for example, \cite[\S~2.3]{dio-geo} and \cite[\S~1.4.6]{BG} for more details).

\begin{Example}[number field]
\label{eg:prod:nf}
If $K$ is a number field, then we take $M_K$ as the set of all inequivalent absolute 
values of $K$. If $v \in M_K^{\rm fin}$ lies over a prime $p \in \QQ$, 
then $|p|_v = 1/p$. If $v \in M_K^\infty$, then $|\cdot|_v$ is the extension of 
the usual absolute value of $\QQ$.  For $v \in M_K^{\rm fin}$, 
the normalizing constant 
$n_{v}$ is defined to be $[K_{v} : \QQ_{p}]/[K:\QQ]$ if $v$ lies over $p$.
For $v \in M_K^\infty$, the normalizing constant $n_{v}$ is defined to be $1/[K:\QQ]$ 
if $v$ is a real place, and $2/[K:\QQ]$ if $v$ is a complex place. Then $K$ is a product formula field. 
\end{Example}

\begin{Example}[function field]
\label{eg:prod:ft}
Let $K = F(B)$ be the function field of an integral projective variety $B$ 
over a field $F$ such that 
$B$ is regular in codimension one. In this case, we take $M_K$ as the set of all 
prime divisors of $B$. The valuation $v_Y$ associated to a prime divisor $Y$ is defined to be 
the vanishing order along $Y$, and the corresponding absolute value is defined by~${|\cdot|_{v}} :=  \exp (-\ord_{v}(\cdot))$.  In particular, $M_K = M_K^{\rm fin}$. 
We fix an ample class $c \in \Pic(B)$. 
The normalizing constant $n_{v} $ is defined to be $\deg_c(Y)$. 
Then $K$ is a product formula field. 

When we regard the rational function field $F(t)$ over a field $F$ 
as a product formula field, we always regard $F(t)$ as the function field of $F(\PP^1)$, 
and we take an ample class $c \in \Pic(\PP^1)$ as a divisor class of degree $1$ above. 
\end{Example}

\medskip
Let $K$ be a product formula field and let $v\in M_K$. Let $H \colon \Aff^2 \to \Aff^2$ be a H\'{e}non map of degree $d\ge 2$ defined over $\KK_v$ 
as in~\eqref{eqn:Henon:map}. Let $P \in \Aff^2(\KK_v)$. The {\em Green functions} (also called {\em local 
canonical height functions}) for $H$ are defined by 
\begin{equation}
\label{eqn:Green:function}
G_v^+ (P) := \lim_{n\to \infty}\,\frac{1}{d^n} \log^+ \norm{H^n(P)}_v 
\quad \text{and}\quad
G_v^-(P) := \lim_{n\to \infty}\,\frac{1}{d^n} \log^+ \norm{H^{-n}(P)}_v. 
\end{equation}
The limits exist for any $P\in \Aff^2(\KK_v)$, see, for example \cite[Proposition~5.5]{ho} and \cite[\S3]{bs}  
when $v$ is archimedean, 
and \cite[Theorem~A]{Ka} when $v$ is nonarchimedean. 

Let $h \colon \Aff^2(\overline{K}) \to \RR$ be the (usual) Weil height function on $\Aff^2$ 
defined by 
\begin{equation}
\label{eqn:standard:height:function}
  h(P) = \frac{1}{[K(P):K]}\sum_{\sigma:K(P)\to\overline{K}}\sum_{v\in M_K}\, n_{v} 
  \log^+\Vert P^\sigma\Vert_v
\end{equation}
for $P \in \Aff^2(\overline{K})$. 
The {\em height functions}  $\hhat_{H}^+$, $\hhat_H^-$ and $\hhat_{H}$ for $H$ are defined by 
\begin{align}
\label{eqn:can:height:pm}
& \hhat_H^{+}(P)  = \lim_{n\to \infty}\frac{1}{d^n} h(H^{n}(P)), 
\quad 
\hhat_H^{-}(P)  = \lim_{n\to \infty}\frac{1}{d^n} h(H^{-n}(P)), \\
\notag 
& \hhat_{H}(P) = \hhat_{H}^+(P) + \hhat_H^{-}(P) 
\end{align}
for  $P \in \Aff^2(\overline{K})$. It follows from \eqref{eqn:Green:function} 
and \eqref{eqn:standard:height:function} that 
\begin{align} 
\hhat_H^{\pm}(P)
 & = \frac{1}{[K(P):K]}\sum_{\sigma:K(P)\to \overline{K}}\sum_{v\in M_K}\,n_v G_{v}^{\pm}\left(P^\sigma\right)
\\
\notag
& =  \sum_{v\in M_K} n_v \left(\frac{1}{[K(P):K]}\sum_{\sigma:K(P)\to \KK_{v}}\,G_{v}^{\pm}\left(P^\sigma\right)\right). 
\end{align}

We consider a variant of  $\hhat_H^{\pm}$. Set  
\begin{equation}
\label{eqn:def:G:intro}
G_v(P) = \max\{
G_v^+(P), G_v^-(P)
\}
\end{equation}
for $P \in \Aff^1(\KK_v)$. 
As in \cite[Equation (6-5)]{Ka}, 
we define 
\begin{align}
\label{eqn:can:height:tilde}
\htilde_{H}(P) & 
= \frac{1}{[K(P):K]}\sum_{\sigma:K(P)\to\overline{K}}\sum_{v\in M_K}\,
n_v G_v(P^\sigma)  \\
\notag
& = \sum_{v\in M_K} n_v \left(\frac{1}{[K(P):K]}\sum_{\sigma:K(P)\to\KK_{v}}\,G_v(P^\sigma)\right)
\end{align}
for $P \in \Aff^2(\overline{K})$. 

In this paper, we mainly study $\htilde_{H}$ rather than $\hhat_H$. For convenience, 
we call $\htilde_H$ the {\em canonical height} associated to the H\'enon map $H$. 
(Note the difference of terminologies used in \cite{Ka06}, where the``canonical height'' refers 
to $\hhat_H$.) 
As we see in Proposition~\ref{prop:northcott}~(1)~(2) below,  
$\htilde_{H}$ and $\hhat_H$ are comparable, and one advantage of $\htilde_{H}$ is that 
$\htilde_{H}$ 
differs from the (usual) Weil height by a bounded function.

\begin{Proposition}
\label{prop:northcott}
\begin{parts}
\Part{(1)}
We have $(1/2) \hhat_H \leq \htilde_{H} \leq  \hhat_{H}$. 
\Part{(2)}
We have $\htilde_{H} = h + O(1)$.
\Part{(3)}
Let $K$ be a number field. Then, for $P\in \Aff^2(\overline{K})$, 
$\htilde_{H}(P) = 0$ if and only if $P$ is periodic with respect $H$. 
\end{parts}
\end{Proposition}

\Proof
The assertion (1) is obvious from the definitions of $\hhat_H $ and $\htilde_{H}$. 
For (2), see \cite[Proposition~6.5]{Ka}, and 
for (3), see \cite[Theorem~6.3~(5) and Proposition~6.5]{Ka}. 
\QED

\begin{Remark}
For Proposition~\ref{prop:northcott}~(2), Lee~\cite[Theorem~6.5]{Lee} proved a stronger result: 
If $K$ is a number field, then the canonical height function $\htilde_{H}$ is a height function associated to 
$\Ocal_{\PP^2}(1)$ with semipositive adelic metrics. (For a brief review of semipositive adelically metrized line bundles, see \S~\ref{subsec:function field henon}.)
\end{Remark}

\begin{Remark}
\label{rmk:isotriviality}
Let $K$ be a product formula field such that any $v \in M_K$ is nonarchimedean. 
Assume that $\delta = 1$ in  equation~\eqref{eqn:Henon:map} of the H\'enon map $H$. 
Ingram~\cite[Theorem~1.2]{ing} showed that, if $H$ is not  isotrivial over $\overline{K}$, then 
the conclusion of Proposition~\ref{prop:northcott} (3) still holds: 
For $P\in \Aff^2(\overline{K})$, $\htilde_{H}(P) = 0$ if and only if $P$ is periodic with respect $H$. 
(For the definition of isotriviality, see \cite[p.~787]{ing}.)
It would be interesting to know if the same conclusion still holds when $\d\ne 1$. 
\end{Remark}

\subsection{One-parameter families of H\'enon maps: setting}
\label{sec:family of Henon}
Let $K$ be a field. We are interested in algebraic families of H\'enon maps defined over $K$. 
We give our setting that will be fixed throughout this paper. 

Let $f_t(x) \in K[t, x]$ be a polynomial such that, as a polynomial in $x$, $f_t(x)$ is monic and 
of degree $d \geq 2$. We write 
\begin{equation}
\label{eqn:f_{t} defn}
f_t(x) 
= x^{d} + \sum_{i=1}^{d} c_i(t) \, x^{d-i}  
= x^{d} + c_1(t) x^{d-1} + \cdots + c_d(t), 
\end{equation} 
where 
$c_i(t)\in K[t]$ for $i=1, \ldots, d$. 
Let $\delta \in K \setminus\{0\}$. 
We consider the one-parameter family of H\'enon maps $H_{t}$ in \eqref{eqn:Henon:map} parametrized by $t$: 
\begin{equation}
\label{eqn:Henon maps}
\bfH = (H_t)_{t \in \overline{K}}\colon
\Aff^2 \to \Aff^2, \quad
(x, y)
\mapsto
\left(\d y + f_t(x), x\right). 
\end{equation}
We also fix an initial point $\bfP = \left(a(t), b(t)\right) \in \Aff^2(K[t])$, which we regard 
as a family of initial points $(P_t)_{t \in \overline{K}}$. 

We remark on our symbols $\bfH$, $\bfP$ and $(H_t)_t$, $(P_t)_t$. 
We use the symbols $\bfH$ and $\bfP$ when they are viewed as a 
H\'{e}non map and a point defined over $K(t)$ (and also over $K[t]$). 
When we emphasize that they are algebraic families of H\'enon maps and points 
parametrized by $t\in \overline{K}$, we use the symbols $(H_t)_t$ and $(P_t)_t$. 

\begin{Example}
\label{eg:quad:Henon:map}
Let $H_t\colon \Aff^2 \to \Aff^2$ be the H\'enon map 
over a field $K$ defined by 
\begin{equation*}
  H_{t}\colon
\Aff^2 \to \Aff^2, \quad
(x, y)
\mapsto
\left(y + x^2 + t, x\right). 
\end{equation*}
Then $\bfH = (H_t)_{t \in \overline{K}}$ is an example of a one-parameter family of H\'enon maps 
of \eqref{eqn:Henon maps} with $\delta = 1$ and $f_t(x) = x^2 + t$. 
Some properties of this family are discussed in Introduction, Example~\ref{eg:quad:Henon:map:2}, Example~\ref{eg:figures:inou}, Example~\ref{eg:quad:Henon:map:3}, and 
Section~\ref{sec:finiteness:per:param}. 
\end{Example}

\subsection{One-parameter families of H\'enon maps: further setting}
\label{sec:family of Henon:continued}
For $n \geq 1$, we set
\begin{equation}
\label{eqn:Phi}
\Phi_n\colon \Aff^1 \to \Aff^4, \quad t \mapsto (H_{t}^n(P_t), H_{t}^{-n}(P_t)).
\end{equation}
We embed $\Aff^m$ into $\PP^m$ by $x \mapsto (x:1)$. We denote  
the extension of $\Phi_n$ to projective spaces by
\begin{equation}
\label{eqn:Phi:bar}
\overline{\Phi}_n\colon \PP^1 \to \PP^4. 
\end{equation}
Note that $\overline{\Phi}_n$ is a morphism for any $n \geq 1$, because $\PP^1$ is a smooth curve.  

To fix the notation, we introduce the following polynomials $A_n(t)$ and $B_n(t)$ associated to 
the H\'{e}non maps $\bfH^{n} = (H_t^n)_{t}$ and the initial point $\bfP = (P_t)_t$. 
Recall that the inverse $H_{t}^{-1}$ is given by
\begin{equation}
\label{eqn:Henon maps:inverse}
H_{t}^{-1}\colon
\Aff^2 \to \Aff^2, \quad
(x, y)
\mapsto
\left(y, \frac{1}{\d} \left(x - f_t(y)\right)\right). 
\end{equation}
 
\begin{Definition}[$A_{n}(t), B_{n}(t)$]
\label{def:An:Bn}
We define $A_n(t), B_n(t)  \in K[t]$  for $n \geq 0$ by the following recursive formulae
\begin{align}
\label{eqn:def:An}
A_{n+1}(t) & = \d A_{n-1}(t) + f_t\left(A_{n}(t)\right) \; \text{for $n\ge 1$}, 
 \quad A_1(t) = \d b(t) + f_t(a(t)),  \; A_0(t) = a(t),   \\
\label{eqn:def:Bn}
B_{n+1}(t) & = \frac{1}{\d}\left( B_{n-1}(t) - f_t\left(B_{n}(t)\right)\right)\; \text{for $n\ge 1$},  
\quad B_1(t) = \frac{1}{\d}\left(a(t) -f_t(b(t))\right), \; B_0(t) = b(t).
\end{align}
\end{Definition}
\noindent It follows from \eqref{eqn:Henon maps} and \eqref{eqn:Henon maps:inverse} that   
\begin{equation}
\label{eqn:H:n:A:B}
 H_t^{n}(P_t) = \left(A_{n}(t), A_{n-1}(t)\right) 
 \quad{\text{and}}\quad
 H_t^{-n}(P_t) = \left(B_{n-1}(t), B_{n}(t)\right)
\end{equation}
for any $n \geq 1$. We set 
\begin{equation}
\label{eqn:defn of elln}
\ell_n := \deg \overline{\Phi}_n = \max\left\{\deg A_n(t), \deg A_{n-1}(t), \deg B_n(t), \deg B_{n-1}(t)\right\}.
\end{equation}

We consider the following assumption on $\bfP$. 
\begin{Assumption}
\label{assumption}
The sequence $(\deg A_{n}(t))_{n \geq 0}$ or $(\deg B_{n}(t))_{n \geq 0}$ is unbounded. 
\end{Assumption}
\noindent 
Note that Assumption~\ref{assumption}  is mild. 
See Example~\ref{eg:quad:Henon:map:2} below, for example. 
As we see in Proposition~\ref{prop:on:assumption}, 
Assumption~\ref{assumption} is equivalent to $\tilde{h}_{\bfH}(\bfP) \neq 0$. 

\begin{Example}
\label{eg:quad:Henon:map:2}
Let $H_t(x, y) = (y + x^2 + t, x)$ as in Example~\ref{eg:quad:Henon:map}. Then 
it is easy to see that, for any $a, b \in K$, $\deg(A_n(t)) = 2^{n-1}$ and 
$\deg(B_n(t)) = 2^{n-1}$. Thus, for any constant initial point $\bfP = (a, b)$, 
Assumption~\ref{assumption} is satisfied. 
\end{Example}

\begin{Proposition}
\label{prop:on:assumption} 
\begin{parts}
\Part{(1)}
If $(\deg A_{n}(t))_{n \geq 0}$ is unbounded, 
then there exists $N \geq 1$ such that 
\begin{equation}
\label{eqn:on:assumption:1}
\deg A_{N+j}(t) = d^j \deg A_{N}(t)
\end{equation}
 for all $j \geq 0$. 
In particular, there exists a rational number $\ell >0$ such that $\deg A_n(t) = d^n \ell$  for any sufficiently large $n$. 
\Part{(2)}
If $(\deg B_{n}(t))_{n \geq 0}$ is unbounded, 
then there exists $N^\prime \geq 1$ such that 
\begin{equation*}
\deg B_{N^\prime +j}(t) = d^j \deg B_{N^\prime}(t) 
\end{equation*}
for all $j \geq 0$. 
In particular, there exists a rational number $\ell^\prime >0$ such that $\deg B_n(t) = d^n \ell^\prime$  for any sufficiently large $n$. 
\Part{(3)}
Assumption~\ref{assumption} is equivalent to $\tilde{h}_{\bfH}(\bfP) > 0$. 
\Part{(4)}
Under Assumption~\ref{assumption}, $\bfP$ is not periodic with respect $\bfH$. 
Further, there exists a rational number $\ell \geq 1$ such that, 
for any sufficiently large $n$, we have 
\begin{equation}
\label{eqn:elln}
 \ell_n = d^n \ell. 
\end{equation}
In particular,  
$\overline{\Phi}_n$ is not a constant map for any sufficiently large $n$.
\end{parts}
\end{Proposition}

\Proof
{\bf (1)} 
Fix a positive integer $D$ such that $D > \max\{\deg c_{1}(t), \ldots, \deg c_{d}(t), \deg A_{0}(t)\}.$  Since $(\deg A_{n}(t))_{n \geq 0}$ is unbounded, there exists an $n$ such that $\deg A_{n}(t) \ge D.$ Let $N$ be the smallest positive integer satisfying this property. Thus $\deg A_{N}(t) \ge D$ and $\deg A_{i}(t) < D$ for $ 0 \le i \le N-1.$ 
We prove \eqref{eqn:on:assumption:1} by induction on $n\ge N.$ We write $n = N+j$ with integer $j\ge 0.$ 

Clearly,  \eqref{eqn:on:assumption:1}  is true for $j=0.$ Let $j\ge 0$ and assume that $\deg A_{N+i}(t) = d^{i} \deg A_{N}(t)$ for integer $0\le i\le j.$ From the recursive relation for polynomials $A_{n}(t)$ we have  
\[
A_{N+j+1}(t) = \d A_{N+j-1}(t) + f_{t}(A_{N+j}(t)) = \d A_{N+j-1}(t) + A_{N+j}(t)^{d} + \sum_{i=1}^{d} c_{i}(t) A_{N+j}(t)^{d-i}.
\]
Note that, by induction, if $j \geq 1$, then  
$\deg (\d  A_{N+j-1}(t)) = d^{j-1} \deg A_{N}(t) < d^{j+1} \deg A_{N}(t)$. 
For $j = 0$, we still have $\deg (\d  A_{N-1}(t))  < D \leq \deg A_{N}(t) < d \deg A_N(t) = d^{j+1} \deg A_{N}(t)$. 
Further, we have 
\begin{align*}
\deg\left( c_{i}(t) A_{N+j}(t)^{d-i}\right)
& < D + (d-i) \deg A_{N+j}(t)  \\
& = D + (d-i) d^j \deg A_{N}(t) < d^{j+1}  \deg A_{N}(t ). 
\end{align*}
Since $\deg A_{N+j}(t)^{d} =  d \deg A_{N+j}(t) = d^{j+1} \deg A_{N}(t)$, we have 
$\deg A_{N+j+1}(t) =  d^{j+1} \deg A_{N}(t)$.  
This completes the induction steps  and hence the proof of \eqref{eqn:on:assumption:1}.

\medskip
{\bf (2)}
One can prove (2) by the same argument as (1). 

\medskip
{\bf (3)}
Note that the H\'enon map $\bfH$ and the point $\bfP$ are defined over the rational function field 
$K(t)$ of $\PP^1$ over $K$. As in Example~\ref{eg:prod:ft}, we regard $K(t)$ as a product formula field equipped with the set of absolute values $M_{K(t)}$. We denote by $\infty$ the point 
of $\PP^1(K)$ at which the function $t$ has a pole. 

Since $A_{n}(t)$ and $B_{n}(t)$ belong to $K[t]$, 
we have, for any $v \in M_{K(t)}$ with $v \neq \infty$ and for any $n \geq 1$,
 \[
 \norm{\bfH^{n}(\bfP)}_{v} \le 1\quad\text{and}\quad \norm{\bfH^{-n}(\bfP)}_{v} \le 1 . 
 \]
It follows that $G_{v}^{\pm}(\bfP) = 0$ if $v \ne \infty.$ 
Then, the canonical height of $\bfP$ is
\[
\htilde_{\bfH}(\bfP) = \sum_{v\in M_{K(t)}}\, n_{v} \max\{G^+_{v}(\bfP), G^-_{v}(\bfP)\} 
   = \max\{G^+_{\infty}(\bfP), G^-_{\infty}(\bfP)\} = \lim_{n\to\infty}\, \frac{\deg\overline{\Phi}_{n}}{d^{n}}. 
\]
Applying~(1) and~(2), we see that Assumption~\ref{assumption} is equivalent to $\htilde_{\bfH}(\bfP) > 0.$ 

\medskip
{\bf (4)}
The first assertion is obvious. The second assertion \eqref{eqn:elln} follows from~(1) and~(2). 
The last assertion is obvious from the second assertion. 
\QED

\begin{Remark}
\label{rem:alphaD-betaD}
\begin{parts}
\Part{(1)}
For later reference, we note that 
in Proposition~\ref{prop:on:assumption}~(1), if we write 
$A_N(t) = \a_{D} t^{D} + \cdots + \a_1 t + \a_0$ ($\a_0, \ldots, \a_{D}\in K$, $\a_D \neq 0$), 
then 
\begin{equation}
\label{eq:An leading coefficient}
A_{N+j}(t) = \left(\a_{D} t^{D} \right)^{d^j} + \;\text{(lower order terms)}.
\end{equation}
\Part{(2)}
Similarly, in Proposition~\ref{prop:on:assumption}~(2), if we write 
$
B_{N'}(t) = \b_{D'} t^{D'} + \cdots + \b_1 t + \b_0$ 
($\b_0, \ldots, \b_{D'}\in K$,  $\b_{D'} \neq 0$), 
then 
\begin{equation}
\label{eq:Bn leading coefficient}
B_{N'+j}(t) = \left(\b_{D'} t^{D'} \right)^{d^j} + \;\text{(lower order terms)}. 
\end{equation}
\end{parts}
\end{Remark}

\begin{Remark}
As we see in Proposition~\ref{prop:on:assumption}~(4), 
Assumption~\ref{assumption} implies that $\bfP$ is not periodic with respect $\bfH$. 
If $\delta = 1$ and $\bfH$ is not isotrivial, 
then it follows from 
Proposition~\ref{prop:on:assumption}~(3) and Remark~\ref{rmk:isotriviality} that 
Assumption~\ref{assumption} is indeed equivalent to 
the assumption that 
$\bfP$ is not periodic with respect $\bfH$. 
\end{Remark}

\subsection{Heights associated to semipositive adelically metrized line bundles}
\label{subsec:function field henon}

In this subsection, following Zhang \cite{zhang}, we recall the notion of semipositive adelically 
metrized line bundles over number fields and function fields. 
For details, see \cite{zhang, Gubler, Yuan}. 

Let $K$ be a number field (resp. a function field $F(B)$ of an integral projective variety 
$B$ over a field $F$ such that $B$ is regular in codimension one). 
In the following, we identify $M_K^{\rm fin}$ with 
$\Spec(O_K) \setminus \{(0)\}$ (resp. the $M_K$ with the set of all generic points of 
prime divisors of $B$). 

Let $X$ be an integral projective variety. 
An {\em adelically metrized line bundle} on $X$, denoted by $\Lcalbar := \left(L, \{\norm{\cdot}_v\}_{v\in M_K}\right)$, 
is a pair consisting of a line bundle $L$ on $X$ and a collection of metrics $\norm{\cdot} := \{\norm{\cdot}_v \mid v \in M_K\}$ 
such that the following conditions are satisfied: 
\begin{enumerate}
\item[(i)]
Each $\norm{\cdot}_v$ is a continuous and bounded metric of $L \otimes \KK_v$ that 
is $\Gal(\KK_v/K_v)$-invariant; 
\item[(ii)]
There exist a non-empty Zariski open subset $U$ of $\Spec(O_K)$ (resp. 
of $B$), a projective scheme $\mathscr{X}$ over $U$ with generic fiber $X$, 
a positive integer $e$,  and a line bundle $\mathscr{L}$ on $\mathscr{X}$ with
$L^{\otimes e} = \rest{\mathscr{L}}{X}$ such that 
$\norm{\cdot}_v^e$ is the metric induced from the model $\mathscr{L}_v := \mathscr{L} \otimes _U \KK_v^\circ$ 
at all $v \in M_K \cap U$. 
(When $K = F(B)$, the condition $v \in M_K \cap U$ means that 
$v \in M_K$ and the prime divisor corresponding to $v$ intersects with $U$.)
\end{enumerate}

A sequence $\{\norm{\cdot}_n\}_{n \geq 1}$ of adelic metrics on $L$ is said to 
{\em converge} to an adelic metric $\norm{\cdot}$ if there exists 
a non-empty Zariski open subset $U$ of $\Spec(O_K)$ (resp. 
of $B$) such that $\norm{\cdot}_{n, v} = \norm{\cdot}_{v}$ for all $n \geq 1$ and all $v \in M_K \cap U$, 
and $\log \left(\frac{\norm{\cdot}_{n, v}}{ \norm{\cdot}_v}\right)$ converges to $0$ uniformly on $X(\KK_v)$ 
for all places $v \in M_K \setminus U$. 
 
The adelically metrized line bundle $\Lcalbar = \left(L, \{\norm{\cdot}_v\}_{v\in M_K}\right)$ 
is said to be {\em semipositive} if there is a family $\Lcalbar_n = \left(L, \{\| \cdot \|_{n, v}\}_{v\in M_K} \right)$  
for $n \geq 1$ of adelically metrized line bundles satisfying the following conditions: 
\begin{enumerate}
\item[(a)] Metrics $\norm{\cdot}_n :=  \{\| \cdot \|_{n, v}\}$ converge 
to $\norm{\cdot} :=  \{\| \cdot \|_{v}\}$. 
\item[(b)] For any $v \in M_K^\infty$, the metric $\| \cdot  \|_{n, v}$ is smooth and
  the curvature $c_1(\Lcalbar)$ of $\|  \cdot \|_{n, v}$ is semipositive for each $n$; 
\item[(c)] 
For each $n \geq 1$, 
there exist a projective scheme $\mathscr{X}_n$ over $\Spec(O_K)$ (resp. of $B$) 
with generic fiber $X$, 
a positive integer $e_n$, and a line bundle $\mathscr{L}_n$ on $\mathscr{X}_n$ with 
$L^{\otimes e_n} = \rest{\mathscr{L}_n}{X}$ such that $\mathscr{L}_n$ is relatively nef  
and that $\norm{\cdot}_{n, v}$ is induced from $(\mathscr{X}_n, \mathscr{L}_n)$
for any $v \in M_K^{\rm fin}$.
\end{enumerate}

For any given semipositive adelically metrized line bundle $\Lcalbar = \left(L, \{\norm{\cdot}\}_{v\in M_K}\right)$ on
$X$, one can associate a height function $h_{\Lcalbar}$ on $X(\overline{K})$ as
follows:  For any $P \in X(\overline{K})$, we take a nonzero rational section $\eta$ of $L$ 
whose support is disjoint from the Galois conjugates of $P$ over $K$; Then we set
\begin{align*}
\label{eqn:ht of metrized line bundle}
h_{\Lcalbar}(P) 
& = \frac{1}{\left[K(P) : K\right]} \sum_{\sigma:K(P)\to \overline{K}}\sum_{v\in M_K}-
n_v \log \norm{\eta(P^{\sigma})}_v \\
\notag
& = 
\sum_{v \in M_K} n_v \left(
 \frac{1}{\left[K(P) : K\right]} \sum_{\sigma:K(P)\to \KK_v} - \log \norm{\eta(P^{\sigma})}_v
\right). 
\end{align*}
Note that, since $K$ satisfies the product formula, 
the value $h_{\Lcalbar}(P)$ does not depend on the choice of a nonzero rational section $\eta$. 
We say that $h_{\Lcalbar}$ is a {\em height associated to a semipositive adelically metrized line bundle}. 

\subsection{Variation of canonical heights in families}
\label{subsec: height specialization}

Let $K$ be a number field, and let  
$\bfH = (H_{t})_{t \in \overline{K}}$ be a H\'enon map over $K$ defined in~\eqref{eqn:Henon maps}. 

Let $\bfP\in\Aff^2(K[T])$. 
In this setting, we have the (function field) height $\hhat_{\bfH}(\bfP)$ defined in~\eqref{eqn:can:height:pm}, and for each $t \in \overline{K}$, we have the specialized 
height $\hhat_{H_t}(P_t)$ defined in~\eqref{eqn:can:height:pm}. 
Ingram~\cite{ing} studied variation of $\hhat_{H_t}(P_t)$  
and obtained the following theorem. 

\begin{Theorem}[{\cite[Theorem~1.1]{ing}}]
\label{thm:Ingram:variation}
Let $K$ be a number field. 
We have, for $t \in \overline{K}$, 
\begin{equation}
\label{eqn:ingram:variation}
\hhat_{H_t}(P_t) =
\hhat_{\bfH}(\bfP) h(t) + O(1), 
\end{equation}
where $h(t)$ is the (usual)  Weil height function 
on the parameter space $t \in \Aff^1(\overline{K})$. 
\end{Theorem}

\begin{Remark}
In fact, Ingram~\cite{ing} treated a more general case, where $K(T)$ is replaced by 
the function field of a smooth projective curve $C$ defined over a number field $K$, and in this case, 
$h$ is replaced by a height $h_\eta$ associated to a degree one divisor $\eta\in \Pic(C)$ 
and~$O(1)$ is replaced by $O(h_{\eta}(t)^{1/2})$.
\end{Remark}

These types of results go back to the study of 
variation of canonical heights of abelian varieties in ~\cite{silverman, tate}. 
See Introduction for more accounts. 
As we discuss in Introduction, just as a N\'eron-Tate height is important 
in the study of the arithmetic of abelian varieties, to study arithmetic properties of 
families of H\'enon maps, a key ingredient is 
the existence of a ``nice'' height function on the parameter space 
that encodes  arithmetic information of the periodic parameters $\Sigma(\bfP).$ 

\medskip
One of the main results in this paper is to show that for families of H\'enon maps such a 
 height function indeed exists on the parameter space (for $\htilde$ in place of $\hhat$). To be precise, 
Theorem~\ref{thm:main:A} in Introduction asserts that, over a number field or a function field as in Example~\ref{eg:prod:ft}, if we put 
$h_{\bfP}(t) := \htilde_{H_t}(P_t)/\htilde_{\bfH}(\bfP)$, then 
$h_{\bfP}$ is the restriction of the height function associated to a semipositive adelically metrized
line bundle $\Lcalbar_{\bfP}= \left(\Ocal_{\PP^{1}}(1), \{\norm{\cdot}_{v}\}_{v\in M_{K}}\right)$ on~$\PP^{1}$. 

The proof of Theorem~\ref{thm:main:A} will be given in Section~\ref{sect:heights and metrized line bundles} after establishing all technical results in the next section.

\setcounter{equation}{0}
\section{Construction of a metrized line bundle over a valued field}
\label{sec:valued fields}

In this section, we construct a metrized line bundle on $\PP^{1}$ 
over an algebraically closed valued field, which will be used 
to prove Theorem~\ref{thm:main:A}. 

Let $\Omega$ be an algebraically closed field that is complete with respect to 
an absolute value $|\cdot|$. Let $\bfH = (H_t)_{t \in \Omega}$ and $\bfP = (a(t), b(t)) = ((P_t)_{t \in \Omega})\in \Aff^2(\Omega[t])$ be as in \S~\ref{sec:family of Henon} such that $\bfP$ satisfies Assumption~\ref{assumption}, where the field $K$ in \S~\ref{sec:family of Henon} is $\Omega$ in this section. 

For later use, we set 
\begin{equation}
\label{eqn:def:m}
m  := \max \{\deg(c_{1}(t)),\ldots, \deg(c_d(t)), \deg(a(t)), \deg(b(t))\}.
\end{equation}
\noindent We also put 
\begin{equation*}
\Delta := \max\left(|\d|, 1/|\d|\right) \geq 1. 
\end{equation*}

We use the following convenient notation: 
for a positive real number $r$, we set 
\[
 [r]
 := 
 \begin{cases}
 r & \quad \text{(if $|\cdot|$ is archimedean),} \\
 1 &  \quad \text{(if $|\cdot|$ is nonarchimedean)}. 
 \end{cases}
\]

\subsection{Statement of the result}

Let $P = (x_0: x_1: \ldots: x_4) \in \PP^4(\Omega)$.  
Let $X_0, \ldots, X_4$ be homogeneous coordinates of $\PP^4$. 
Then for any $\xi = c_0 X_0 + \cdots  + c_4 X_4 \in H^0(\OO_{\PP^4}(1))$, 
\[
  \Vert \xi \Vert_{\mathrm{st}}(P)
  := \frac{|c_0 x_0 + \cdots  + c_4 x_4|}{\norm{(x_0, x_1, \ldots, x_4)}}
\]
is well-defined, thus defining a metric $\Vert \cdot \Vert_{\mathrm{st}}$ (called the standard metric)  on $\OO_{\PP^4}(1)$. 

Recall that, for each positive integer $n$, we have defined a morphism
\begin{align}
\label{eqn:Phi:bar:2}
  & \overline{\Phi}_n\colon \PP^1 \to \PP^4, \\
  \notag
  & \qquad (t:1) \mapsto (H_{t}^n(P_t): H_{t}^{-n}(P_t): 1) 
  = (A_n(t): A_{n-1}(t): B_{n-1}(t): B_n(t):1)
\end{align}
in \eqref{eqn:Phi} and \eqref{eqn:Phi:bar}. 
We write $\ell_n := \deg(\overline{\Phi}_n)$ as before. 
We remark that by Proposition~\ref{prop:on:assumption}~(4), $\ell_n = d^n \ell > 0$ for sufficiently large $n$, so that $\overline{\Phi}_n$ is not a constant morphism.

Now let $n$ be a large integer such that $\ell_{n} > 0.$ 
We define a metric $\Vert \cdot \Vert_{n}$ on $\OO_{\PP^1}(1)$ by pulling back
$\Vert \cdot \Vert_{\mathrm{st}}$ via  $\overline{\Phi}_n.$  To do this, 
we take a lift $\Aff^2 \to \Aff^5$ of $\overline{\Phi}_n$ given by 
$(t, 1) \mapsto (H_{t}^n(P_t), H_{t}^{-n}(P_t), 1)$ which determines 
the isomorphism $\alpha_n\colon \overline{\Phi}_n^*(\OO_{\PP^4}(1)) \cong \OO_{\PP^1}(1)^{\otimes \ell_n}$.
Then we define 
\[
\Vert\cdot\Vert_{n} := \left(\alpha_n (\overline{\Phi}_n^* \Vert\cdot \Vert_{{\rm st}})\right)^{1/\ell_n}. 
\]
Concretely, for a section $\eta = a_0 X_0 + a_1 X_1$ of $H^0(\OO_{\PP^1}(1))$, we have 
\begin{equation}
\label{eqn:def:norm:n}
  \Vert \eta \Vert_{n}(t:1) = \frac{|a_0 t + a_1|}{\norm{(A_n(t), A_{n-1}(t), B_{n-1}(t), B_n(t), 1)}^{1/\ell_n}}. 
\end{equation}

The purpose of this section is to prove the following theorem.

\begin{Theorem}
\label{thm:main:Mn}
The metrics $\{\log \Vert\cdot\Vert_{n}\}$ uniformly converge on
$\PP^{1}(\Omega)$ as $n \to \infty$.
\end{Theorem}

\noindent The proof of Theorem~\ref{thm:main:Mn} will be given in subsequent subsections.
We first introduce a quantity related
to the denominator of~\eqref{eqn:def:norm:n}.

\begin{Definition}[$M_{n}(t)$]
\label{def:M:n:t}
Let $n \geq 1$.
For $t \in \Aff^1(\Omega)$, we set
\[
  M_{n}(t)  
   := \max\left\{|A_n(t)|,  |A_{n-1}(t)|, |B_{n-1}(t)|, |B_n(t)|, 1\right\}. 
\]
\end{Definition}

For a positive constant $L$,  we denote by  $\mathcal{B}_L$ the bounded region in the parameter space $\Aff^1(\Omega)$ defined 
by
\[
 \mathcal{B}_L = \{t \in \Omega \mid |t| \leq L\}.
\]
If we take any sufficiently large $L$,  we then have, for any $t \in \mathcal{B}_L$,  
\begin{equation}
\label{eq:large L}
|c_1(t)| \leq L^{m+1}, \; \ldots, \; |c_d(t)| \leq L^{m+1}, \;
|a(t)| \leq L^{m+1}, \; |b(t)| \leq L^{m+1}, \; [3d] \D \leq L^{m+1}. 
\end{equation}
In what follows, we fix $L \geq 1$ that satisfies \eqref{eq:large L}.

\subsection{Uniform convergence on a bounded region of $\Aff^1$}
\label{subsec:bounded region}

In this subsection, we compare $M_{n+1}(t)$ with $M_{n}(t)^{d}$ 
on $\mathcal{B}_L$. 

\begin{Lemma}
\label{lem:estimate:Mn:upper}
Let $n \geq 1$ be any integer. Then
for any $t \in \mathcal{B}_L$, we have
\[
M_{n+1}(t) \leq [d+2] \,\D \,L^{m+1}\, M_n(t)^{d}.
\]
\end{Lemma}

\Proof
To simplify the notation, we write $M_n, A_n, B_n$ for $M_n(t), A_n(t), B_n(t)$.
By \eqref{eqn:def:An}, we have
\begin{align*}
  |A_{n+1}| & = \left|\d \, A_{n-1}+A_{n}^{d} + \sum_{i=1}^{d} c_i(t) A_n^{d-i}\right| \\
  & \leq [d+2] \max\left\{|\d \,A_{n-1}|, |A_{n}|^{d}, |c_1(t)A_n^{d-1}|, \ldots, |c_{d-1}(t)A_n|,  |c_d(t)| \right\}
\\
  & \leq [d+2] \max\left\{\D\, |A_{n-1}|, |A_{n}|^{d}, L^{m+1} |A_n|^{d-1}, \ldots, L^{m+1} |A_n|, L^{m+1}
\right\} \\
  &  \leq [d+2] \,\D\,  L^{m+1} \max\left\{|A_{n-1}|, |A_{n}|^{d}, |A_{n}|^{d-1},\ldots, 1\right\}  \qquad \text{(since
$L \ge 1$ and $\Delta \ge 1$)} \\
 &   \leq [d+2]\,\D\,  L^{m+1}\, M_n^{d} .
\end{align*}
Similarly, using \eqref{eqn:def:Bn} we have
\begin{align*}
|B_{n+1}| 
& = \left| \frac{1}{\d}\left(B_{n-1} - B_{n}^{d} - \sum_{i=1}^{d} c_{i} (t) B_{n}^{d-i} \right)\right| \\
& \leq \D \, \left| B_{n-1} - B_{n}^{d} - \sum_{i=1}^{d} c_{i} (t) B_{n}^{d-i} \right| \\
& \leq [d+2]\,\D\,  \max\left\{ |B_{n-1}|, |B_{n}|^{d}, L^{m+1} |B_n|^{d-1}, \ldots, L^{m+1} |B_n|,
L^{m+1} \right\} \\
& \leq [d+2] \,\D\,  L^{m+1} \max\left\{|B_{n-1}|, |B_{n}|^{d}, |B_{n}|^{d-1},\ldots, 1\right\}  \qquad \text{(since
$L \ge 1$)} \\
              & \leq [d+2] \, \D\,L^{m+1} M_n^{d}. 
\end{align*}
We also have
$$|A_n| \leq \max\{|A_n|, 1\} \leq [d+2] \,\D\,L^{m+1} \max\{|A_n|, 1\}^{d} \leq [d+2]\,\D\, L^{m+1}
M_n^{d},$$
and $|B_n| \leq [d+2]\,\D\, L^{m+1} M_n^{d}$ as well.
Combining these estimates,  we get
\[
  M_{n+1}
  = \max\left\{|A_{n+1}|,  |A_{n}|,
     |B_{n}|,   |B_{n+1}|, 1\right\}  \leq [d+2]\,\D\,  L^{m+1} M_n^{d}.
\]
This completes the proof.
\QED

\noindent For the other direction, we have the following estimate.

\begin{Proposition}
\label{prop:estimate:Mn:lower}
Let $n \geq 1$ be any integer. Then
for any $t\in \mathcal{B}_L$, we have
\[
M_{n+1}(t) \geq \frac{ M_n(t)^{d}}{([3d]\,\D\, L^{m+1})^d}.
\]
\end{Proposition}
\medskip

To prove Proposition~\ref{prop:estimate:Mn:lower}, 
we use some properties of H\'enon maps.
Set
\begin{align}
\label{eqn:VL}
V_L^+
& =
\left\{
(x, y) \in \Aff^2(\Omega) \;\left|\; \Vert (x, y)\Vert  \leq [3d] \,\D\, L^{m+1} \quad\text{or}\quad
\Vert H_{t}(x, y) \Vert \geq \frac{1}{L^{m+1}} \Vert (x, y)\Vert^{d}
\right.\right\},
\\
\notag
V_L^-
& =
\left\{
(x, y) \in \Aff^2(\Omega) \;\left|\; \Vert (x, y)\Vert  \leq [3d] \,\D\, L^{m+1} \quad\text{or}\quad
\Vert H_{t}^{-1}(x, y) \Vert \geq \frac{1}{L^{m+1}} \Vert (x, y)\Vert^{d}
\right.\right\}.
\end{align}

\noindent 
Before we prove Proposition~\ref{prop:estimate:Mn:lower}, 
we show the following lemma, which gives a filtration  according to 
behaviour of the orbits of points under the action of the H\'{e}non maps.
\begin{Lemma}
\label{lem:property:Henon}
Let $t \in \mathcal{B}_L$.
Then the following hold.
\begin{parts}
\Part{(1)}
$H_{t}(V_L^+ ) \subseteq V_L^+$ and
$H_{t}^{-1}(V_L^-) \subseteq V_L^-$. 
\Part{(2)}
For any $(x, y) \in V_L^+$,
we have
$$\max\left\{\Vert H_{t}(x, y) \Vert , 1\right\} \geq \frac{\max\{\Vert (x, y)\Vert, 1\}^{d}}
{([3d]\D\,L^{m+1})^d} .$$
\Part{(3)}
For any $(x, y) \in V_L^-$, we have $$\max\left\{\Vert H_{t}^{-1}(x, y) \Vert , 1\right\} \geq \frac{\max\{\Vert
(x, y)\Vert, 1\}^{d}}{([3d]\,\D\,L^{m+1})^d} .$$
\Part{(4)}
$\Aff^2(\Omega)  = V_L^+ \cup V_L^-$.
\end{parts}
\end{Lemma}

\Proof
This is essentially shown in \cite[Lemmas~2.6, 2.7, 2.9]{Ka} with slightly different $V_L^+$ and $V_L^-$.
Here we give an explicit proof. We first give useful estimates.

Let $(x, y) \in \Aff^2(\Omega)$ be any point.
Noting that
\[
H_{t}^{-1}(x, y) = \left(y, \frac{1}{\d}\left(x- y^{d} - \sum_{i=1}^{d}c_i(t) y^{d-i}\right)\right),
\]
we have
\begin{align}
\label{eqn:property:Henon:1}
|y^{d}| & = \left|\left(x- y^{d} - \sum_{i=1}^{d}c_i(t) y^{d-i}\right) + \left(-x  + \sum_{i=1}^{d}c_i(t) y^{d-i}\right)
\right| \\ \notag
& \leq [3] \max\left\{|\d|\,\left|\frac{1}{\d}(x- y^{d} - \sum_{i=1}^{d}c_i(t) y^{d-i})\right|,\, |x|,\,  \left|\sum_{i=1}^{d}
c_i(t) y^{d-i}\right|\right\}  \\ \notag
& \leq [3] \,\max\{\D\,\Vert H_{t}^{-1}(x, y) \Vert, |x|, [d]  L^{m+1} \Vert(y, 1)\Vert^{d-1} \} \\ \notag
& \leq [3d]\,\D\, \max\left\{\Vert H_{t}^{-1}(x, y) \Vert, L^{m+1}   \Vert (x, y, 1) \Vert^{d-1}\right\} \\
\notag
& \leq [3d]\,\D\, \max\left\{\Vert H_{t}^{-1}(x, y) \Vert, L^{m+1}   \Vert (x, y) \Vert^{d-1}, L^{m+1} \right\}.
\end{align}
Similarly, since
\[
H_{t}(x, y) = \left(\d y + x^{d}  + \sum_{i=1}^{d}c_i(t)x^{d-i}, x\right)
\]
and  $x^{d} = (\d y +x^{d}  + \sum_{i=1}^{d}c_i(t)x^{d-i}) - (\d y  + \sum_{i=1}^{d}c_i(t)x^{d-i})$, 
we have
\begin{equation}
\label{eqn:property:Henon:2}
  |x^{d}| \leq [3d]\,\D\, \max\left\{\Vert H_{t}(x, y) \Vert, L^{m+1}   \Vert (x, y) \Vert^{d-1}, L^{m+1}
\right\} .
\end{equation}
In conclusion,
for any $(x, y) \in \Aff^2(\Omega)$, we obtain from \eqref{eqn:property:Henon:1} and \eqref{eqn:property:Henon:2}
that
\begin{equation}
\label{eqn:property:Henon:3}
 \Vert (x, y) \Vert^{d} \leq [3d]\,\D\, \max\left\{\Vert H_{t}^{-1}(x, y) \Vert, \Vert H_{t}(x, y) \Vert, L^{m+1}
\Vert (x, y) \Vert^{d-1}, L^{m+1}  \right\}.
\end{equation}
\smallskip
Now we prove the statements (1)--(4) of the lemma.
\medskip
\\
\noindent {\bf (1)} We first show the first assertion.

Suppose that $(x, y) \not\in V_L^+$. We will show that $H_{t}^{-1}(x, y) \not\in V_L^+$, which
gives $\Aff^2(\Omega) \setminus V_L^+ \supseteq H_{t}^{-1}\left(\Aff^2(\Omega) \setminus V_L^+\right)$,
and thus  $H_{t}(V_L^+ ) \subseteq V_L^+$. In other words, we assume
\begin{equation}
\label{eqn:property:Henon:known}
  \Vert (x, y) \Vert > [3d] \,\D\, L^{m+1}  \quad\text{and}\quad \Vert H_{t}(x, y) \Vert < \frac{1}{L^{m+1} }
\Vert (x, y)\Vert^{d},
\end{equation}
and we will show
\begin{equation}
\label{eqn:property:Henon:goal}
\Vert H_{t}^{-1}(x, y) \Vert > [3d]\,\D \, L^{m+1}  \quad\text{and}\quad
\Vert (x, y)\Vert < \frac{1}{ L^{m+1} } \Vert H_{t}^{-1}(x, y) \Vert^{d}.
\end{equation}
Observe that \eqref{eqn:property:Henon:3}, \eqref{eqn:property:Henon:known}  and the choice of $L$  
give 
\begin{equation}
\label{eqn:conseq:Henon:1}
\Vert (x, y) \Vert^{d} \leq [3d] \,\D\, \Vert H_{t}^{-1}(x, y) \Vert.
\end{equation}
It follows from \eqref{eqn:conseq:Henon:1}, \eqref{eqn:property:Henon:known} and $d \geq 2$ that
\[
\Vert H_{t}^{-1}(x, y) \Vert \geq \frac{\Vert (x, y) \Vert^{d}}{[3d] \,\D} > \frac{([3d]\,\D\,L^{m+1})^{d}}
{[3d] \,\D} \geq  [3d] \,\D\, L^{m+1}.
\]
This gives the first part of~\eqref{eqn:property:Henon:goal}.

For the second part of~\eqref{eqn:property:Henon:goal}, 
using~\eqref{eqn:conseq:Henon:1}, \eqref{eqn:property:Henon:known} and~$d \geq 2,$ 
the same estimate as above gives
\[
\norm{H_t^{-1}(x,y)}  \geq \frac{\norm{(x,y)}^{d}}{[3d]\,\D} > \frac{([3d]\,\D L^{m+1})^{d-1}}{[3d]\,\D}
\, \norm{(x,y)} \geq \norm{(x,y)}.
\]
Thus
\[
\frac{1}{L^{m+1} } \Vert H_{ t}^{-1}(x, y) \Vert^{d}
> \frac{1}{L^{m+1} } \norm{(x,y)}^{d}  
> \frac{([3d] \, \D \, L^{m+1})^{d-1}}{L^{m+1}} \norm{(x, y)} \geq \Vert (x, y)\Vert.  
\]
This completes the proof of the first assertion of (1). 
Similar arguments also show that $H_{t}^{-1}(V_L^-) \subseteq V_L^-$ which gives the second assertion of (1).
\medskip
\\
{\bf (2)}  Let $(x,y)\in V_L^+$ be given.  By the definition of $V_L^+$ (see \eqref{eqn:VL}), we have
\[
\Vert (x, y)\Vert  \leq [3d] \,\D\, L^{m+1} \quad\text{or}\quad
\Vert H_{t}(x, y) \Vert \geq \frac{1}{L^{m+1}} \Vert (x, y)\Vert^{d} > \frac{\Vert (x, y)\Vert^{d}}
{([3d]\D\,L^{m+1})^d}.
\]
If the latter inequality holds, then 
noting that $1 \geq 1/([3d]\D\,L^{m+1})^d$ 
we obtain the assertion. Suppose that the former inequality holds. Then  $\Vert (x,
y)\Vert^{d}  \leq \left([3d]\,\D \, L^{(m+1)}\right)^d$.  Consequently,
\[
\max\left\{\Vert H_{t}(x, y) \Vert , 1\right\} \geq
1 \geq
\frac{1}{([3d]\,\D\,L^{(m+1)})^d} \max\left\{\Vert (x, y)\Vert^{d}, 1\right\},
\]
and we also obtain the assertion.
\medskip
\\
{\bf (3)}  
By symmetry, one can show (3) by the same arguments as (2). 
\medskip
\\
{\bf (4)}  
We claim that 
there does not exist $(x, y) \in \Aff^2(\Omega)$ such that
\begin{equation}
\label{eqn:property:Henon}
\Vert (x, y)\Vert  >[3d]\,\D\, L^{m+1}  , \quad
\Vert H_{t}(x, y) \Vert < \frac{1}{L^{m+1}  } \Vert (x, y)\Vert^{d} \quad\text{and}\quad
\Vert H_{t}^{-1}(x, y) \Vert < \frac{1}{L^{m+1} } \Vert (x, y)\Vert^{d}.
\end{equation}

Indeed, suppose that there exists $(x, y) \in  \Aff^2(\Omega)$ satisfying \eqref{eqn:property:Henon}. 
It follows from the first condition that 
\begin{equation}
\label{eqn:x:y:1}
\Vert (x, y)\Vert^d > [3d]\,\D\, L^{m+1} \Vert (x, y)\Vert^{d-1}
\quad\text{and}\quad
\Vert (x, y)\Vert^d > [3d]\,\D\, L^{m+1}. 
\end{equation}
Since $L^{m+1} > [3d]\,\D$ by our choice of $L$ (see \eqref{eq:large L}), 
the second condition in~\eqref{eqn:property:Henon} gives 
\begin{equation}
\label{eqn:x:y:2}
\Vert (x, y)\Vert^d > 
 [3d]\,\D\,  \Vert H_t(x, y)\Vert. 
\end{equation}
Similarly, the third condition gives 
\begin{equation}
\label{eqn:x:y:3}
\Vert (x, y)\Vert^d > [3d]\,\D\,  \Vert H^{-1}_t(x, y)\Vert. 
\end{equation}
These estimates \eqref{eqn:x:y:1},  \eqref{eqn:x:y:2},  \eqref{eqn:x:y:3} contradict~\eqref{eqn:property:Henon:3}. 

Thus we have shown the claim. Taking the converse, we obtain (4). 
\QED

\medskip
\noindent \textsl{Proof of Proposition~\ref{prop:estimate:Mn:lower}.}\quad
As in the proof of Lemma~\ref{lem:estimate:Mn:upper}, 
we write $M_n$, $A_n$, $B_n$ for $M_n(t)$, $A_n(t)$, $B_n(t)$.

Let $t \in \Bcal_L$. Let $P_t = (a(t), b(t))$ be the initial point. 
Since 
$\Vert P_t \Vert  \leq  L^{m+1} \leq [3d] \Delta L^{m+1}$ (see~\eqref{eq:large L}), we have
$P_t \in V_L^+$. It follows from  Lemma~\ref{lem:property:Henon}~(1) that
$H_{t}^n\left(P_t\right) \in V_L^+$ for any $n \geq 1$ and Lemma~\ref{lem:property:Henon}~(2) says that
\[
\max\left\{\Vert H_{t}^{n+1}\left(P_t\right) \Vert , 1\right\} \geq \frac{1}{([3d]\,\D\,L^{m+1})^d}
\max\left\{\Vert H_{t}^n\left(P_t\right)\Vert^{d}, 1\right\}.
\]
Since  $H_{t}^n\left(P_t\right) = \left(A_n, A_{n-1}\right)$, we obtain
\[
\max\left\{|A_{n+1}|, |A_n|, 1\right\} \geq \frac{1}{([3d]\,\D\,L^{m+1})^d} \max\left\{|A_{n}|^{d}, |A_{n-1}|
^{d}, 1\right\}.
\]
Similarly, using Lemma~\ref{lem:property:Henon}~(1)~(3), we have 
\[
\max\left\{|B_{n+1}|, |B_n|, 1\right\} \geq \frac{1}{([3d]\,\D\,L^{m+1})^d} \max\left\{|B_{n}|^{d}, |B_{n-1}|
^{d}, 1\right\}.
\]
Thus
\[
\max\left\{|A_{n+1}|, |A_n|, |B_{n+1}|, |B_n|,1\right\} \geq \frac{1}{([3d]\,\D\,L^{m+1})^d} \max\left\{|A_{n}|, |
A_{n-1}|, |B_{n}|, |B_{n-1}|, 1\right\}^{d},
\]
which completes the proof of Proposition~\ref{prop:estimate:Mn:lower}.
\QED

\subsection{Uniform convergence on an unbounded region of $\Aff^1$}
\label{subsec:uniform:convergence:unbounded}
In this subsection we compare $M_{n+1}(t)$ with $M_n^{d}(t)$ on $\Aff^1(\Omega)\setminus \Bcal_L$. 
We begin with the following lemma.
\begin{Lemma}
\label{lem:unbounded estimate}
Let $s \geq 2$ be an integer. 
Let 
\[
 \vp(x,y) = \g y + c x^s + \sum_{i=1}^s c_i(t) x^{s-i} \in (\Omega[t])[x, y],
\] 
where  $\g, c\in \Omega\setminus\{0\}$ and $c_i(t)\in
\Omega[t]$ for $i =1,\ldots, s$. We take an integer $m \geq 1$ with $m \geq  \max_i\{\deg c_i(t)\}$. 
Let $x(t), y(t) \in \Omega[t]$ be polynomials  
such that there exists an $L^\prime \geq 1$ with the following properties:
\begin{enumerate}
\item[(i)]
$|x(t)| \geq |t|^{m+1}$ for any $t \in \Omega$ with $|t| >L^\prime$; 
\item[(ii)]
$|y(t)/x(t)^s| \leq 1/|t|^s$ for any $t \in \Omega$ with $|t| >L^\prime$. 
\end{enumerate}
Then, for any $\k > 1$, there exists an $L \geq 1$ with the following properties: 
\begin{enumerate}
\item[(a)]
$L^{m+1} \geq \sqrt[s-1]{\k/|c|}$;  
\smallskip
\item[(b)]
$\displaystyle \frac{|c|\,|x(t)|^s}{\k} \leq |\vp(x(t),y(t))| \leq \k \,|c|\,|x(t)|^s$  
for any $t \in \Omega$ with $|t| >L$;  
\smallskip
\item[(c)]
$\displaystyle |\vp(x(t),y(t))| \geq |t|^{m+1}$ for any $t \in \Omega$ with $|t| >L$;  \\
\smallskip
\item[(d)]
$\displaystyle \left|\frac{x(t)}{\vp(x(t),y(t))^s}\right| \leq \frac{1}{|t|^s}$ 
for any $t \in \Omega$ with $|t| >L$. 
\end{enumerate}
\end{Lemma}

\Proof
First, we note that if we take any sufficiently large $L \geq L^\prime$, then (a) holds. 

We write
\[
\vp(x,y) = c x^s\left(1 + \frac{1}{c}\left(\sum_{i=1}^s \frac{c_i(t)}{x^i} + \frac{\g\,y}{x^s}\right)\right).
\]

For any $t \in \Omega$ with $|t| \geq L^\prime$,  conditions~(i) and~(ii) give 
\[
\left| \sum_{i=1}^s \frac{c_i(t)}{x(t)^i} + \frac{\g\,y(t)}{x(t)^s} \right| 
\leq [s+1] \max\left\{\frac{|c_1(t)|}{|t|^{m+1}}, \, 
\ldots, \, \frac{|c_s(t)|}{|t|^{s(m+1)}} , \,  \frac{|\g|}{|t|^s} \right\}. 
\]
Since $\deg c_i(t) \le m$ for $i = 1, \ldots, s$, we have 
\[
 [s+1]  \lim_{|t|\to\infty}\,
\max\left\{\frac{|c_1(t)|}{|t|^{m+1}}, \, 
\ldots, \, \frac{|c_s(t)|}{|t|^{s(m+1)}} ,  \,  \frac{|\g|}{|t|^s} \right\}  = 0, 
\]
hence 
\[
\lim_{|t|\to\infty}\, \left| \sum_{i=1}^s \frac{c_i(t)}{x(t)^i} + \frac{\g\,y(t)}{x(t)^s}\right| = 0.
\]

We set $\a  := 1 - 1/\k$. Since $\kappa > 1$, we have $0 < \alpha < 1$. 
We note that $1 + \alpha = 2 - 1/\k <  \k$.  
If necessary, we replace $L$ by a larger number, and 
we may assume that 
\[
 \left| \sum_{i=1}^s \frac{c_i(t)}{x(t)^i} + \frac{\g\,y(t)}{x(t)^s}\right|  
 \leq |c| \a  \quad \text{for all $t \in \Omega$ with $|t|> L$.}
\]
Then, for any $t \in \Omega$ with  $|t| > L$, we have 
\[
(1 - \a)\,|c| \,|x(t)|^s \le |\vp(x(t),y(t))| \le (1+\a)\,|c|\,|x(t)|^s .
\]
By the definition of $\a$, we conclude that  $|c |\, |x(t)|^s/\k \leq |\vp(x(t),y(t))| \leq \k |c|\,|x(t)|^s.$ 
This proves (b).

For (c), by  condition~(i) and $L \geq L^\prime$, we have 
$|x(t)| \geq |t|^{m+1}$ for any $t \in \Omega$ with $|t| > L$. Then 
\[
|\vp(x(t),y(t))|  \geq \frac{|c|\,|x(t)|^s}{\k}
   \geq \frac{|c| \, |t|^{(s-1)(m+1)}}{\k} |t|^{m+1} 
   \geq \frac{L^{(s-1)(m+1)}|c|}{\k} |t|^{m+1}  \geq |t|^{m+1}, 
\]
where we use the property (a) in the last inequality. 
For (d), observe that
\begin{align*}
\left|\frac{x(t)}{\vp(x(t),y(t))^s}\right| & \le \frac{\left(\k/|c|\right)^s |x(t)|}{|x(t)|^{s^2}} \le  \frac{\left(\k/|c|\right)^s }{|x(t)|
^{s^2-1}} \\
& \leq \frac{\left(\k/|c|\right)^s}{|t|^{(s^2-1)(m+1)}} \leq \frac{\left(\k/|c|\right)^s}{L^{((s^2-1)(m+1)-s)}} \frac{1}{|t|
^s}.
\end{align*}
Since $s \geq 2$ and $m \geq 1$, we have $(s^2-1)(m+1)-s \ge s(s-1)(m+1)$.  
If follows from (a) that
$$L^{(s^2-1)(m+1)-s} \ge L^{s(s-1)(m+1)} \geq \left(\frac{\k}{|c|}\right)^s$$
and thus $|x(t)/\vp(x(t),y(t))^s| \leq 1/|t|^s$ for all $t \in \Omega$ with $|t| \geq L$.
\QED

\begin{Remark}
\label{rem:condition for L}
The proof of Lemma~\ref{lem:unbounded estimate} shows that, 
if a number $L$ satisfies  
$L \geq L^\prime$, $L^{m+1} \geq \sqrt[s-1]{\k/|c|}$, and 
\begin{equation}
\label{remark:choice of L}
\left|\sum_{i=1}^s  \frac{c_i(t)}{t^{i(m+1)}} + \frac{\g}{t^s}\right| \leq |c|\,\left(  1 - \frac{1}{\k}\right)\;
\;\text{for all  $t \in \Omega$ with $|t| > L$}, 
\end{equation}
then we have (a)--(d) for such $L$. 
\end{Remark}

Applying Lemma~\ref{lem:unbounded estimate} to the family of H\'{e}non maps $H_t$ with initial point $P_t =
\left(a(t),b(t)\right)$ we have the following.

\begin{Proposition}
\label{prop:unbounded estimate}
Let $\bfH = (H_t)_{t \in \Omega}$ be the family of H\'enon maps, and let 
$\bfP = (P_t)_{t \in \Omega}$ be the family of initial points 
as in the beginning of \S~\ref{sec:valued fields}. 
Fix a constant $\k > 1.$  Then, the following statements hold.
\smallskip

\begin{parts}
\Part{(1)} 
If $(\deg A_n(t))_{n \geq 0}$ is unbounded, then there exist an  $L_{A} \geq 1$ and an
integer $N_{A} \ge 1$  such that for all $n\ge N_{A}$  
\begin{align*}
& \norm{H_t^{n}(P_t)} = |A_n(t)| \quad\text{and}\quad
\frac{|A_n(t)|^{d}}{\k}  \le \norm{H_t^{n+1}(P_t)} \le \k |A_n(t)|^{d} \\
& \hspace{15eM}\text{for all $t \in \Omega$ with $|t| > L_{A}$.}
\end{align*}
\Part{(2)} 
If $(\deg B_n(t))_{n \geq 0}$ is unbounded, then there exists an $L_{B} \geq 1$ and an
integer $N_{B} \ge 1$  such that for all $n\ge N_{B}$
\begin{align*}
& \norm{H_t^{-n}(P_t)} = |B_n(t)| \quad \text{and}\quad
 \frac{|B_n(t)|^{d}}{|\d| \,\k}  \le \norm{H_t^{-n-1}(P_t)} \le \frac{\k\,|B_n(t)|^{d}}{|\d|} \\
& \hspace{15eM}\text{for all $t \in \Omega$ with $|t| > L_{B}$.}
\end{align*}
\end{parts}
\end{Proposition}

\Proof
We start with the proof of (1). Assume that $\left(\deg A_n(t)\right)_{n\geq 0}$ is unbounded.  
By Proposition~\ref{prop:on:assumption},  
there exists an integer $N$ such that 
$\deg A_{n+1}(t) = d \deg A_{n}(t) \;(> 0)$ for all $n\ge N$. 
Let $m$ be the integer defined in \eqref{eqn:def:m}. 
Replacing $N$ by a larger number if necessary, we may assume that  
$\deg A_{N-1}(t) >  m+1$  and $\deg A_{N}(t) - \deg A_{N-1}(t) > d$. 

We claim that there exists an $L \geq 1$ such that, for all $t\in \Omega$ with $|t| > L$ 
and for all $n \geq N$, 
the following inequalities hold: 
\begin{align}
  |A_n(t)| \geq |t|^{m+1}\; & \text{and} \;   \norm{H_t^n(P_t)}   = |A_n(t)|, \label{ineq:1}\\
\frac{|A_n(t)|^{d}}{\k} & \leq |A_{n+1}(t)| \leq \k |A_n(t)|^{d}, \label{ineq:2}\\
 \left|\frac{A_{n}(t)}{A_{n+1}(t)^{d}}\right| & \leq \frac{1}{|t|^{d}} \label{ineq:3}
\end{align}
We prove the claim by induction on $n\ge N$. 

The  recursive relation~\eqref{eqn:def:An} of $A_n(t)$  gives
\[
A_{N+1}(t) = \d \,A_{N-1}(t) + A_N(t)^d + \sum_{i=1}^d \, c_{i}(t) A_{N}(t)^{d-i}.
\]
By our choice of $N$, we can find an $L^\prime \geq 1$ such that, for all  $|t| > L^\prime$,  (i) $|A_N(t)| \geq  |t|^{m+1}$ and (ii) $|A_{N-1}(t)/A_N(t)| \leq 1/|t|^{d}$. In particular, $|A_{N-1}(t)/A_{N}(t)| \leq 1$. 
Hence, for any $t \in \Omega$ with $|t| > L^\prime$, we have \eqref{ineq:1} with $n = N$. 

Take $\vp(x,y) :=\d y + f_t(x)$, $\g := \d$,  $c:=1$,  $s := \deg_x f_t(x) = d$, 
$x  := A_N(t)$ and $y  := A_{N-1}(t)$ in Lemma~\ref{lem:unbounded estimate}. 
We note that 
\[
|A_{N-1}(t)/A_N(t)^d| = |A_{N-1}(t)/A_N(t)|\, |1/A_N(t)^{d-1}| \leq |A_{N-1}(t)/A_N(t)| \leq 1/|t|^d .
\]
We take $L \geq L^\prime$ such that  $L^{m+1} \ge \sqrt[d-1]{\k}$ 
and such that~\eqref{remark:choice of L} holds. 
Then by Remark~\ref{rem:condition for L} and Lemma~\ref{lem:unbounded estimate}~(b)~(d), 
we have~\eqref{ineq:2} and ~\eqref{ineq:3} with $n= N$. 

By induction on $n$, we assume that \eqref{ineq:1}, \eqref{ineq:2} and \eqref{ineq:3} hold 
for all $t \in \Omega$ with $|t| > L$ for $n$, and we are going to deduce them for $n+1$. 

Let $t \in \Omega$ with $|t| > L$. Since $|A_n(t) \geq |t|^{m+1}  \geq L^{m+1}$ by \eqref{ineq:1}, 
it follows from 
\eqref{ineq:2} that 
\begin{equation}
|A_{n+1}(t)| \geq \left(\frac{|A_n(t)|^{d-1}}{\k}\right) |A_n(t)| 
   \geq \left(\frac{(L^{m+1})^{d-1}}{\k}\right) |A_n(t)| 
   \geq  |A_n(t)| 
   \geq |t|^{m+1} \label{ineq:4}
\end{equation}
and
\[
\left|\frac{A_n(t)}{A_{n+1}(t)}\right|  
\leq \frac{\k|A_n(t)|}{|A_n(t)|^{d}}   = \frac{\k}{|A_n(t)|^{d-1}}
\leq \frac{\k}{L^{(d-1)(m+1)}}  
\leq 1.
\]

Thus $\norm{H_t^{n+1}(P_t) }= |A_{n+1}(t)|$ and~ \eqref{ineq:1} is satisfied for $n+1.$  Note that
 \eqref{ineq:3} and \eqref{ineq:4} show that conditions (i) and (ii) in Lemma~\ref{lem:unbounded estimate} are
satisfied for $s = d$, $x(t) := A_{n+1}(t)$ and $y(t) := A_n(t)$.  Now \eqref{ineq:2} and \eqref{ineq:3}  for $ n+1$ are simply the conclusions of Remark~\ref{rem:condition for L} and Lemma~\ref{lem:unbounded estimate}~(b)~(d). 
Thus the claim is proved and by setting $N_{A} =  N$ and $L_{A} = L$, we obtain~(1). 
\smallskip

The proof of (2) is similar to that of (1), so that we only give a sketch of it. 
We choose an integer $N_{B}$ such that 
$\deg B_{n+1}(t) = d \deg B_{n}(t) \;(> 0)$ for all $n\ge N_{B}$, 
$\deg B_{N_{B} -1}(t) >  m+1$  and $\deg B_{N_{B}}(t) - \deg B_{N_{B} -1}(t) > d$. 

We take the polynomial $\vp(x,y) := (1/\d) y - (1/\d) f_t(x)$, $\g := 1/\d$, $c =  -1/\d$, $s: = d$, $x(t) = B_n(t)$ and $y(t)= B_{n-1}(t)$ in~Lemma~\ref{lem:unbounded estimate}. We take $L_{B} \geq 1$ such that, for all $t \in \Omega$ with $|t| > L_{B}$, 
we have (i) $|B_{N}(t) | \geq |t|^{m+1}$ and (ii) $|B_{N-1}(t)/B_{N}(t)^d| \leq 1/|t|^d$ and such that 
$(L_{B})^{m+1} \geq \sqrt[d-1]{|\d|\,\k}$ and \eqref{remark:choice of L} holds. 
Then by a similar argument as in (1), one can show that 
 for all $t\in \Omega$ with $|t| > L_{B}$  and for all  $n \ge N_{B}$, 
the following equalities hold: 
\begin{align}
  |B_n(t)| \geq |t|^{m+1}\; & \text{and} \;   \norm{H_t^{-n}(P_t)}   = |B_n(t)|, \label{ineq:1 for B}
  \\
\frac{|B_n(t)|^{d}}{|\d|\,\k} & \leq |B_{n+1}(t)| \leq \frac{\k |B_n(t)|^{d}}{|\d|}, \label{ineq:2 for B}\\
 \left|\frac{B_{n}(t)}{B_{n+1}(t)^{d}}\right| & \leq \frac{1}{|t|^{d}}. 
 \notag
\end{align}
Then \eqref{ineq:1 for B} and \eqref{ineq:2 for B} give the assertion (2).
\QED

As a corollary to Proposition~\ref{prop:unbounded estimate}, we have the main result of this subsection.

\begin{Proposition}
\label{prop:Mn:uniform bound}
Fix a constant $\k > 1$. 
Then, under Assumption~\ref{assumption}, there exist an $L \geq 1$ and $N \in \ZZ_{>0}$ such that
\[
\frac{M_n^{d}(t)}{\Delta\,\k} \leq  M_{n+1}(t) \leq \Delta\,\k M_n^{d}(t)
\]
holds for all $t\in \Omega$ with $|t| > L$ and all $n \ge N. $
\end{Proposition}

\Proof
{\bf Case 1.}\quad 
Suppose that both $(\deg A_n(t))_{n \geq 0}$ and $(\deg B_n(t))_{n \geq 0}$ are unbounded. 
Then we set $L := \max\{L_{A}, L_{B}\}$ and $N:= \max\{N_{A}, N_{B}\}$ in Proposition~\ref{prop:unbounded estimate}, 
and then 
the estimates in (1) and (2) in Proposition~\ref{prop:unbounded estimate} hold for $A_n(t)$ and $B_n(t)$ for all $n\ge N$ and all $|t| > L$. We note that $\Delta := \max(|\d|, 1/|\d|) \ge 1$. 
It follows from Proposition~\ref{prop:unbounded estimate}~(1)~(2) that 
 \begin{align*}
 \frac{\norm{H_t^{n}(P)}^{d}}{\Delta\,\k} \leq  &  \norm{H_t^{n+1}(P)} \leq  \Delta\,\k \norm{H_t^{n}(P)}^{d}, \\
 \frac{\norm{H_t^{-n}(P)}^{d}}{\Delta\,\k} \leq & \norm{H_t^{-n-1}(P)} \leq \Delta\,\k \norm{H_t^{-n}(P)}^{d}
 \end{align*}
 for all $n \ge N$ and all $|t| > L.$ Our assertion now follows from the definition of $M_n(t)$ (see 
 Definition~\ref{def:M:n:t}).

\smallskip
{\bf Case 2.}\quad 
Suppose that only one of the sequences $\left(\deg A_{n}(t)\right)_{n \geq 0}$ and $\left(\deg B_{n}(t)\right)_{n \geq 0}$ is unbounded. 
We will only give a proof for the case where $\left(\deg A_{n}(t)\right)_{n \geq 0}$ is unbounded. The case where $\left(\deg B_{n}(t)\right)_{n \geq 0}$ is unbounded can be proved similarly with the same choice of $\ve$  below.  
 
 We fix an integer $D$ such that $D > \deg B_n(t)$ for any $n \geq 0$ and 
 $D > \deg c_i(t)$ for any $i = 1, \ldots, d$ (cf.~\eqref{eqn:f_{t} defn}). 
 We fix an integer $N$ with $\deg A_{N}(t) \ge D$, whose existence is assured by 
 the unbounded assumption of $(\deg A_n(t))_{n \geq 0}$, and 
 such that the conclusion of Proposition~\ref{prop:unbounded estimate}~(1) 
 and equations~\eqref{ineq:1}, \eqref{ineq:2} and~\eqref{ineq:3} 
 hold with some $L \geq 1$. 
 
We fix any $0 < \ve < 1$ satisfying $\left([d+2] \k \Delta^{2}\right) \ve^{d-1} \le 1.$ Note that,
for such a $\ve$, we have $\max\{\k \ve^{d-1} , \left(\k [d+2]/|\d|\right)\ve^{d-1} \} \leq 1.$  
We replace $L$ by a larger number if necessary, we may assume that, 
for all $t \in \Omega$ with $|t| > L$, we have 
 \[
 \left| \frac{1}{A_{N}(t)}\right| \leq \ve, \quad
 \left| \frac{B_{N-1}(t)}{A_{N}(t)}\right| \leq \ve, 
 \quad  \left| \frac{B_{N}(t)}{A_{N}(t)}\right| \leq \ve 
 \quad\text{and}\quad 
 \left| \frac{c_{i}(t)}{A_{N}(t)}\right| \leq \ve\; \;\text{for any} \; i = 1, \ldots, d,
 \]
where the $c_i(t)$ are coefficients of $f_t(x)$ (see \eqref{eqn:f_{t} defn}). 
With this setting, we claim the following. 

\begin{Claim}
\label{claim:ratio of An and Bn}
For  all $t\in \Omega$ with  $|t| > L$ and all $n \ge N,$ the following inequalities hold: 
\begin{equation}
\label{eqn:ratio of An and Bn}
 \left| \frac{1}{A_{n}(t)}\right| \leq \ve, \quad
 \left| \frac{B_{n-1}(t)}{A_{n}(t)}\right| \leq \ve, \quad  
 \left| \frac{B_{n}(t)}{A_{n}(t)}\right| \leq \ve 
 \quad\text{and}\quad 
 \left| \frac{c_{i}(t)}{A_{n}(t)}\right| \leq \ve\; \;\text{for} \; i = 1, \ldots, d. 
\end{equation}
\end{Claim} 
We prove the claim by induction on integers $n\ge N.$ Write $n = N + j$, $j\ge 0$ for  $n\ge N.$ Then the claimed inequalities hold for the case $j = 0$ by our choice of $N$ and $L.$ Assume that~\eqref{eqn:ratio of An and Bn} hold for integer $j \ge 0.$ It follows from the recursive relation \eqref{eqn:def:Bn} for $B_{n}(t)$ and \eqref{ineq:2} that 
\begin{align*}
\left|\frac{B_{N+j+1}(t)}{A_{N+j+1}(t)}\right| 
& \le \frac{ \k [d+2]}{|\d|} \max\left(\frac{|B_{N+j-1}(t)|}{|A_{N+j}(t)|^{d}}, 
\frac{|B_{N+j}(t)|^d}{|A_{N+j}(t)|^d},
\ldots, 
\frac{|c_{i}(t)B_{N+j}^{d-i}(t)|}{|A_{N+j}(t)|^{d}}, 
\ldots ,  
\frac{|c_{d}(t)|}{|A_{N+j}(t)|^{d}}\right) 
\\
&  \leq \frac{ \k [d+2]}{|\d|}  \ve^{d} \hspace{10eM} \text{(by induction hypothesis)},\\
& \leq \ve  \hspace{14eM} \text{(since $\left(\k [d+2]/|\d|\right) \ve^{d-1} \leq 1$)}.
\end{align*}
 The induction hypothesis together with \eqref{ineq:2} also ensure that 
 \begin{align*}
& \left| \frac{1}{A_{N+j+1}(t)}\right| 
 \leq \frac{\kappa}{|A_{N+j}(t)|^d} 
 \leq \kappa \ve^d \leq (\kappa \ve^{d-1}) \ve \leq \ve, 
 \\
& \left| \frac{B_{N+j}(t)}{A_{N+j+1}(t)}\right| 
\leq \frac{\kappa |B_{N+j}(t)|}{|A_{N+j}(t)|^d} 
\leq  (\kappa \ve^{d-1}) \ve \leq \ve,
\\
& \left| \frac{c_{i}(t)}{A_{N+j+1}(t)}\right| 
\leq \frac{\kappa |c_i(t)|}{|A_{N+j}(t)|^d} \leq
  (\kappa \ve^{d-1}) \ve \leq \ve \quad \text{for} \; i = 1, \ldots, d. 
\end{align*}
This finishes the induction steps and hence the claim is proved. 
 
 Claim~\ref{claim:ratio of An and Bn} and 
 Proposition~\ref{prop:unbounded estimate}~(1) imply that $M_{n}(t) = |A_{n}(t)|$ for all $t\in \Omega$ with 
 $|t| > L$ and all $n \ge N.$ Noting that $\Delta := \max\{|\delta|, |1/\delta|\} \geq 1$ and 
 applying Proposition~\ref{prop:unbounded estimate}~(1),  we have 
\[
\frac{M_n^{d}(t)}{\Delta\,\k} \leq  M_{n+1}(t) \leq \Delta\,\k M_n^{d}(t). 
\]
(We note that $\Delta$ is needed when we treat the case where $\left(\deg B_{n}(t)\right)_{n \geq 0}$ is unbounded.)
\QED

\subsection{Proof of Theorem~\ref{thm:main:Mn}}

Combining the previous two subsections, we now give a proof for our main result in this section.

\medskip

\begin{proof}[Proof of Theorem~\ref{thm:main:Mn}] \quad
Recall that 
\[
\overline{\Phi}_n  \colon  \PP^1  \to \PP^4\quad \text{is given by} \quad 
 (t:1)  \mapsto (A_n(t): A_{n-1}(t): B_{n-1}(t): B_n(t): 1)
\]
and that  $\ell_n = \deg(\overline{\Phi}_n)$. 
By Proposition~\ref{prop:on:assumption}, there exist a positive rational number $\ell$ and a positive integer $N$ such that 
$
\ell_n = d^{n} \ell
$
for all $n \geq N$. We fix such an~$N$. 

For $n \geq N$, we take the lift $\Aff^2 \to \Aff^5$ of $\overline{\Phi}_n$ given by 
$(t, 1) \mapsto (A_n(t), A_{n-1}(t), B_{n-1}(t), B_n(t), 1)$. Then the lift gives rise to 
an isomorphism $\alpha_n\colon \overline{\Phi}_n^*(\OO_{\PP^4}(1)) \overset{\sim}{\longrightarrow}
\OO_{\PP^1}(\ell_n)$. Let $\Vert \cdot \Vert_n$ be the metric 
on $\OO_{\PP^1}(1)$ defined by $\Vert \cdot \Vert_n
:= (\alpha_n(\overline{\Phi}_n^*(\Vert\cdot\Vert_{\mathrm{st}})))^{1/\ell_n}$. 

Our goal is to show that
$
  \log (\Vert\cdot\Vert_{n+1}/\Vert\cdot\Vert_{n})
$
converges uniformly on $\PP^1(\Omega)$ as $n$ tends to~$\infty$. 

For $(t^\prime: t^{\prime\prime}) \in \PP^1(\Omega)$, we have
\[
\log\frac{\Vert\cdot\Vert_{n+1}}{\Vert\cdot\Vert_{n}}(t^\prime: t^{\prime\prime}) = \log \frac{\norm{\Phibar_{n}(t^\prime:t^{\prime\prime})}^{1/\ell_{n}}}
{\norm{\Phibar_{n+1}(t^\prime:t^{\prime\prime})}^{1/\ell_{n+1}}}  
= \begin{cases} 
\log {\displaystyle \frac{M_{n}\left(t^\prime/t^{\prime\prime}\right)^{1/\ell_{n}}}{M_{n+1}\left(t^\prime/
t^{\prime\prime}\right)^{1/\ell_{n+1}}}}  & \text{(if $t^{\prime\prime}\ne 0$),} \\  
\smallskip
\log {\displaystyle \frac{\norm{\Phibar_{n}(1:0)}^{1/\ell_{n}}}{\norm{\Phibar_{n+1}(1:0)}^{1/\ell_{n+1}}}}
    & \text{(if $t^{\prime\prime} =0$), }
  \end{cases}
\]
see \eqref{eqn:def:norm:n} and Definition~\ref{def:M:n:t}.

We set $\k := \max\{ [3d]\D,  2\}.$ 
Let $m$ be as in \eqref{eqn:def:m}. 
Choose $L \geq 1$ such that $L^{m+1} \geq \sqrt[d-1]{\k\D}$   and
that $\norm{P_{t}} \leq L^{m+1} $  for  $t \in \Bcal_{L}.$ 
We replace $L$ and $N$ by lager numbers if necessary, 
so that Lemma~\ref{lem:estimate:Mn:upper}, Proposition~\ref{prop:estimate:Mn:lower} and Proposition~\ref{prop:Mn:uniform bound} hold. 

We claim that, for any $ n \ge N$ and for any $(t^\prime: t^{\prime\prime}) \in \PP^1(\Omega)$, we have
\begin{equation}
\label{eqn:claim:uniform:conv}
\left|\log\frac{\Vert\cdot\Vert_{n+1}}{\Vert\cdot\Vert_{n}}(t^\prime: t^{\prime\prime})\right|
\leq \frac{1}{\ell d^{n}} \log (\k L^{m+1}). 
\end{equation}

\smallskip
{\bf Case 1.}\quad 
Suppose that $t^{\prime\prime} \neq 0$ and $t :=t^\prime/t^{\prime\prime}$ is in the bounded region 
$\mathcal{B}_{L}$. Then, for all $n \geq N$, we have 
\begin{align*}
\frac{M_n(t)^d}{(\k L^{m+1})^d} 
& \le \frac{M_n(t)^d}{([3d]\D L^{m+1})^d}  \\
& \le M_{n+1}(t)  \hspace{15eM} (\text{by Proposition~\ref{prop:estimate:Mn:lower}})\\
& \le [d+2]\D L^{m+1} M_n(t)^d \hspace{11.3eM} (\text{by Lemma~\ref{lem:estimate:Mn:upper}})  \\
& \leq [3d] \Delta L^{m+1} M_n(t)^d 
\le (\k L^{m+1}) M_n(t)^d 
\le (\k L^{m+1})^d M_n(t)^d.  
\end{align*}
Thus
\begin{align*}
\left|\log\frac{\Vert\cdot\Vert_{n+1}}{\Vert\cdot\Vert_{n}}(t:1)\right|
& = \left|\log\frac{M_{n}(t)^{1/\ell_{n}}}{M_{n+1}(t)^{1/\ell_{n+1}}} \right|\\
& = \frac{1}{\ell d^{n+1}}\left| \log M_{n+1}(t) - \log M_{n}(t)^d\right|
\leq  \frac{1}{\ell d^{n}} \log(\k L^{m+1}).
\end{align*}
The claim now follows in this case. 

\smallskip
{\bf Case 2.}\quad 
Suppose that 
$t^{\prime\prime} \neq 0$ and $t := t^\prime/t^{\prime\prime}$ is outside $\mathcal{B}_{L}$.  
By Proposition~\ref{prop:Mn:uniform bound}, for all $n \geq N$, we then have 
\begin{align*}
\left|\log\frac{\Vert\cdot\Vert_{n+1}}{\Vert\cdot\Vert_{n}}(t:1)\right|
& = \frac{1}{\ell d^{n+1}}\left| \log M_{n+1}(t) - \log M_{n}(t)^d\right| \\
& \leq  \frac{1}{\ell d^{n+1}} \log(\k\D) \leq  \frac{1}{\ell d^{n}} \log(\k L^{m+1}), 
\end{align*}
where in the last inequality we use 
$(\kappa L^{m+1})^d = (L^{m+1})^{d-1} (\kappa^d L^{m+1}) \geq (L^{m+1})^{d-1}  \geq \kappa \Delta$. 
We  obtain the claim in these case as well. 

\smallskip
{\bf Case 3.}\quad 
Suppose that $t^{\prime\prime} = 0$. In this case, $(t^\prime: t^{\prime\prime}) = (1:0)$.

\smallskip
{\sl Subcase 3-1.}\quad
Suppose that $(\deg(A_n(t)))_{n \geq 0}$ is unbounded, but 
 $(\deg(B_n(t)))_{n \geq 0}$ is bounded. 
It follows from~\eqref{eq:An leading coefficient}
that for all $n \ge N$, 
\[
\log \norm{\Phibar_n(1, 0)}^{1/\ell_n} 
= 
\frac{1}{d^n \ell} \log |\alpha_D|^{d^{n-N}}  
= \frac{1}{d^{N} \ell} \log |\alpha_D|  
= \norm{\Phibar_{n+1}(1, 0)}^{1/\ell_{n+1}}. 
\]
Thus, for all $n\ge N$, we have 
\begin{equation}
\label{eqn:case:3}
\left|\log\frac{\Vert\cdot\Vert_{n+1}}{\Vert\cdot\Vert_{n}}(1:0)\right|  
= \log \frac{\norm{\Phibar_{n}(1, 0)}^{1/\ell_{n}}}
{\norm{\Phibar_{n+1}(1, 0)}^{1/\ell_{n+1}}} = 0. 
\end{equation}

\smallskip
{\sl Subcase 3-2.}\quad
Suppose that $(\deg(B_n(t)))_{n \geq 0}$ is unbounded, but 
$(\deg(A_n(t)))_{n \geq 0}$ is bounded. 
Then using \eqref{eq:Bn leading coefficient}, we similarly 
have \eqref{eqn:case:3}. 

\smallskip
{\sl Subcase 3-3.}\quad
Suppose that both $(\deg(A_n(t)))_{n \geq 0}$ and 
$(\deg(B_n(t)))_{n \geq 0}$ are unbounded. 
We may take $N = N^\prime$ in \eqref{eq:An leading coefficient} and 
\eqref{eq:Bn leading coefficient}. 
If $D > D'$, then as in Subcase 3-1, we have  \eqref{eqn:case:3}. 
If $D < D'$, then as in Subcase 3-2, we have  \eqref{eqn:case:3}. 
If $D = D'$, then we have, for any $n \geq N$, 
\begin{align*}
\log \norm{\Phibar_n(1, 0)}^{1/\ell_n} 
& = 
\frac{1}{d^n \ell} \log \max\{ |\alpha_D|^{d^{n-N}},  |\beta_D|^{d^{n-N}}  \}\\
& = \frac{1}{d^{N} \ell} \log \max\{|\alpha_D|,    |\beta_D|\} 
= \norm{\Phibar_{n+1}(1, 0)}^{1/\ell_{n+1}}, 
\end{align*} 
and again we have  \eqref{eqn:case:3}. 

\smallskip
We have obtained the claim. 
By \eqref{eqn:claim:uniform:conv}, 
$\{\log \Vert\cdot\Vert_{n}\}_{n \geq N}$ converges uniformly on
the parameter space $\PP^1(\Omega)$. 
This completes the proof of Theorem~\ref{thm:main:Mn}.
\end{proof}

\setcounter{equation}{0}
\section{Height associated to a semipositive adelically metrized line bundle on a parameter space}
\label{sect:heights and metrized line bundles}

Let $K$ be number field or a function field $K = F(B)$ of an integral projective variety $B$ over a field $F$ that is regular in codimension one. 
Let the family of H\'enon maps $\bfH = (H_t)_{t \in \overline{K}}$ and the family of initial points 
$\bfP = (a(t), b(t)) = (P_t)_{t \in\overline{K}}\in \Aff^2(K[t])$ be as in \S~\ref{sec:family of Henon}. 
By Theorem~\ref{thm:main:Mn},  associated to $\bfH$ and $\bfP$, there exists 
a sequence of adelic metrics $\{\norm{\cdot}_{n, v}\}_{v\in M_{K}}$ of $\Ocal_{\PP^{1}}(1)$ such that 
$\log \norm{\cdot}_{n, v}$ converges uniformly to $\log \norm{\cdot}_{v}$ as $n \to \infty$ for every $v\in M_{K}.$

\begin{Proposition}
\label{prop:line bundle LP}
The pair $\Lcalbar_{\bfP} = (\Ocal_{\PP^1}(1),\{\norm{\cdot}_v\}_{v\in M_K})$ is a semipositive adelically metrized line bundle
over $\PP^1$ over $K$. 
\end{Proposition}

\Proof
For brevity, in what follows, we will write $B$ for both $\Spec(O_K)$ if $K$  is a number field and   
the underlying projective variety if $K = F(B)$ is a function field over another field $F.$
As in \S~\ref{subsec:function field henon}, 
we identify $M^{\rm fin}_K$ with $\Spec(O_K) \setminus \{(0)\}$ 
if $K$ is a number field, and with the set of all generic points of prime divisors of $B$ if 
$K = F(B)$.

Let $U$ be a non-empty Zariski open subset of $B$ with the following properties:
\begin{itemize}
\item
We regard $U$ as a subset of $M^{\rm fin}_K$. 
(When $K = F(B)$, a place $v \in M^{\rm fin}_K$ belongs to $U$ if 
the prime divisor corresponding to $v$ intersects with $U$.)
Then for each $v \in U$, 
all the coefficients of $c_i(t)$ ($i = 1, 2, \ldots, d$) have 
absolute value $1$ with respect to $|\cdot|_v$ (cf. \eqref{eqn:f_{t} defn}); 
\item
For each $v \in U$, 
all the coefficients of $a(t)$ and $b(t)$ have 
absolute value $1$ with respect to $|\cdot|_v$; 
\item 
For each $v \in U$, 
if $(\deg A_n(t))_{n \geq 0}$ is unbounded, then the leading coefficients 
of $A_{n}(t)$ has absolute value $1$ with respect to~$|\cdot|_v$ for sufficiently large $n$. 
For each $v \in U$, 
if $(\deg B_n(t))_{n \geq 0}$ is unbounded, then then the leading coefficients 
 of $B_{n}(t)$ has absolute value $1$ with respect to~$|\cdot|_v$ for sufficiently large $n$. 
(See Remark~\ref{rem:alphaD-betaD}.)
\item
For any sufficiently large $n$, the non-constant morphism $\Phibar_n\colon   \PP^1 \to \PP^4$ over $K$ extends to a morphism, denoted by $\Phibar_{n, U},$
over $U$.
\end{itemize}
Then $M_K \setminus U$ is a finite set of places. 
In the following, we assume that $n$ is sufficiently large 
such that  $\Phibar_n\colon   \PP^1 \to \PP^4$  is a non-constant morphism defined over $K$.  
Then the morphism $\Phibar_n$ induces a rational map, still denoted by the same $\Phibar_n$, over $B$.  
\begin{equation}
\label{eqn:Phibar:revisited:2}
  \Phibar_n\colon 
  \PP^1_{B} \dasharrow \PP^4_{B}. 
\end{equation}

Let $\PP_B^{1\; \prime}$ be the normalization of the map 
$\PP^1_{U} \overset{\Phibar_{\small{n,U}}}{\rightarrow} \PP^4_{U} \hookrightarrow \PP^4_{B}$ (cf. \cite[II.~Ex.~3.8]{Ha}). 
Then we have a morphism $\widetilde{\Phi}_n\colon \PP_B^{1\; \prime} \to \PP^4_{B}$. We set $\mathscr{L}_n :={\widetilde{\Phi}_n}^*(\OO_{\PP^4_B}(1)).$ 
Note that, since $\OO_{\PP^4_B}(1)$ is relatively ample and $\widetilde{\Phi}_n$ is the normalization, $\mathscr{L}_n$ is relatively ample, and in particular relatively nef. 

For $v \in M_K^{\rm fin}$, we endow $\OO_{\PP^1}(1) \otimes_{B} \KK_v$ with the metric 
$\Vert\cdot\Vert_{n, v}^{\prime}$ induced from 
the model $(\PP_B^{1\; \prime}, \mathscr{L}_n)$. For $v \in  M_K^\infty$, we endow 
$\OO_{\PP^1}(1) \otimes_{B} \KK_v$ with the metric $\Vert\cdot\Vert_{n, v}^{\prime}$ 
by pulling back the Fubini-Study metric $\Vert\cdot\Vert_{{\rm FS}, v}$ on $\OO_{\PP^4}(1)$. 

We show that $(\OO_{\PP^1}(1), \{\Vert\cdot\Vert_{n, v}^{\prime}\})$ 
satisfies the conditions (a)(b)(c) in \S~\ref{subsec:function field henon} and 
converges to $\Lcalbar_{\bfP} := (\Ocal_{\PP^1}(1),\{\norm{\cdot}_v\}_{v\in M_K})$. Then by definition, 
 $\Lcalbar_{\bfP}$ is a semipositive adelically metrized line bundle. 

We check (a). Let $v \in M_K^{\rm fin}$. Note that the standard metric $\Vert \cdot \Vert_{{\rm st}, v}$ is 
the metric induced from the model $\OO_{\PP^4_B}(1)$. 
Then $\Vert\cdot\Vert_{n, v}^{\prime}$ is the 
pull-back  of $\Vert \cdot \Vert_{{\rm st}, v}$, and the convergence 
$\log (\Vert\cdot\Vert_{n, v}^{\prime}/{\Vert\cdot\Vert_{v}})$ follows from Theorem~\ref{thm:main:Mn}. 
Let $v \in M_K^\infty$. Since 
\[
(1/5)\Vert\cdot\Vert_{{\rm st}, v}^2 \leq \Vert\cdot\Vert_{{\rm FS}, v}^2 \leq  \Vert\cdot\Vert^2_{{\rm st}, v}
\]
and $5^{1/\ell_n} \to 1$ as $n \to +\infty$, the convergence of 
$\log (\Vert\cdot\Vert_{n, v}^{\prime}/\Vert\cdot\Vert_{v})$ 
also follows from Theorem~\ref{thm:main:Mn}. Thus we have checked (a). 

Condition (b) is obvious, because the Fubini-study metric is smooth and has positive curvature. 
Condition (c) follows from the definition of $(\OO_{\PP^1}(1), \{\Vert\cdot\Vert_{n, v}^{\prime}\})$ 
where  $e_n := \ell_n$ (see \eqref{eqn:defn of elln}).

It remains to show that $(\OO_{\PP^1}(1), \{\Vert\cdot\Vert_{n, v}^{\prime}\})$ converges to $\Lcalbar_{\bfP}$. 
Since we have shown that $\log (\Vert\cdot\Vert_{n, v}^{\prime}/\Vert\cdot\Vert_{v})$ uniformly converges to $0$, 
it suffices to show that $\Vert\cdot\Vert_{n, v}^{\prime} = \Vert\cdot\Vert_v$ for all sufficiently large  
$n$ and for all $v \in U$. 
Let $v \in U$, and we take any $(a:1) \in \PP^1(\KK_v)$. Note that 
all the coefficients of $A_n(t), B_{n}(t)$ have absolute value at most $1$ with respect to $|\cdot|_v$, 
and that the leading coefficients of $A_n(t)$ (or $B_n(t)$) have absolute value $1$
 if the degrees of $A_n(t)$ (respectively $B_n(t)$)  are unbounded as $n$ runs through all positive integers. 
Then, we have
\[
  \Vert (A_n(a), A_{n-1}(a), B_{n-1}(a), B_n(a), 1) \Vert_v 
  = 
  \begin{cases}
  1 & \quad \text{(if $|a|_v \leq 1$)}, \\
   |a|_v^{\ell_n} & \quad \text{(if $|a|_v > 1$)}. 
  \end{cases}
\]
It follows that $\Vert\cdot\Vert_{n, v}^{\prime}$ coincides with the standard metric $\Vert\cdot\Vert_{{\rm st}, v}$ 
on $\OO_{\PP^1}(1)$, and thus $\Vert\cdot\Vert_{n, v}^{\prime} = \Vert\cdot\Vert_{v} \, (= \Vert\cdot\Vert_{{\rm st}, v})$ 
for all $v \in U$ and any sufficiently large $n$. 
\QED

\begin{Definition}
We denote by $h_{\Lcalbar_\bfP}\colon \PP^1(\overline{K}) \to \RR_{\geq 0}$ the height function associated to $\Lcalbar_{\bfP}$.
\end{Definition}

Before we prove Theorem~\ref{thm:main:A}, we introduce the function 
$G_{\bfP, v}\colon \Aff^1(\KK_v) \to \RR_{\geq 0}$ defined by 
\begin{align}
\label{eqn:def:G}
G_{\bfP, v}(t) := G_v(P_t) =  \max\left\{G^{+}_v(P_{t}), G^{-}_v(P_{t})\right\} \quad \text{for $v\in M_{K}$,}
\end{align}
where $G_v(P_{t})$ is the $v$-adic Green function as in~\eqref{eqn:def:G:intro}.
Recall from Definition~\ref{def:M:n:t} that $M_n(t) = \Vert (H^n_t(P_t), H^{-n}_t(P_t), 1) \Vert_v$. 

\begin{Lemma}
\label{lem:thm:main:Mn}
We have 
$\displaystyle G_{\bfP, v}(t) = \lim_{n\to \infty} \frac{1}{d^n} \log M_{n}(t)$ 
for all $t \in \Aff^1(\KK_v)$. 
\end{Lemma}

\Proof
In general, suppose that $\{a_n\}_{n \geq 1}$ and $\{b_n\}_{n \geq 1}$ are convergent sequences of real numbers, and we write  $\lim_{n\to\infty} a_n = \alpha$ and $\lim_{n\to\infty} b_n = \beta$. We set $c_n = \max\{a_n, b_n\}$. Then we see that $\{c_n\}_{n \geq 1}$ is also a convergent sequence, and that 
$\lim_{n\to\infty} c_n = \max\{\alpha, \beta\}$. 

We apply this observation to $a_n = (\log^+ \norm{H_t^n(P_t)})/d^n$ and 
$b_n = (\log^+ \norm{H_t^{-n}(P_t)})/d^n$. Then we obtain the assertion. 
\QED

Now we prove Theorem~\ref{thm:main:A} by showing 
the following Theorem. 

\begin{Theorem}
\label{thm:main2:revisited}
With the notation and assumption as in Theorem~\ref{thm:main:A}, 
we have  $h_{\bfP} = h_{\Lcalbar_\bfP}$. 
\end{Theorem}

\Proof
For any  $Q = (t:1) \in \PP^1(\overline{K})$ in the parameter space, 
we take a nonzero rational section $\eta$ of $\OO_{\PP^1}(1)$ 
whose  support is disjoint from the Galois conjugates of $Q$ over $K$. 

Let $v \in M_K$. Setting $\Omega = \KK_v$ in Section~\ref{sec:valued fields}, 
we denote $M_n$ in Definition~\ref{def:M:n:t} by $M_{n, v}$. 
Further, for each place $v\in M_{K}$ we fix an embedding 
$\overline{K} \hookrightarrow \KK_{v}.$

Then for any embedding $\sigma: K(Q) \to \overline{K}$, 
it follows from \eqref{eqn:elln}, \eqref{eqn:def:norm:n}, Definition~\ref{def:M:n:t} 
and Lemma~\ref{lem:thm:main:Mn} that 
\begin{align*}
\log \norm{\eta(Q^\sigma)}_v & =  \log |\eta(Q^\sigma)|_v - \lim_{n\to \infty}\frac{1}{\ell_{n}}  \log M_{n, v}(t^\sigma)\\
& = \log |\eta(Q^\sigma)|_v - \lim_{n\to \infty}\frac{1}{\ell \, d^{n}}  \log M_{n, v}(t^\sigma)\\
& = \log |\eta(Q^\sigma)|_v  - \frac{G_{\bfP, v}(t^\sigma)}{\ell}. 
\end{align*}

Then, we have 
\begin{align}
\label{eqn:thm:main2:revisited:1}
h_{\Lcalbar_\bfP}(Q) & = \frac{1}{[K(Q):K]}\sum_{\sigma:K(Q)\to \overline{K}}
\sum_{v\in M_K}- n_v \,\left(\log |\eta(Q^{\s})|_v  - \frac{G_{\bfP, v}(t^{\s})}{\ell}\right) \\
\notag
&  = \frac{1}{\ell}\,\left(\frac{1}{[K(t):K]} \sum_{v\in M_K} \sum_{\sigma:K(t)\to \KK_{v}} \, n_{v} G_{\bfP, v}(t^{\s})\right)
= \frac{1}{\ell}\, \htilde_{H_t}(P_{t}). 
\end{align}

Since $\bfP \in \Aff^2(K[t])$, 
Proposition~\ref{prop:on:assumption}~(4) gives 
\begin{equation}
\label{eqn:can:height:tilde:2}
\htilde_{\bfH}(\bfP) = \lim_{n\to\infty}\, \frac{\deg \overline{\Phi}_{n}}{d^{n}} = \ell.  
\end{equation}
We obtain
\begin{equation}
\label{eqn:thm:main2:revisited:2}
h_{\bfP}(t) := \frac{\htilde_{H_t}(P_{t})}{\htilde_{\bfH}(\bfP)} =  \frac{\htilde_{H_t}(P_{t})}{\ell}.
\end{equation}
Comparing \eqref{eqn:thm:main2:revisited:1} and \eqref{eqn:thm:main2:revisited:2}, 
we obtain the assertion. 
\QED

\setcounter{equation}{0}
\section{Local properties of parameter space height $h_{\bfP}$: a first application} 
\label{sec:local:parameters:height}

Let $K$ be a product formula field. 
Let  $\bfH = (H_t)_{t \in \overline{K}}$ be a family of H\'enon maps and 
let $\bfP = (a(t), b(t)) = (P_t)_{t \in\overline{K}}$ be a family of initial points 
satisfying Assumption~\ref{assumption}  as in \S~\ref{sec:family of Henon}. 
Recall from Introduction that 
\begin{equation}
\label{eqn:Sigma:bfP}
\Sigma(\bfP) := \{t \in \Aff^1(\overline{K}) \mid \;\text{$P_t$ is periodic with respect to $H_t$}\}.
\end{equation}
In this section, we establish some properties of $h_{\bfP}$ on the parameter space. 

\begin{Proposition}
\label{prop:periodic parameter}
Suppose that  $K$ is a number field. 
Then, for $t \in \Aff^1(\overline{K})$, 
 $P_{t}$ is periodic with respect to $H_t$ if and only if $h_\bfP (t) = 0.$ 
In other words, we have
\[
\Sigma(\bfP) = 
\{
t \in \Aff^1(\overline{K}) \mid h_\bfP (t) = 0
\}. 
\]
\end{Proposition}

\Proof
Since $h_{\bfP}(t) := \htilde_{H_t}(P_{t}) / \htilde_{\bfP}(t)$ by definition, 
$h_\bfP (t) = 0$ is equivalent to $\htilde_{H_t}(P_{t}) = 0$. 
The assertion follows from Proposition~\ref{prop:northcott}~(2). 
\QED

Next, we decompose the parameter space height $h_{\bfP}$ into the sum of 
$v$-adic functions. Note that $h_{\bfP}(t) := \widetilde{h}_{H_t}(P_t)/ \widetilde{h}_{\bfH}(\bfP)$ 
is defined over any product formula field $K$ and that the function $G_{\bfP, v}$ 
in~\eqref{eqn:def:G} is defined for any place $v \in M_K.$

\begin{Proposition}
\label{prop:decomposition}
Let $K$ be a product formula field. 
Then for any $t \in \Aff^1(\overline{K})$, we have
\begin{equation}
\label{eqn:can:height:tilde:3}
  h_{\bfP}(t) = 
  \frac{1}{[K(t):K]}
  \sum_{v \in M_K} 
  \sum_{\sigma\colon K(t) \to \KK_v} n_v 
  \frac{G_{\bfP, v}(t^\sigma)}{\ell}. 
\end{equation}
\end{Proposition}

\Proof
As we saw in the proof of Theorem~\ref{thm:main2:revisited},  
the definition of $\widetilde{h}_{H_t}(P_t)$ gives 
\[
  \widetilde{h}_{H_t}(P_t) = 
  \frac{1}{[K(t):K]}
  \sum_{v \in M_K} 
   \sum_{\sigma\colon K(t) \to \KK_v} n_v G_{\bfP, v}(t^\sigma). 
\]
Also we have $h_{\bfP}(t) = \widetilde{h}_{H_t}(P_t)/\ell$ by \eqref{eqn:thm:main2:revisited:2}, which 
holds for any product formula field. 
Thus we obtain the assertion. 
\QED

As an application of $h_{\bfP}$, we show that 
the set of parameter values $t$ where $\{\Vert H_{t}^n(P_t)\Vert_{v}\}_{n \in \ZZ}$ is bounded 
is characterized as the zero set of $G_{\bfP, v}$. 

Let $\bfH$ and $\bfP$ be as in Theorem~\ref{thm:main:A}. 
In particular, we assume that $\bfP$ satisfies Assumption~\ref{assumption}. 
Let $v \in M_K$. 
We set 
\begin{align*}
W_{\bfP, v}  &:= \{t \in \Aff^1(\KK_v) \mid \text{$\Vert 
(H_t^n(P_t), H_t^{-n}(P_t)) \Vert_v \to +\infty$ as $n \to +\infty$}
\}, \\
K_{\bfP, v}  &:=
 \{t \in \Aff^1(\KK_v) \mid \text{$\{\Vert H_t^n(P_t)\Vert_{v}\}_{n \in \ZZ}$ is bounded}\}. 
\end{align*}

\begin{Theorem}
\label{thm:K:P:v:characterization}
We have $\Aff^1(\KK_v) = W_{\bfP, v} \,\amalg\, K_{\bfP, v}$ (disjoint union), and 
$K_{\bfP, v}$ is exactly the set of points where $G_{\bfP, v}$ vanish:
\[
  K_{\bfP, v}  = \{t \in \Aff^1(\KK_v) \mid G_{\bfP, v}(t) = 0\}. 
\]
\end{Theorem}

\Proof
For the first assertion, let $t \in \Aff^1(\KK_v)$ and suppose that 
$\{\Vert H_t^n(P_t)\Vert_{v}\}_{n \in \ZZ}$ is not bounded. 
Then $\{\Vert H_t^n(P_t)\Vert_{v}\}_{n \geq 0}$ is not bounded, or 
$\{\Vert H_t^n(P_t)\Vert_{v}\}_{n \leq 0}$ is not bounded. 
In the former case, it follows from \cite[Theorem~3.1~(2)]{Ka} that 
$\lim_{n \to +\infty} \Vert  H_t^n(P_t) \Vert_v = +\infty$. In the latter case, 
we similarly have $\lim_{n \to -\infty} \Vert  H_t^n(P_t) \Vert_v = +\infty$ (see \cite[p.~1237]{Ka}). 
Thus $\Aff^1(\KK_v) = W_{\bfP, v} \,\amalg\, K_{\bfP, v}$ (disjoint union). 

We show the second assertion. 
Let $t \in \Aff^1(\KK_v)$. 
Suppose that $t \in K_{\bfP, v}$. This means that  
$\left\{\left\Vert A_n(t), A_{n-1}(t), B_{n-1}(t), B_n(t), 1 \right\Vert_v\right\}_{n \geq 1}$ is 
bounded, because $H_t^n(P_t) = (A_n(t), A_{n-1}(t))$ and $H_t^{-n}(P_t) = (B_{n-1}(t), B_{n}(t))$ for 
any $n \geq 1$. Then Lemma~\ref{lem:thm:main:Mn} 
gives $G_{\bfP, v}(t) = 0$. 

Next suppose that $t \not\in K_{\bfP, v}$. 
By \cite[Theorem~3.1~(1) and p.~1237]{Ka} and Lemma~\ref{lem:thm:main:Mn}, 
we then have $G_{\bfP, v}(t) > 0$.  
This completes the proof. 
\QED

\begin{Example}
\label{eg:figures:inou}
Let $\bfH = \{H_t\}_t$ with $H_t(x, y) = (y + x^2 + t, x)$ be as 
in Example~\ref{eg:quad:Henon:map} where 
$\bfH$ is regarded as a H\'enon map defined over $\QQ(t)$. 
We start from a constant initial point $\bfP = (a, b)$ with 
$a, b \in \QQ$, and we consider the archimedean place $v_\infty \in M_\QQ$. 

The following figures 
(Figure~2 and Figure~3), drawn with Qfract (by Hiroyuki Inou), 
are for $\bfP = (0, 1/2)$  (centered around $t=0.1$) and for $\bfP = (1, -1)$ (centered around $t=-1$). 
Shades depend on how fast the orbit $\{H^n_t((a, b))\}$ escapes as $|n|$ becomes large (the escaping rate is very small in the black region). See also the figure in Introduction for $\bfP = (0, 0)$. 

\begin{figure}[htb!]
  \includegraphics[width=7.0cm, bb=0 0 640 640]{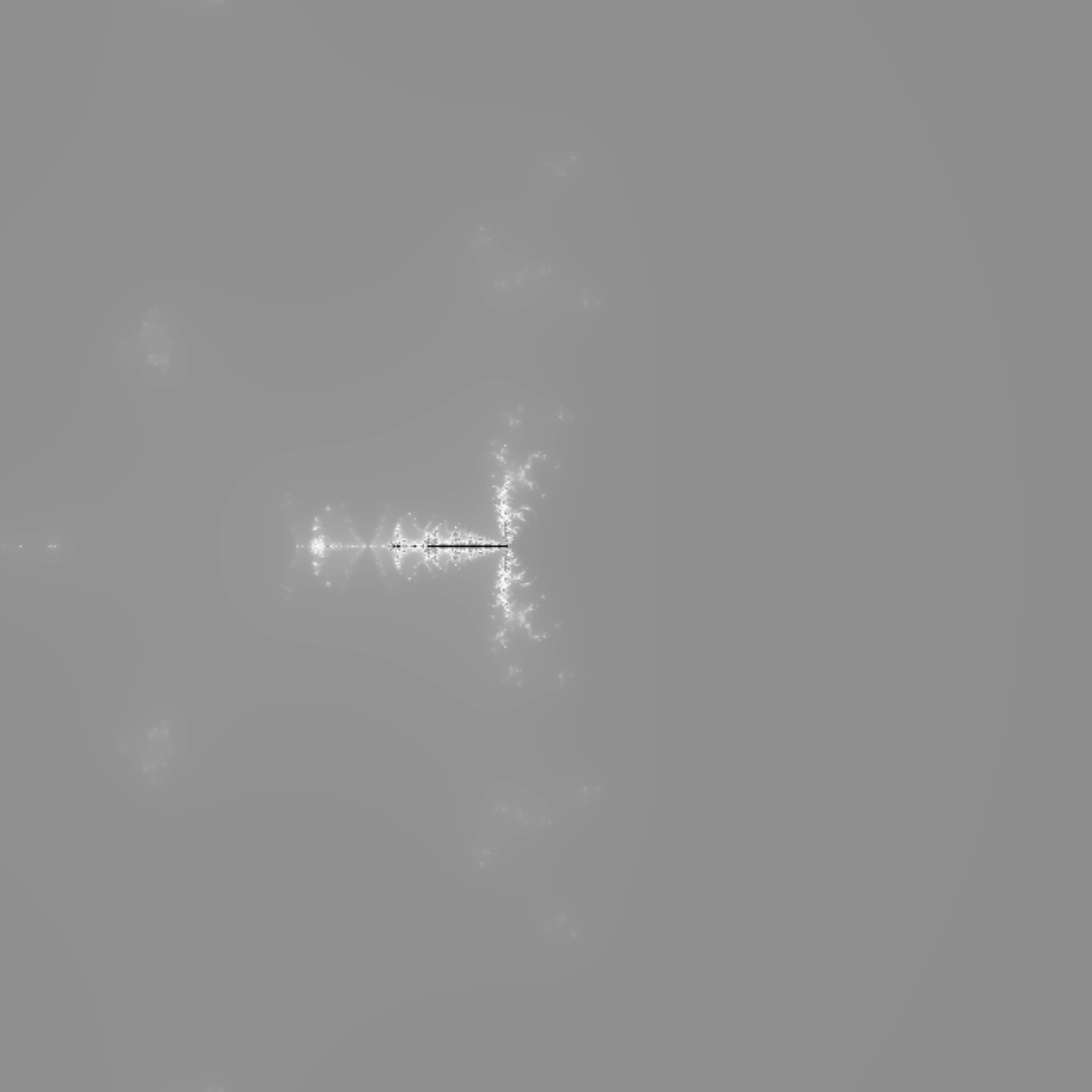}\hspace{2eM}
  \includegraphics[width=7.0cm, bb=0 0 640 640]{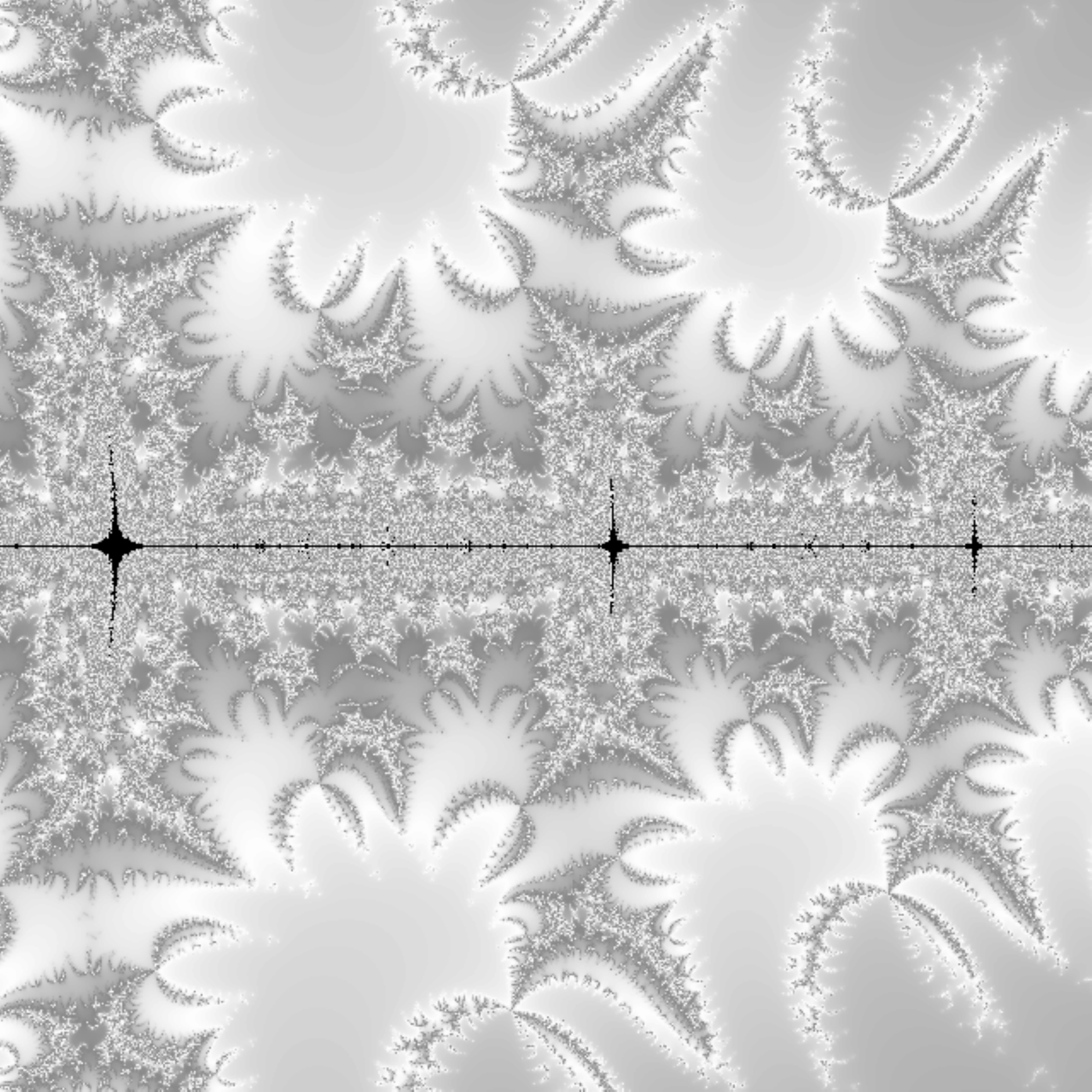}
  \caption{Figures for $\bfP = (0, 1/2)$. The region on the left is 
 $\{t \in \CC \mid |{\rm Re}(t) -0.1| \leq 1, \, 
   |{\rm Im}(t)| \leq 1\}$, 
  and that on the right is $\{t \in \CC \mid |{\rm Re}(t) -0.1| \leq 0.01, \, |{\rm Im}(t)| \leq 0.01\}$. 
  Proposition~\ref{prop:intro:finiteness} says that $\Sigma((0, 1/2)) = \emptyset$.
  }
\end{figure}

\begin{figure}[htb!]
  \includegraphics[width=7.0cm, bb=0 0 640 640]{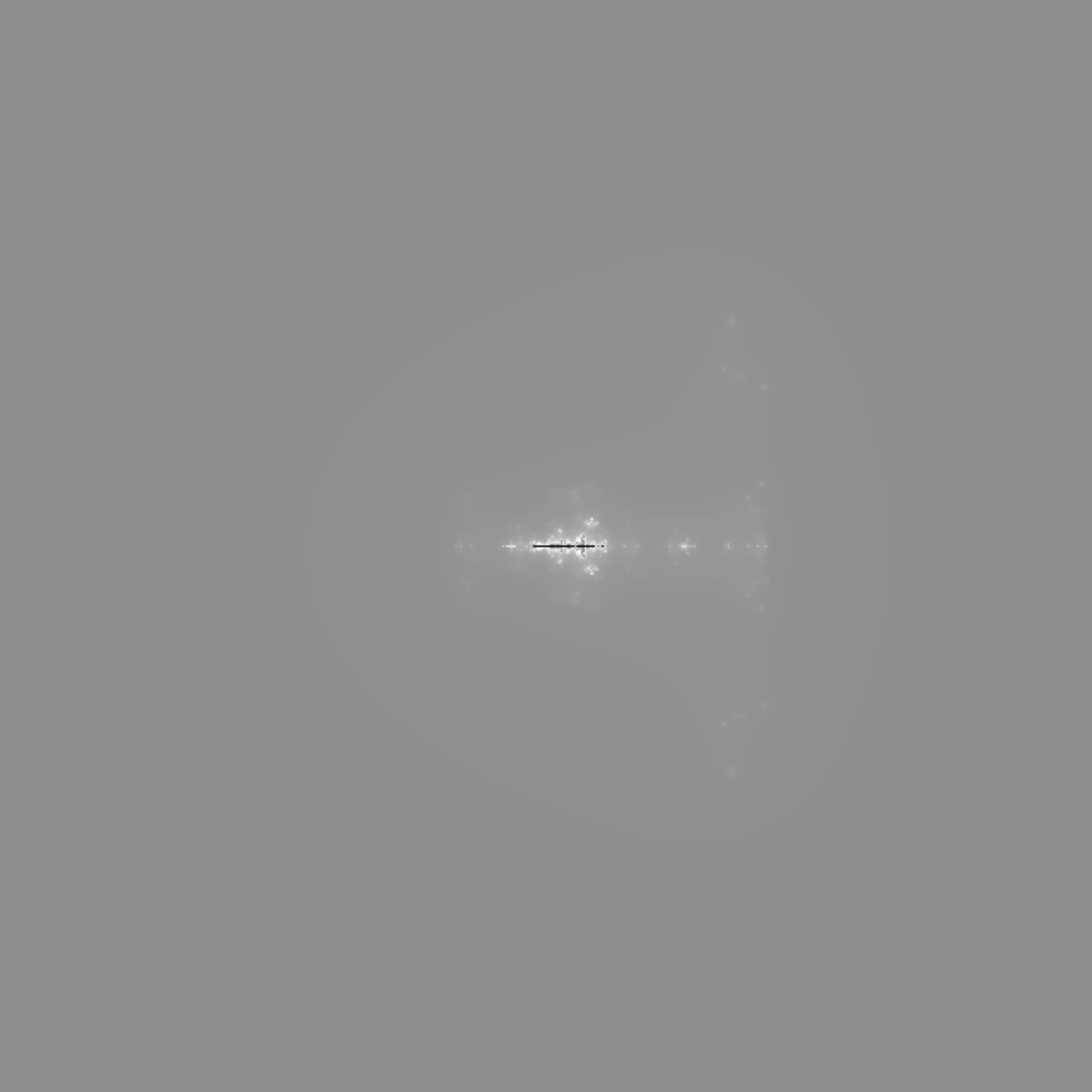}\hspace{2eM}
  \includegraphics[width=7.0cm, bb=0 0 640 640]{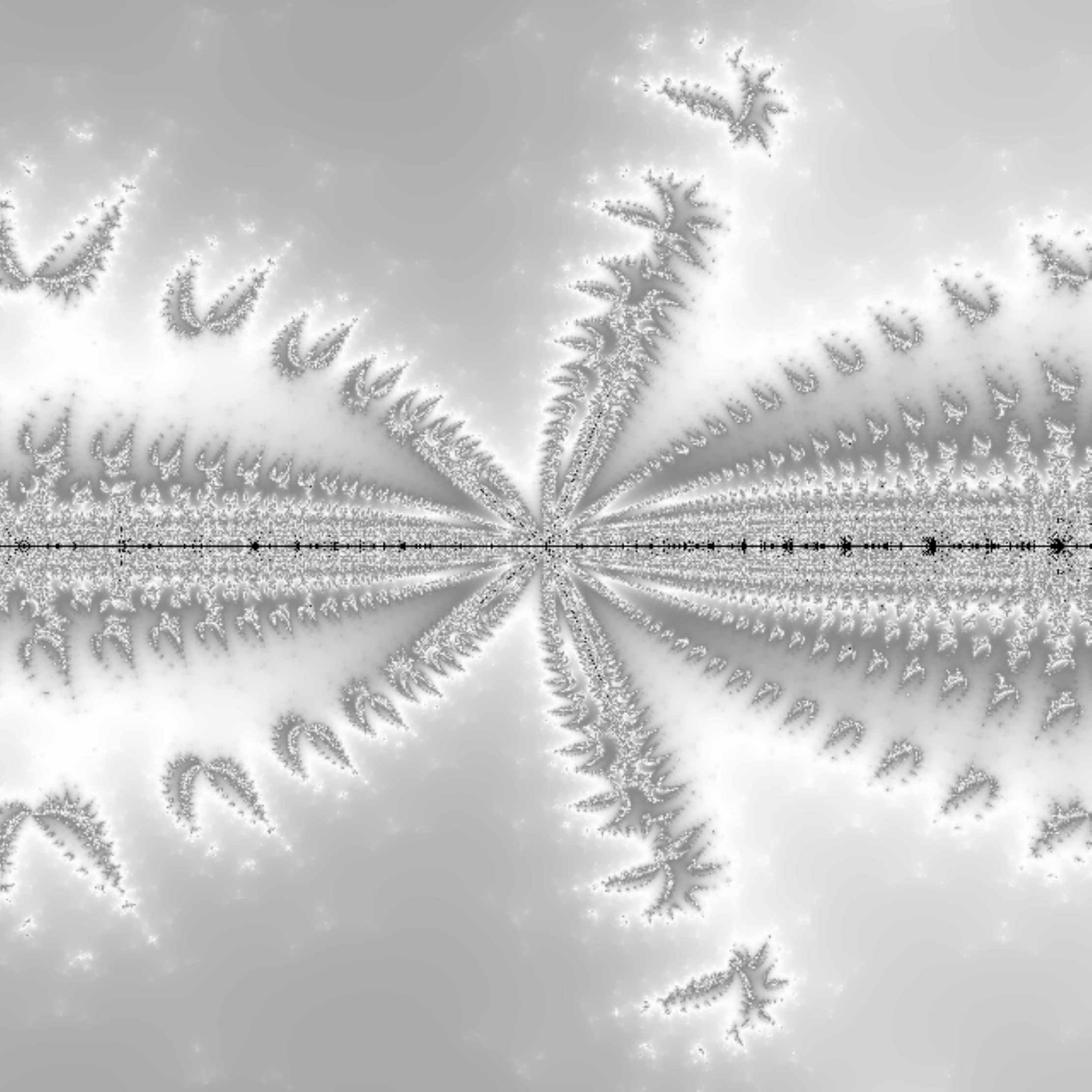}
  \caption{Figures for $\bfP = (-1, 1)$. The region on the left is 
  $\{t \in \CC \mid |{\rm Re}(t) +1| \leq 1, \, |{\rm Im}(t)| \leq 1\}$, 
  and that on the right is $\{t \in \CC \mid |{\rm Re}(t) +1| \leq 0.01, \, |{\rm Im}(t)| \leq 0.01\}$. 
  Note that, at the center $t = -1$ of the images, the point $(-1, 1)$ is periodic with respect to $H_{-1}$ 
  with period $2$. Theorem~\ref{thm:intro:infiniteness} says that $\Sigma((-1, 1))$ is an infinite set. 
  }
\end{figure}
\end{Example}

\goodbreak

The following result,  as a corollary of Theorem~\ref{thm:main:Mn},  gives the asymptotic behavior of $G_{\bfP, v}(t)$ as $|t|_{v} \to \infty$.

\begin{Corollary}
\label{cor:thm:main:Mn}
For any $v \in M_K$, $G_{\bfP, v}(t) - \widetilde{h}_{\bfH}(\bfP)\log|t|_v$ converges as $|t |_{v}\to\infty$. 
\end{Corollary} 

\Proof
Let $X_0, X_1$ be the homogeneous coordinates of $\PP^1$, and we take 
the global section $X_0$ of $\OO_{\PP^1}(1)$. It follows from Theorem~\ref{thm:main:Mn} 
that 
\[
 \Vert X_0 \Vert_v(t) = \lim_{n\to +\infty} \log \frac{|t|_v}{M_n(t)^{1/d^n \ell}}
\]
uniformly converges around $t = \infty$. Since $G_{\bfP}(t) =  \lim_{n\to +\infty} \frac{1}{d^n} M_n(t)$ by Lemma~\ref{lem:thm:main:Mn} and 
$\ell = \widetilde{h}_{\bfH}(\bfP)$ by \eqref{eqn:can:height:tilde:2}, 
we obtain that $G_{\bfP, v}(t) - \widetilde{h}_{\bfH}(\bfP) \log|t|_v$ converges as $t \to\infty$. 
\QED

\begin{Remark}
\label{rmk:cor:thm:main:Mn}
It follows from the proof of Theorem~\ref{thm:main:Mn} that, with the notation therein, 
the limit in Corollary~\ref{cor:thm:main:Mn} 
is given by  $-(1/d^N) \log\max\{|\alpha_{D^{\prime}}|_v, 1\}$,  $-(1/d^N) \log\max\{|\beta_{D^{\prime}}|_v, 1\}$, or 
$-(1/d^N) \log\max\{|\alpha_D|_v, |\beta_D|_v, 1\}$, according to Subcases 3-1, 3-2, 3-3 of the 
proof of Theorem~\ref{thm:main:Mn}. 
\end{Remark}

\setcounter{equation}{0}
\section{Set of periodic parameter values: result on infiniteness} 
\label{sec:set:periodic:parameters}

Now we would like to consider another application of $h_{\bfP}$, 
i.e., an application toward unlikely intersections of the sets of periodic parameter values of 
two different families of initial points. 

As before, let $K$ be a field, and we let $\bfH = (H_t)_{t \in \overline{K}}$ be a family of H\'enon map as in~\eqref{eqn:Henon maps} and $\bfP \in \Aff^2(K[t])$ be a family of  initial points.
Let $\bfQ \in \Aff^2(K[t])$ be another family of initial points. We are interested in what will happen 
if the intersection of periodic parameters $\Sigma(\bfP) \cap \Sigma(\bfQ)$ is infinite. 
As we discussed in 
\S~\ref{subsec:intro:periodic:para} that $\Sigma(\bfP)$ (or $\Sigma(\bfQ)$) may not be an infinite set for a 
family of H\'enon maps.  

The purpose of this section is to prove Theorem~\ref{thm:intro:infiniteness}, 
which gives a sufficient condition for $\Sigma(\bfP)$ being infinite. 

As we explain in Introduction, 
we assume that there exists an involution $\iota\colon \Aff^{2} \to\Aff^{2}$ over $K$ such that 
$\iota \circ \bfH \circ \iota = \bfH^{-1}$. Then necessarily, $\delta \in \{1, -1\}$. 
If $\delta = 1$, we assume that $f_t(-x) = f(x)$ in $K[t, x]$. Then one can check that 
the involution $\iota_\delta\colon (x, y) \mapsto (-\d y, -\d x)$  reverses $\bfH$ 
(see \S~\ref{subsec:intro:periodic:para} for more discussions). 
We note that the set of fixed points of $\iota_{\d}$ is given by the line 
\[
\Ccal_{\d}\colon  \d x + y = 0. 
\]
We recall Theorem~\ref{thm:intro:infiniteness}, which 
we prove at the end of this section. 

\begin{Theorem}
\label{thm:infiniteness:Sigma:P}
Let $K$ be a field (of any characteristic).
Let $\bfH$ be the family of H\'enon maps 
in\eqref{eqn:Henon maps}  
such that $\delta \in \{1, -1\}$. 
If $\delta = -1$, we assume that $f_t(x)$ is an even polynomial in $x$. 
Let $\bfP = (a(t), b(t))\in \Aff^2(K[t])$ satisfy Assumption~\ref{assumption}. 
Then, if $\bfP$ lies in $C_\delta$, i.e., $\d a(t) + b(t) = 0$ in $K[t]$, then 
$\Sigma(\bfP)$ is an infinite set.
\end{Theorem}

\begin{Example}
\label{eg:quad:Henon:map:3}
Let $H_t(x, y) = (y + x^2 + t, x)$ be as in Example~\ref{eg:quad:Henon:map}, 
which is a family of H\'enon maps that satisfies the assumption 
of 
Theorem~\ref{thm:infiniteness:Sigma:P} 
with $\delta = 1$ and 
$f_t(x) = x^2+t$. We take $\bfP = (a, b) \in K^2$. As we see 
in Example~\ref{eg:quad:Henon:map:2}, $\bfP$ satisfies Assumption~\ref{assumption}. 
Theorem~\ref{thm:infiniteness:Sigma:P} says that, 
if $a + b = 0$, then 
$\Sigma((a, b))$ is an infinite set.
\end{Example}

First, we prove several lemmas. 
For an ideal $\ga$  of $\overline{K}[t]$, let 
\[
  {}^{\sim}\colon K[t] \to \overline{K}[t]/\ga
\]
be composition of the inclusion map and the quotient map. 
We denote the base-changes by the morphism ${}^{\sim}$ of $\bfH$ (over $\Spec(K[t])$), 
$\bfP = (a(t), b(t))\in \Aff^2(K[t])$, $\bfH^n(\bfP) = (A_n(t), A_{n-1}(t))$ 
by $\widetilde{\bfH}$, $\widetilde{\bfP} =( \widetilde{a(t)}, \widetilde{b(t)})$, 
and $\widetilde{\bfH}^n(\widetilde{\bfP}) = (\widetilde{A_n(t)}, \widetilde{A_{n-1}(t)})$. 

\begin{Lemma}
\label{lemma:A:f}
Assume that $\bfP$ lies in $C_\delta$. 
Let $m \geq 1$ be a positive integer. 
Then the following are equivalent.
\begin{enumerate}
\item[(i)]
$\d \widetilde{{A}_m(t)} +\widetilde{{A}_{m-1}(t)} = 0$ in $\overline{K}[t]/\ga$.
\item[(ii)]
$\widetilde{\bfH}^{2m}(\widetilde{\bfP}) = \widetilde{\bfP}$ in $\Aff^2(\overline{K}[t]/\ga)$. 
\end{enumerate}
\end{Lemma}

\Proof
We show that (i) implies (ii). Indeed,
condition (i) means that 
$\widetilde{\bfH}^{m}(\widetilde{\bfP}) \in C_{\delta}(\overline{K}[t]/\ga)$. 
Since $\bfP$ lies in $C_\delta$, we have 
$\widetilde{\bfP} \in C_{\delta}(\overline{K}[t]/\ga)$ as well. 
Thus $\iota(\widetilde{\bfP}) = \widetilde{\bfP}$ and 
$\iota\left(\widetilde{\bfH}^{m}(\widetilde{\bfP})\right) = \widetilde{\bfH}^{m}(\widetilde{\bfP})$. 
We compute
\[
  \widetilde{\bfH}^{-m}(\widetilde{\bfP}) = \iota \circ \widetilde{\bfH}^m \circ \iota(\widetilde{\bfP}) =  \iota \circ
\widetilde{\bfH}^m(\widetilde{\bfP}) =  \widetilde{\bfH}^m(\widetilde{\bfP}),
\]
so that $\widetilde{\bfH}^{2m}(\widetilde{\bfP}) = \widetilde{\bfP}$. 

\smallskip
Next we show that (ii) implies (i).
Assume that $\widetilde{\bfH}^{2m}(\widetilde{\bfP}) = \widetilde{\bfP},$ then
\[
\widetilde{\bfH}^m(\widetilde{\bfP}) = \widetilde{\bfH}^{-m}(\widetilde{\bfP}) = \iota \circ \widetilde{\bfH}^m
\circ \iota(\widetilde{\bfP}) =  \iota \circ \widetilde{\bfH}^m(\widetilde{\bfP}).
\]
Thus $\widetilde{\bfH}^m(\widetilde{\bfP}) \in \Ccal_{\d}(\overline{K}[t]/\ga)$, 
which amounts to $\d \widetilde{{A}_m(t)} +\widetilde{{A}_m(t)} = 0$ in $\overline{K}[t]/\ga$.
\QED

Applying Lemma~\ref{lemma:A:f} to the ideal $\ga = \left((t-\alpha)^e\right)$ generated by the
polynomial $(t-\alpha)^e$ for $\alpha\in \overline{K}$ and $e\geq 1$, 
we have the following. 

\begin{Corollary}
\label{cor:A:f}
Assume that $\bfP$ lies in $C_\delta$. 
Let $m \geq 1$ be a positive integer. 
Then the following are equivalent.
\begin{enumerate}
\item[(i)]
$\d A_m(t) + A_{m-1}(t) \equiv 0 \pmod{(t-\alpha)^e}$.
\item[(ii)]
$\bfH^{2m}(\bfP) \equiv  \bfP\pmod{(t-\alpha)^e}$.
\end{enumerate}
\end{Corollary}

\begin{Remark}
\label{rem:odd period for -1delta}
In the case where $\d = -1,$ the same arguments also lead to similar results for odd period $n = 2 m +1$ where one considers the polynomial $\d A_{m+1}(t) + A_{m-1}(t)$ in place of $\d A_m(t) + A_{m-1}(t) $ in Lemma~\ref{lemma:A:f}  and Corollary~\ref{cor:A:f}. 
\end{Remark}

\medskip
Let $\alpha\in \overline{K}$. 
For $u(t) \in \overline{K}[t]$, we let 
$\ord_{\a}(u(t))$ be the multiplicity of the polynomial $
(t-\a)$ as a factor of $u(t)$. By convention, if $u(t)$ is the zero polynomial, 
then we set $\ord_{\a}(u(t)) = +\infty$. 
More generally, if $\bfr = (a_{1}(t), \ldots, a_{n}(t))\in \overline{K}[t]^{n}$, 
then 
\[
\ord_{\a}(\bfr) = \min\{\ord_{\a}(a_{i}(t)), \ldots, \ord_{\a}(a_{n}(t))\}.
\] 
Thus, $\ord_{\alpha}(\bfH^{2m}(\bfP) - \bfP)$ is 
the largest non-negative integer $e$ such that Corollary~\ref{cor:A:f}~(ii) holds.
By definition, the condition that  $P_\a$ is periodic with respect to $H_{\alpha}$ with period $n$ is equivalent to 
$\ord_{\alpha} (\bfH^n(\bfP)-\bfP) \geq 1$.  Further, Corollary~\ref{cor:A:f} says that
\begin{equation}
\label{eqn: order for 2m}
\ord_{\alpha}(\bfH^{2m}(\bfP) - \bfP) = \ord_{\alpha}\left(\d \, A_m(t) + A_{m-1}(t)\right).
\end{equation}

Suppose that $P_\a$ is periodic with respect to $H_\alpha$ with period $n \ge 1.$
Then $H_\alpha^{nk}(P_\alpha) = P_\alpha$
for any positive integer $k$.  Further, one has 
$$
\ord_\alpha(\bfH^{nk}(\bfP) - \bfP ) \geq \ord_\alpha(\bfH^{n}(\bfP) - \bfP)\;\text{for any $k \geq 1.$}
$$
Indeed, we put $e = \ord_\alpha(\bfH^{n}(\bfP) - \bfP)$. Then, we have
$\bfH^{n}(\bfP)\equiv \bfP \pmod{(t-\alpha)^e}.$ Consequently, $\bfH^{nk}(\bfP)\equiv \bfP \pmod{(t-\alpha)^e}
$ for any positive integer $k$. This is equivalent to saying that $\ord_\alpha(\bfH^{nk}(\bfP) - \bfP ) \geq e$. 

The following result gives a control on the growth of $\ord_{\alpha}(\bfH^{n}(\bfP) - \bfP )$, 
when $n$ ranges over a certain subset of positive integers.

\begin{Proposition}
\label{prop:multiplicity}
Let $q$ be a positive integer. 
Then there exists an integer $D_{q}\ge 1$ such that, 
for any $\alpha \in \overline{K}$ with $\ord_{\alpha}(\bfH^q(\bfP) - \bfP) \geq 1$ and 
for any positive integer $k$ relatively prime to $D_{q}$, we have 
\begin{equation}
\label{eqn:multiplicity}
\ord_{\alpha}(\bfH^{qk}(\bfP) - \bfP) = \ord_{\alpha}(\bfH^q(\bfP) - \bfP).
 \end{equation}
\end{Proposition}

Before we start the proof of Proposition~\ref{prop:multiplicity}, 
we show two lemmas. Let $\alpha\in \overline{K}$. 
We set $e_{q} := \ord_{\alpha}(\bfH^q(\bfP) - \bfP)$, 
and we assume that $e_q \geq 1$.  
We write
$\bfH^q(\bfP) - \bfP = (\varepsilon(t),\zeta(t))$, 
so that $e_q = \min\left(\ord_{\alpha}(\varepsilon(t)), \ord_{\alpha}(\zeta(t))\right)$. 
Then 
\[
\bfH^{q}(\bfP) = (A_{q}(t), \;  A_{q-1}(t)) = (a(t) + \varepsilon(t), \; b(t) +
\zeta(t)). 
\]
Set the Jacobian matrix of  $\bfH^{q}$ at the point $\bfP$ to be
\begin{align}
\label{eqn:Jacobian}
\Psi_{q} (t) 
& := \bfJ_{\bfH^{q}}(\bfP)  
= 
\bfJ_{\bfH}(\bfP)\cdots \bfJ_{\bfH}(\bfH^{q-2}(\bfP))\bfJ_{\bfH}(\bfH^{q-1}(\bfP))
 \\
\notag
& = 
\begin{pmatrix}  f^\prime_t(a(t)) & 1 \\ \delta & 0 \end{pmatrix}
\cdots
\begin{pmatrix}  f^\prime_t(A_{q-2}(t)) & 1 \\ \delta & 0 \end{pmatrix}
\begin{pmatrix}  f^\prime_t(A_{q-1}(t)) & 1 \\ \delta & 0 \end{pmatrix}. 
\end{align}
Let $\mathfrak{b} \subseteq K[t]$ denote the ideal  
generated by $\varepsilon(t)^2, \varepsilon(t) \zeta(t), \zeta(t)^2$. 

\begin{Lemma}
\label{lem: expansion at P}
Let $k \geq 1$. We define $\bfr_{k} \in \left(K[t]\right)^{2}$ by 
\begin{equation}
\label{eqn: expansion at P}
\bfH^{qk}(\bfP) = \bfP +(\varepsilon(t), \; \zeta(t))  \left(\sum_{i=0}^{k-1}\Psi_{q}(t)^{i}\right) + \bfr_{k}. 
\end{equation}
Then $\bfr_{k} \in \left(\mathfrak{b}\right)^2 \subseteq K[t]^2$. 
In particular, $\ord_{\alpha}(\bfr_{k}) \ge 2 e_{q}. $
\end{Lemma}

\Proof
We prove the lemma by induction on $k$. 
For $k=1$, we have 
\[
\bfH^{q}(\bfP) = \bfP + (\varepsilon(t), \; \zeta(t)), 
\]
so that $ \bfr_{1} = (0, 0)$ and $\ord_{\alpha}(\bfr_{1}) = + \infty \geq 2 e_q$. 

We assume the case for $k$, and we write 
\begin{align}
\label{eqn:Jacobian:induction}
\bfH^{q(k+1)}(\bfP) & = \bfH^{q}\left(\bfH^{q k}(\bfP)\right) \\
\notag
    & = \bfH^{q}\left(\bfP+  (\varepsilon(t), \zeta(t)) \left(\sum_{i=0}^{k-1}\Psi_{q}(t)^{i}\right)
   + \bfr_{k}\right)\\
\notag
    & = \bfH^{q}(\bfP) +   (\varepsilon(t), \zeta(t)) \left( \left(\sum_{i=0}^{k-1}\Psi_{q}(t)^{i}\right)
   + \bfr_{k}\right)  \Psi_{q}(t) +  \mathbf{s},  
\end{align}
where the last inequality is a consequence of the Taylor expansion 
and $\mathbf{s} \in \left(\mathfrak{b}\right)^2  \subseteq K[t]^2$. 

We continue the computation: 
\begin{align*}
    & \text{The right-hand side of \eqref{eqn:Jacobian:induction}} \\
    & \qquad = \bfP  + (\varepsilon(t), \zeta(t))  
    +  (\varepsilon(t), \zeta(t)) \left(\sum_{i=1}^{k}\Psi_{q}(t)^{i}\right)
   + \bfr_{k}\Psi_{q}(t)  +  \mathbf{s} \\
    & \qquad = \bfP  + (\varepsilon(t), \zeta(t))  \left(\sum_{i=0}^{k}\Psi_{q}(t)^{i}\right)  +
      \left(\bfr_{k}\Psi_{q}(t)+  \mathbf{s} \right). 
\end{align*}
Thus $\bfr_{k+1} =  \bfr_{k}\Psi_{q}(t) +  \mathbf{s}$.
Since $\bfr_{k} \in \left(\mathfrak{b}\right)^2$  by induction hypothesis and that 
$\Psi_{q}(t)$ is a matrix with entries in  $K[t]$, we conclude that 
$\bfr_{k+1} \in \left(\mathfrak{b}\right)^2$, and thus 
$\ord_{\alpha}(\bfr_{k+1}) \ge 2 e_{q}$. 
We complete the induction step, and obtain the lemma. 
\QED

\begin{Lemma}
\label{lem: equality of order}
Let $R$ be a discrete valuation ring equipped with valuation $v.$ Let $\bfr = (a_{1},\ldots, a_{n}) $ be a non-zero
element of $R^{n}.$ Set $v(\bfr) = \min \{v(a_{1}), \ldots, v(a_{n})\}.$ Then, for every invertible matrix $A \in
\GL(n,R)$, we have $v\left(\bfr A\right) = v(\bfr).$
\end{Lemma}

\Proof
It is an elementary exercise in algebra. Indeed, 
Let $\bfr$ and $A\in \GL(n,R)$ be as given. We note that  coordinates of
\[
\mathbf{s} := \bfr A = (b_{1}, \ldots, b_{n})
\]
are $R$-linear combinations of the coordinates $a_{1}, \ldots, a_{n}$. Thus $v(b_{i}) \ge v(\bfr)$ and we have
$v(\mathbf{s}) \ge v(\bfr).$ The same argument applies to the vector $\bfr = \mathbf{s} A^{-1}$ and note that $A^{-1}$ also
has entries in $R.$ We can therefore conclude that $v(\bfr)\ge v(\mathbf{s})$ and $v\left(\bfr A\right) = v(\mathbf{s}) = v(\bfr)
$. 
\QED

We are ready to prove Proposition~\ref{prop:multiplicity}.

\begin{proof}[Proof of Proposition~\ref{prop:multiplicity}]
For a positive integer $k$, we set
$
\Ecal_{q, k}(t) := \sum_{i=0}^{k-1}\,\Psi_{q}(t)^{i}
$, 
where $\Psi_{q}(t)$ is defined in \eqref{eqn:Jacobian}.  Then 
By Lemma~\ref{lem: expansion at P}, we have 
\[
\bfH^{qk}(\bfP) = \bfP +  (\varepsilon(t), \; \zeta(t)) \Ecal_{q, k}(t)+ \bfr_{k}.
\]

Let $\alpha\in \overline{K}$ such that $\ord_{\alpha}\left( \bfH^{q}(\bfP) - \bfP \right) \geq 1$. 
Let $\xi_{1}, \xi_{2}$ be the eigenvalues of $\Psi_{q}(\alpha).$ 
If the characteristic of $K$ is zero,   
we define the integer $D_{q, \alpha}$ to be the least
common multiple of the multiplicative orders of  $\xi_{1}$ and $\xi_{2}$, where by convention, 
we define the multiplicative order of $\xi_{i}\, (i=1, 2)$ to be $1$ if it is not a root of unity. 
If the characteristic of $K$ is positive, then 
we define the integer $D_{q, \alpha}$ to be the least
common multiple of the characteristic of $K$ and 
the multiplicative orders of  $\xi_{1}$ and $\xi_{2}$. 
We claim that, for any positive
integer $k$ relatively prime to $D_{q, \alpha}$, the equality
\[
\ord_{\alpha}\left( \bfH^{qk}(\bfP) - \bfP \right) = \ord_{\alpha}\left( \bfH^{q}(\bfP) - \bfP \right)
\] 
holds.  

Since $\Psi_{q}(\alpha)$ is similar to the matrix $\begin{pmatrix} \xi_1 & * \\ 0 & \xi_2\end{pmatrix}$, 
$\Ecal_{q, k}(t)$ is similar to $\begin{pmatrix}1 + \cdots +\xi_1^{k-1} & * \\ 0 & 1 + \cdots + \xi_2^{k-1} \end{pmatrix}$. As $k$ is relatively prime to $D_{q, \alpha}$, we have 
$1 + \cdots + \xi_i^{k-1} \neq 0$ for $i = 1, 2$, and thus $\det \Ecal_{q, k}(\alpha) \ne 0$. 
It follows that $\Ecal_{q, k}(t)\in \GL(2, R_{\alpha})$ where $R_{\alpha} = \overline{K}[t]_{(t-\alpha)}$ is 
the localization of $\overline{K}[t]$ at the prime ideal  $(t-\alpha).$  

By Lemma~\ref{lem: equality of order}, we have $
\ord_{\alpha} \left(\bfr\Ecal_{q, k}(t)\right) = \ord_{\alpha}(\bfr)$ for any non-zero row vector $\bfr\in
\overline{K}[t]^{2}.$ Thus 
\[
\ord_{\alpha} \left( (\varepsilon(t), \; \zeta(t)) \Ecal_{q, k}(t)  \right) = \ord_{\alpha}
(\varepsilon(t), \; \zeta(t))  = e_{q}.
\]
It follows from Lemma~\ref{lem: expansion at P} that  $\ord_{\alpha} \bfr_{k} \ge 2 e_{q}.$ Thus
\[
\ord_{\alpha}\left( \bfH^{qk}(\bfP) - \bfP \right)  = \ord_{\alpha} \left((\varepsilon(t), \; \zeta(t)) \Ecal_{q, k}(t) + \bfr_{k} \right)  = e_{q} = \ord_{\alpha}\left( \bfH^{q}(\bfP) - \bfP
\right). 
\]
Thus we obtain the claim. 

If $\bfH^{q}(\bfP) = \bfP$, then the assertion clearly holds because 
$\ord_{\alpha}(\bfH^{qk}(\bfP) - \bfP) = \ord_{\alpha}(\bfH^q(\bfP) - \bfP) = +\infty$. 
If $\bfH^{q}(\bfP) \neq \bfP$, then 
there are only finitely many $\alpha \in \overline{K}$ with 
$\ord_{\alpha}(\bfH^q(\bfP) - \bfP) \geq 1$. We consider 
the product of $ D_{q, \alpha}$ for all such $\alpha$, and put 
$D_q := \prod_{\alpha} D_{q, \alpha}$. 
Then $D_q$ gives the desired property that~\eqref{eqn:multiplicity}  holds for all positive integers $k$ which are relatively prime to~$D_{q}.$ 
\end{proof}

Finally, we prove Theorem~\ref{thm:infiniteness:Sigma:P}. 
For a positive integer $n$, 
we define the subset of $\Sigma(\bfP)$ in \eqref{eqn:Sigma:bfP}
by 
$$
\Sigma_n(\bfP) := \{ t \in \Aff^1(\overline{K}) \mid H_t^n(P_t) = P_t\}. 
$$

\begin{proof}[Proof of Theorem~\ref{thm:infiniteness:Sigma:P}]
Observe that $\Sigma\left(\bfH^{m}(\bfP)\right) = \Sigma\left(\bfP\right)$ for any integer $m$. 
Since $\iota \circ \bfH \circ \iota = \bfH^{-1}$ and $\iota(\bfP) = \bfP$, it follows from 
Assumption~\ref{assumption} hat both $(\deg A_n(t))_{n \geq 0}$ and 
$(\deg B_n(t))_{n \geq 0}$ are unbounded. 
We replace $\bfP$ with $\bfH^{m}(\bfP)$ for sufficiently large $m$ 
if necessary, and we may assume that the degree $\deg a(t)
> 0$ and that $\deg A_{n}(t) = d^{n} \deg a(t)$ for all $n\ge 0$ 
(see Proposition~\ref{prop:on:assumption}). 

Let $D_{2}$ be the integer as given in Proposition~\ref{prop:multiplicity} for $q =  2$ and let $p$ be an odd 
prime number such that $p\nmid D_{2}$.  We claim that there always exists a parameter 
$\beta \in \overline{K}$ 
such that $P_\b$ has period $p$ or $2p$ under $H_\beta.$  Indeed, by Lemma~\ref{lemma:A:f}  
\[
  \Sigma_{2p}(\bfP) = \{\a \in \Aff^1(\overline{K}) \mid \d \, A_p(\a) + A_{p-1}(\a) = 0\}.
\]
By definition, if $\alpha \in \Sigma_{2p}(\bfP)$ then $P_\a$ has period of $1, 2, p$ or $2p$ for $H_\alpha.$

Let $\alpha \in \overline{K}$ be a root of $\d\, A_{1}(t) + A_{0}(t) = 0$. 
It follows from Corollary~\ref{cor:A:f}  that $\ord_\alpha(\bfH^{2}(\bfP)-\bfP) \geq 1$. 
As a result, we have 
\begin{align*}
\ord_{\a} \left(\d\,A_{p}(t) + A_{p-1}(t)\right)  
& = \ord_\alpha\left(\bfH^{2p}(\bfP)-\bfP\right) & \text{by Corollary~\ref{cor:A:f}} & \\
& =  \ord_\alpha\left(\bfH^{2}(\bfP)-\bfP\right)  &  \text{by Proposition~\ref{prop:multiplicity}}  & \\
& = \ord_{\a} \left(\d\, A_{1}(t) + A_{0}(t)\right) &  \text{by Corollary~\ref{cor:A:f}} & .
\end{align*}
Thus there exists $W_{2p}(t) \in K[t]$ such that 
\[
\d \, A_{p}(t) + A_{p-1}(t) = \left(\d A_{1}(t) + A_{0}(t)\right) W_{2p}(t), 
\]
and $\d A_{1}(t) + A_{0}(t)$ and $W_{2p}(t)$ are relatively prime. 
Since 
\[
\deg \left(A_{p}(t) + A_{p-1}(t)\right) = 
\deg(A_p(t)) =  d^{p}\ell > d \ell =  \deg A_1(t) = \deg \left(A_{1}(t) + A_{0}(t)\right), 
\]
$W_{2p}(t)$ is a non-constant polynomial. 

We take a root  $\b\in \overline{K}$ of $W_{2p}(t)$. Then $\b \in \Sigma_{2p}(\bfP)$. 
Since $\d A_{1}(\beta) + A_{0}(\beta) \neq 0$, it follows that $H_{\b}^2(P_{\b}) \neq P_{\b}$. 
In particular, we also have $H_{\b}(P_{\b}) \neq P_{\b}$. Thus 
$P_\b$ has period $p$ or $2p$. 

We denote the above $\beta$ by $\beta_p$.
Then $\beta_p \in \Sigma\left(\bfP\right)$. Further, if $p$ and $p'$ are distinct odd prime numbers relatively prime
to $D_{2},$ then the period of $P$ under $H_{\beta_p}$ and that under $H_{\beta_{p'}}$ are different,  so
that $\beta_p \neq \beta_{p'}$. We conclude that $\Sigma\left(\bfP\right)$ is an infinite set as desired.
\end{proof}

\begin{Remark}
\label{rem:odd periodic parameters}
For $\delta = 1$ and $\delta = -1$, 
we actually show that the set $\Sigma_{{\rm even}}(\bfP) := \bigcup_{m\ge 1} \Sigma_{2m}(\bfP)$ is an infinite set. In the case where $\d = -1$, as indicated in Remark~\ref{rem:odd period for -1delta},  by applying similar arguments one can prove that $\Sigma_{{\rm odd}}(\bfP) := \bigcup_{m\ge 1} \Sigma_{2m+1}(\bfP)$ is also an infinite set. 
\end{Remark}

\begin{Remark}[Question about primitive prime divisors in a family of H\'enon maps]
\label{rem:primitive divisors}
Related to Theorem~\ref{thm:infiniteness:Sigma:P} is the question of {\em primitive prime divisors} in arithmetic 
dynamics which has been considered in the case of one-variable rational maps in recent years (see for 
instance~\cite{fabgran, gnt, ing-sil}). For families of H\'enon maps $\bfH$ and initial points
$\bfP$, we say that  $t-\a\in \overline{K}[t]$ is a  primitive prime divisor for 
$\bfH^{n}(\bfP) - \bfP$ if $H_{\a}^{n} (P_{\a}) = P_{\a}$ but $H_{\a}^{m}(P_{\a}) \ne P_{\a}$ 
for any $1\le m \le n-1.$ It would be interesting to investigate how often the sequence $\bfH^{n}(\bfP) - \bfP$ 
has a primitive prime divisor as $n$ runs through all positive integers. 
Let $\bfH$ and $\bfP$ be  given as in  Theorem~\ref{thm:infiniteness:Sigma:P}. 
We would like to ask the following question: Is it true that $\bfH^{2m}(\bfP) - \bfP$ has a primitive prime divisor for all but finitely many positive integer $m$? In the case of $\d = -1$ one can further ask  whether or not $\bfH^{n}(\bfP) - \bfP$ has a primitive prime divisor for all but finitely many positive integer $n$. An affirmative answer to these question strengthens Theorem~\ref{thm:infiniteness:Sigma:P}.
\end{Remark}
\setcounter{equation}{0}
\section{Unlikely intersection for one-parameter families of H\'enon maps}
\label{sec:unlikely intersection}

In this section, we prove Theorem~\ref{thm:intro:unlikely:intersection}, 
which addresses unlikely intersections for a one-parameter family 
H\'enon maps. As we discuss in Introduction, a key element for the proof is the equidistribution theorem for points of small height associated to a semipositive adelically metrized line bundle. 

Here we recall the following result as an application of Yuan's equidistribution theorem \cite{Yuan} which holds for varieties of any dimension. It is implicit in Section~3 of the arXiv version of \cite{YZ}, and is stated in~\cite{ght2}. 
We will apply it for $X = \PP^1$. 
We note that, as our parameter space is $\PP^1$, one could also apply the equidistribution theorem for curves as proven in~\cite{autissier, Baker-Rumely06, CL, fr04, fr06, Th}.

\begin{Proposition}[\protect{\cite[Corollary~4.3]{ght2}}]
\label{prop:to:use}
Let $K$ be a number field. 
Let $X$ be an integral projective variety over $K$, and let 
$L$ be an ample line bundle on $X$. 
Let $\Vert\cdot\Vert_{1} := \{\Vert\cdot\Vert_{1, v}\}_{v \in M_K}$ and 
$\Vert\cdot\Vert_{2} := \{\Vert\cdot\Vert_{2, v}\}_{v \in M_K})$ 
be adelic metrics on $L$ such that $\Lcalbar_1 := (L, \Vert\cdot\Vert_{1})$ and
$\Lcalbar_1 := (L, \Vert\cdot\Vert_{2})$ are semipositive adelically metrized line bundles. 
Assume that there exists a nonzero rational section $s$ and a point $P\in X(\overline{K})$ such that, for all but finitely many places $v \in M_K$,  $\|s(P) \|_{1, v}=\|s(P)\|_{2, v}$.   Let $\{x_m\}_{m \geq 1}$ be an infinite
sequence of algebraic points in $X ({\overline{K}})$ that are Zariski dense in $X$. Suppose that
  \begin{equation}
  \label{eqn:small}
 \lim_{m\to\infty} h_{{\Lcalbar}_1} (x_m) = \lim_{m\to\infty} h_{{\Lcalbar}_2}(x_m)
  = h_{{\Lcalbar}_1} (X) = h_{{\Lcalbar}_2}(X) = 0.
\end{equation}
Then $h_{{\Lcalbar}_1} = h_{{\Lcalbar}_2}$ on $X(\overline{K})$.
\end{Proposition}

For the convenience of the reader, we restate Theorem~\ref{thm:intro:unlikely:intersection} as follows. 

\begin{Theorem}
\label{thm:unlikely intersection thm}
Let $K$ be a number field. 
Let $\bfH = (H_t)_{t \in \overline{K}}$ be the one-parameter family 
of H\'enon maps as in \S~\ref{sec:family of Henon}. 
Let  $\bfP, \bfQ \in \Aff^2(K[t])$ be initial points 
satisfying Assumption~\ref{assumption} such that 
$\Sigma(\bfP)$ and $\Sigma(\bfQ)$ are infinite. 
Then the followings are equivalent: 
\begin{enumerate}
\item[(i)]
$\Sigma(\bfP)\cap \Sigma(\bfQ)$ is infinite; 
\item[(ii)]
$\Sigma(\bfP) = \Sigma(\bfQ)$; 
\item[(iii)]
$G_{\bfP, v} = G_{\bfQ, v}$ on the parameter space $\PP^1(\KK_v)$ for all $v \in M_K$. 
\item[(iv)]
$h_{\bfP} = h_{\bfQ}$ on the parameter space $\PP^1(\overline{K})$. 
\end{enumerate}
\end{Theorem}

\Proof
We first show ``(i) $\Longrightarrow$ (iv).'' 
Since $\Sigma(\bfP)\cap \Sigma(\bfQ)$ is infinite, we can take a sequence 
$\{t_m\}_{m \geq 1}$ of distinct points in $\Sigma(\bfP)\cap \Sigma(\bfQ)$. 
By the definition of $\Sigma(\bfP)$  
we have $\widetilde{h}_{H_{t_m}}(P_{t_m}) = 0$ 
and $\widetilde{h}_{H_{t_m}}(Q_{t_m}) = 0$. 
It follows from Theorem~\ref{thm:main2:revisited} (see also~\eqref{eqn:thm:main2:revisited:1}) that 
\begin{equation}
h_{{\Lcalbar_{\bfP}}} (t_m) = h_{{\Lcalbar}_{\bfQ}} (t_m) = 0. 
\end{equation}

Note that the underlying line bundle of the semipositive adelic metrics of ${\Lcalbar}_{\bfP}$ is
$\OO_{\PP^1}(1)$, and from the proof of Proposition~\ref{prop:line bundle LP}, 
it is the uniform limit of $\left(\mathscr{X}_n, (\mathscr{L}_n, \{\Vert\cdot\Vert^{\prime}_{n, v}\}\right)$, 
where $\mathscr{L}_n$ is relatively ample.   Zhang's fundamental inequality \cite[Theorem~(1.10)]{zhang} 
gives 
\[
\inf_{t \in \PP^1(\overline{K})} h_{{\Lcalbar}_{\bfP}}(t) 
\leq h_{{\Lcalbar}_{\bfP}} (\PP^1)
\leq \sup_{\dim Z = 0} \inf_{t \in (\PP^1\setminus Z)(\overline{K})} 
h_{{\Lcalbar}_{\bfP}}(t). 
\]
The existence of the infinite set $\{t_m\}_{m \geq 1}$ with $h_{{\Lcalbar}_1} (t_m) = 0$ 
implies that both the left-hand and right-hand sides are equal to $0$, and we obtain 
$h_{{\Lcalbar}_{\bfP}} (\PP^1) = 0$. Similarly, we have $h_{{\Lcalbar}_{\bfQ}} (\PP^1) = 0$. 
By Proposition~\ref{prop:to:use} and Theorem~\ref{thm:main2:revisited}, 
we have $h_{{\Lcalbar}_{\bfP}} = h_{{\Lcalbar}_{\bfQ}}$, so that $h_{\bfP} = h_{\bfQ}$. 

The implication ``(iv) $\Longrightarrow$ (ii)'' follows from Proposition~\ref{prop:periodic parameter}, 
and the implication ``(ii) $\Longrightarrow$ (i)'' is obvious. The implication ``(iii) $\Longrightarrow$ (iv)'' 
follows from Proposition~\ref{prop:decomposition}. 

Lastly we show  ``(i) $\Longrightarrow$ (iii).'' 
Let $v \in M_K$. The proof of \cite[Corollary~4.3]{ght2} shows that there exists a constant $c_v \in \RR$ such that 
$G_{\bfP, v} = G_{\bfQ, v} + c_v$. Take any $t \in \Sigma(\bfP) \cap \Sigma(\bfQ)$. 
Since $t \in  \Sigma(\bfP) \subseteq K_{\bfP, v}$, we have 
$G_{\bfP, v}(t) = 0$ by Theorem~\ref{thm:K:P:v:characterization}. 
Similarly, we have $G_{\bfQ, v}(t) = 0$. 
Thus $c_v = 0$, and we obtain (iii). 
\QED

Note that in the case of a family of one-variable polynomial dynamics over $K[t]$, 
using the B\"ottcher coordinate, it is proved that $\bfP$ and $\bfQ$ satisfy some orbital relations. 
For a family of H\'enon maps, the techniques using the B\"ottcher coordinates seem not easily extended. 
In the following, we would like to  point out that reversible H\'enon maps   play a role here. 
This is the easy direction, and we ask in Question~\ref{q:intro:1} in Introduction if the converse 
also holds. 

\begin{Proposition}
\label{prop:converse:q:intro:1}
\begin{parts}
\Part{(1)}
Let $\Omega$ be be an algebraically closed field that is complete with 
respect to an absolute value. 
Let $H\colon \Aff^2 \to \Aff^2$ be a H\'enon map over $\Omega$. We assume that 
there exists an invertible affine map $\sigma\colon \Aff^2 \to \Aff^2$ over $\Omega$ such that 
$\sigma^{-1} \circ H^m \circ \sigma = H^{m}$  or 
$\sigma^{-1} \circ H^m \circ \sigma = H^{-m}$ for some $m \geq 1$. 
Let $G$ be the Green function defined in \eqref{eqn:def:G:intro}. 
Then for any $P \in \Aff^2(\Omega)$, 
we have $G(\sigma(P)) = G(P)$. 
\Part{(2)}
Let $K, \bfH, \bfP$ be as in Theorem~\ref{thm:unlikely intersection thm} such that 
$\Sigma(\bfP)$ is infinite. Assume that 
there exists an invertible affine map ${\boldsymbol \sigma}\colon \Aff^2 \to \Aff^2$ over $\overline{K}[t]$ 
with ${\boldsymbol \sigma}^{-1} \circ \bfH^m \circ {\boldsymbol \sigma} = \bfH^{m}$ or 
${\boldsymbol \sigma}^{-1} \circ \bfH^m \circ {\boldsymbol \sigma} = \bfH^{-m}$ or some $m \geq 1$. 
Let $n \in \ZZ$. Then $h_{\bfH^n( {\boldsymbol \sigma}(\bfP))} = h_{\bfP}$. 
\end{parts}
\end{Proposition}

\Proof
{\bf (1)} Since $\sigma$ is an invertible affine map of $\Aff^2(\Omega)$, there exists a constant $C \geq 1$ 
such that 
$
C^{-1} \Vert Q \Vert \leq \Vert \sigma(Q) \Vert \leq C \Vert Q \Vert
$ 
for any $Q \in \Aff^2(\Omega)$. 
In particular, we have 
$
C^{-1} \Vert H^n(P) \Vert \leq \Vert \sigma(H^n(P)) \Vert \leq C \Vert H^n(P) \Vert
$ 
for any $n \in \ZZ$. 

Suppose that $\sigma^{-1} \circ H^m \circ \sigma = H^{m}$. 
Then we have
\begin{align*}
 G^+(\sigma(P)) 
 & =  \lim_{n\to +\infty}\frac{1}{d^n} \log^+ \Vert H^n (\sigma(P))\Vert 
 =  \lim_{k\to +\infty}\frac{1}{d^{km}} \log^+ \Vert H^{km} (\sigma(P))\Vert \\
 & =  \lim_{k\to +\infty}\frac{1}{d^{km}} \log^+ \Vert \sigma^{-1} H^{km} (\sigma(P))\Vert
 = \lim_{k\to +\infty}\frac{1}{d^{km}} \log^+ \Vert H^{km}(P)\Vert \\
 & =  \lim_{n\to +\infty}\frac{1}{d^n} \log^+ \Vert H^n (P)\Vert  = G^+(P).  
\end{align*}
Similarly, we have  $G^-(\sigma(P))  = G^-(P)$. 
Since $G = \max\{G^+, G^-\}$,  we have $G(\sigma(P))= G(P)$. 

Next suppose that $\sigma^{-1} \circ H^m \circ \sigma = H^{-m}$. Then 
by similar arguments as above we have 
$G^+(\sigma(P))  = G^-(P)$ and $G^-(\sigma(P))  = G^+(P)$. 
Since $G = \max\{G^+, G^-\}$,  we again have $G(\sigma(P))= G(P)$. 
 
\smallskip
{\bf (2)} 
Since $h_{\bfP}$ does not change under a finite extension of fields $K^\prime/K$, 
replacing $K$ by a finite extension field, we may assume that ${\boldsymbol \sigma}$ is defined over $K[t]$. 
It then follows from (1) that $h_{{\boldsymbol \sigma}(\bfP)} = h_{\bfP}$.  
Suppose that $t \in \Sigma(\bfP)$, and we write $H_t^n(P_t) = P_t$. Then 
$H_t^{mn}(P_t) = P_t$, and we have $(\sigma_t \circ H_t^{mn}) (P_t) = \sigma_t(P_t)$, 
which implies $H_t^{mn}(\sigma_t(P_t)) = \sigma_t(P_t)$ if 
${\boldsymbol \sigma}^{-1} \circ \bfH^m \circ {\boldsymbol \sigma} = \bfH^{m}$, and 
$H_t^{-mn}(\sigma_t(P_t)) = \sigma_t(P_t)$ if ${\boldsymbol \sigma}^{-1} \circ \bfH^m \circ {\boldsymbol \sigma} = \bfH^{-m}$. In either case, we have $t \in \Sigma({\boldsymbol \sigma}(\bfP))$. If follows that $\Sigma({\boldsymbol \sigma}(\bfP))$ 
is infinite. Since $\Sigma({\boldsymbol \sigma}(\bfP)) = \Sigma(\bfH^n({\boldsymbol \sigma}(\bfP)))$, Theorem~\ref{thm:unlikely intersection thm} 
implies that $h_{\bfH^n({\boldsymbol \sigma}(\bfP))} = h_{{\boldsymbol \sigma}(\bfP)} = h_{\bfP}$. 
\QED

\setcounter{equation}{0}
\section{Set of periodic parameter values: result on emptiness}
\label{sec:finiteness:per:param}

This section is complementary to Section~\ref{sec:set:periodic:parameters}. 
Let $K$ be a field. 
We consider the family of quadratic H\'enon maps in Example~\ref{eg:quad:Henon:map}: 
\begin{equation}
\label{eqn:Henon:quadratic}
\bfH = (H_{t})_{t \in \overline{K}}\colon
\Aff^2 \to \Aff^2, \quad
(x, y)
\mapsto
\left(y + x^2 + t, x\right). 
\end{equation}
We take $a, b \in \overline{K}$, and 
consider a constant family $\bfP = (a, b)$ as an initial point.  
In the following, we will simply denote the specialization of $\bfP$ at $t$ 
by $P$ instead of $P_t$, because it is a constant family. 

As we see in Example~\ref{eg:quad:Henon:map:3}, if 
$a + b = 0$, then 
\[
  \# \Sigma((a, b)) 
  = \#\{t \in \Aff^1(\overline{K}) \mid \text{$(a, b)$ is periodic with respect to $H_t$}\}
  = \infty. 
\]

\medskip
In this section, we study if $\Sigma(\bfP) = \emptyset$. In other words, we study if, 
for some initial points $\bfP$ of some one-parameter families of H\'enon maps, 
there are no periodic parameter values. 

To begin with, we consider the condition $H_{t}^n(P) = P$ for any fixed $n \geq 2$. 
Recall that we write $H_t^n(P) = (A_n(t), A_{n-1}(t))$ (see \eqref{eqn:H:n:A:B}). 
If $a, b, t \in \overline{K}$ satisfy
$H_t^n(P) = P$, then, as polynomials in $t$, $A_n(t) -a = 0$ and $A_{n-1}(t) - b = 0$
have a common solution, so that the resultant $\Res_t(A_n(t)-a, A_{n-1}(t)-b) = 0$.

\begin{Lemma}
\label{lem:a:b:indep}
If $a, b \in K$ are algebraically independent over the prime subfield of $K$, 
then $\Sigma((a, b)) = \emptyset$.  
\end{Lemma}

\Proof
We observe that $\bfH$ is defined over the prime field. 
Since $a - b \neq 0$, there is no $t \in \overline{K}$ with $H_{t}(P) = P$. 
For $n \geq 2$, since $\Res_t(A_n(t)-a, A_{n-1}(t)-b)$ is not zero as a polynomial of $a, b$ (which follows, 
for example, from Proposition~\ref{prop:resultant:is:monic:app} below),  algebraic independence 
of $a, b$ implies that there is no $t \in \overline{K}$ with $H_{t}^n(P) = P$. 
Hence $\Sigma((a, b)) = \emptyset$. 
\QED

\begin{Remark}
In the proof of Lemma~\ref{lem:a:b:indep}, 
we use the fact that $\Res_t(A_n(t)-a, A_{n-1}(t)-b)$ is not zero as a polynomial of $a, b$ for any $n \geq 2$, 
which may not be straightforward to show. Our proof uses 
Proposition~\ref{prop:resultant:is:monic:app} below. 
\end{Remark}

From Lemma~\ref{lem:a:b:indep}, we are interested in the case where $K$ 
is the prime field $\FF_p$ when the characteristic of $K$ is $p \geq1$, 
and $\QQ$ when the  the characteristic of $K$ is zero. 

For the positive characteristic case, we have the following. 

\begin{Lemma}
\label{lem:a:b:finite}
Let $K = \FF_p$. 
Then for any $(a, b) \in \overline{\FF}_p$, we have 
$\Sigma((a, b)) = \overline{\FF}_p$.  
\end{Lemma}

\Proof
We take any $t \in \overline{\FF}_p$. We take a finite subfield $F$ of $\overline{\FF}_p$ 
such that $a, b, t \in F$. Then the map $H_t(x, y) = (y + x^2 + t, x)$ gives an 
automorphism of $F^2$. Thus $(a, b) \in F^2$ is periodic with respect to $H_t$ (with period dividing 
$((\#F)^2)!$). 
We obtain that $t \in \Sigma((a, b))$. Since $t$ is arbitrary, we have $\Sigma((a, b)) = \overline{\FF}_p$. 
\QED

Thus the most interesting case would be the case $K = \QQ$. 
Our main result in this section is the following proposition, 
giving many (constant) points $\bfP$ with empty periodic parameter values.  
Let $\overline{\ZZ}$ denote the ring of algebraic integers. 
 
\begin{Proposition}[$=$ Proposition~\ref{prop:intro:finiteness}]
\label{prop:resultant:is:monic}
Let $K = \QQ$ (or any field of characteristic zero).  Then for any $(0, b)$ with $b \notin \overline{\ZZ}$, 
we have $\Sigma((0, b)) = \emptyset$. 
\end{Proposition}

\Proof
Let $\bfP = (0, b)$ with $b \in \overline{K} \setminus \overline{\ZZ}$. 
Since $b \neq 0$, we have $H_t(P) \neq P$ for any $t$.
Suppose that $\Sigma(P) \neq \emptyset$. Then there exist $n \geq 2$ and $t_0 \in \CC$ such that
$H_{t_0}^n(P) = P$. Since the equations $A_n(t) = 0$ and $A_{n-1}(t) - b = 0$
have the common solution $t_0$, we have $\Res_t(A_n(t), A_{n-1}(t)) = 0$.
By Proposition~\ref{prop:resultant:is:monic:app} below, we then have $b \in \overline{\ZZ}$, which contradicts with
our assumption.
\QED

\begin{Example}
Let $K= \QQ$. By Example~\ref{eg:quad:Henon:map:3}, 
we have $\# \Sigma((0, 0)) = \infty$ and $\# \Sigma((-1, 1)) = \infty$. 
By Proposition~\ref{prop:resultant:is:monic}, we have $\Sigma((0, 1/2)) = \emptyset$. 
On the other hand, Figure~1, 2, 3 depicting escaping rates 
of $\bfP$ each looks complicated with fractal structures. 
\end{Example}

We finish this paper by showing the following proposition, which was used to prove 
Proposition~\ref{prop:resultant:is:monic}. 

\begin{Proposition}
\label{prop:resultant:is:monic:app}
Let $n \geq 2$.
As a polynomial in $b$, we have
\begin{equation}
  \label{eqn:prop:resultant:is:monic}
\Res_t(A_n(t), A_{n-1}(t)-b)
= \pm b^{2^{n-1}} + c_{1} b^{2^{n-1} -1} + \cdots + c_{2^{n-1}}
\end{equation}
with $c_i \in \ZZ$ for $i = 1, \ldots, 2^{n-1}$.
\end{Proposition}

\Proof
We compute some Sylvester resultants to prove Proposition~\ref{prop:resultant:is:monic:app}. 
Computations are not difficult, but lengthy. We only sketch a proof, leaving  the details to 
the reader. 

We observe that the condition $H^n_t(P) = P$ is equivalent to $H^{n/2}_t(P) = H^{-n/2}_t(P)$ (resp. $H^{(n+1)/2}_t(P) = H^{-(n-1)/2}_t(P)$) for even $n$ (resp. odd $n$).  We divide the proof into two cases depending on whether $n$ is an even integer or an odd integer. We sketch a proof for the even case. (The odd case is shown similarly.) We set $n = 2m$. 

In the following, we treat $b$ as another variable which is independent of $t$ and regard $A_n(t)$ and $B_n(t)$ as polynomials in variables $b$ and $t$ for all integers $n\ge 0.$ 
By properties of resultants (cf. \cite[Chap.~3.~Sect.~1]{CLO}) and the recursive relations \eqref{eqn:def:An}, \eqref{eqn:def:Bn}, we have 
\begin{equation}
\label{eqn:CLO}
\Res_t(A_{2m}(t), A_{2m-1}(t)-b) = \pm 
\Res_t(A_{m}(t) - B_{m-1}(t), A_{m-1}(t)- B_m(t)).  
\end{equation}

We set $C_m(t) := A_{m}(t) - B_{m-1}(t)$ and $D_m(t) := A_{m-1}(t)- B_m(t)$. 
Then one can check that $\deg_t(C_m(t)) = \deg_t(D_m(t)) = 2^{m-1}$. We write 
\begin{align*}
C_m(t) & = \gamma_{m\, 0}(b) t^{2^{m-1}} + \gamma_{m\, 1}(b)  t^{2^{m-1}-1} + \cdots +  \gamma_{m\,
2^{m-1}}(b), \\
D_m(t) & := \delta_{m\, 0}(b) t^{2^{m-1}} + \delta_{m\, 1}(b)  t^{2^{m-1}-1} + \cdots +  \delta_{m\, 2^{m-1}}
(b)
\end{align*}
with $\gamma_{m\, i}(b), \delta_{m\, i}(b) \in \ZZ[b]$.

For $C_m(t)$, one can check that $\gamma_{m\, 0}(b) = 1$, 
and $\deg(\gamma_{m\, i}(b))  \le i$ for any $0 \leq i \leq 2^{m-1}$.
For $D_m(t)$, one can check $\delta_{m\, 0}(b) = 1$ and 
$\deg(\delta_{m\, i}(b))  \leq 2i$ 
for any $0 \leq i \leq 2^{m-1}$.
We also have
\[
\delta_{m\, 2^{m-1}}(b) = b^{2^{m}} + \text{(lower terms in $b$ with integer coefficients)}.
\]

Then, by repeating elementary transformations of the Sylvester matrix of the resultant 
$\Res_t(C_m(t), D_m(t))$, one can show that 
\begin{equation}
\label{eqn:resul:C:D}
  \Res_t(C_m(t), D_m(t))= b^{2^{2m-1}} + c_{1}^\prime b^{2^{2m-1}-1} + \cdots + c_{2^{2m-1}}^\prime
\end{equation}
with $c_i^\prime \in \ZZ$.

By \eqref{eqn:CLO} and \eqref{eqn:resul:C:D}, we obtain \eqref{eqn:prop:resultant:is:monic}. 
\QED

\begin{Example}
We treat the case of $n = 4$ (i.e. $m=2$) to illustrate 
how we repeat elementary transformations of the Sylvester matrix of the resultant 
$\Res_t(C_m(t), D_m(t))$ in the proof of Proposition~\ref{prop:resultant:is:monic:app}.  
Our goal is to verify that $\Res_t(A_4(t), A_{3}(t)-b)
= \pm b^{8} + c_{1} b^{7} + \cdots + c_{8}$ with $c_i \in \ZZ$ 
in a way that generalizes to any even $n$.

It suffices to show that $\Res_t(C_2(t), D_2(t))
= b^{8} + c_{1}^\prime b^{7} + \cdots + c_{8}^\prime$ with $c_i^\prime \in \ZZ$.
We have 
\begin{align}
\label{eqn:app:property:c2}
C_2(t) & = t^2 + \gamma_{2\, 1}(b) t + \gamma_{2\, 2}(b) \\
\notag
& \hspace{5eM} (\gamma_{2\, i}(b) \in \ZZ[b], \; \deg(\gamma_{2\, 1}(b)) \leq 1, \; \deg(\gamma_{2\, 2}(b)) \leq 2),
\\
\label{eqn:app:property:d2}
D_2(t) & = t^2 + \delta_{2\, 1}(b) t + \delta_{2\, 2}(b) \\
\notag
& \hspace{5eM} (\delta_{2\, i}(b) \in \ZZ[b], \; \deg(\delta_{2\, 1}(b)) \leq 2, \\
\notag
& \hspace{5eM} \; \delta_{2\, 2}(b) = b^4 + \text{(lower
terms in $b$ with integer coefficients)}).
\end{align}

To simplify the notation, we denote $\gamma_{2\, i}(b)$ and $\delta_{2\, i}(b)$ by
$\gamma_{i}$ and $\delta_{i}$ for $i = 1, 2$.
Then the resultant is given by the determinant of the Sylvester matrix 
{\allowdisplaybreaks
\begin{align*}
\Res_t(C_2(t), D_2(t))
&  =
\begin{vmatrix}
1 & \gamma_1 & \gamma_2 & 0 \\
0 & 1 & \gamma_1 & \gamma_2\\
1 & \delta_1 & \delta_2 & 0 \\
0 & 1 & \delta_1 & \delta_2\\
\end{vmatrix}
=
\begin{vmatrix}
1 & 0 & 0 & 0 \\
0 & 1 & \gamma_1 & \gamma_2\\
1 & \delta_1-\gamma_1  & \delta_2-\gamma_2 & 0 \\
0 & 1 & \delta_1 & \delta_2
\end{vmatrix}
\\
& =
\begin{vmatrix}
1 & 0 & 0 & 0 \\
0 & 1 & 0 & 0 \\
1 & \delta_1-\gamma_1  & - \gamma_1 (\delta_1 - \gamma_1) + (\delta_2-\gamma_2)&
- \gamma_2(\delta_1 - \gamma_1) \\
0 & 1 & \delta_1 -  \gamma_1 & \delta_2 - \gamma_2
\end{vmatrix}
\\
& =
\begin{vmatrix}
- \gamma_1 (\delta_1 - \gamma_1) + (\delta_2-\gamma_2)  & - \gamma_2(\delta_1 - \gamma_1) \\
\delta_1 -  \gamma_1 & \delta_2 - \gamma_2
\end{vmatrix},
\end{align*}
}%
where, for the second equality, the second (resp. third) column is subtracted by $\gamma_1$ (resp. $\gamma_2$)
times the first column, and, for the third equality, the third (resp. fourth) column is subtracted by $\gamma_1$ (resp.
$\gamma_2$) times the second column.

We write the right hand side of the above displayed equality as
\[
\begin{vmatrix}
f_{11} & f_{12} \\
f_{21} & f_{22}
\end{vmatrix},
\]
so that, for example, $f_{11} =  - \gamma_1 (\delta_1 - \gamma_1) + (\delta_2-\gamma_2)$. By
\eqref{eqn:app:property:c2} and \eqref{eqn:app:property:d2}, we have
\[
  f_{ii} = b^4 + \text{(lower terms in $b$ with integer coefficients)} \quad \text{for $i=1, 2$},
\]
and $\deg(f_{21}) \leq 2$ and $\deg(f_{12}) \leq 4$. Then
\[
  \Res_t(C_2(t), D_2(t)) = f_{11} f_{22} - f_{12} f_{21}
  = b^8 + \text{(lower terms in $b$ with integer coefficients)}, 
\]
as desired.
\end{Example}

\end{document}